\documentclass[a4paper,10pt]{article}
\usepackage[english]{babel}
\usepackage{amssymb}
\usepackage{amsmath}
\usepackage{indentfirst}
\usepackage[dvips]{graphicx}
\usepackage{epsfig}
\usepackage{lscape}
\usepackage{hyperref}
\usepackage{color}
\usepackage{amsthm}
\usepackage{comment}
\newtheorem{defi}{Definition}[section]
\newtheorem{cor}{Corollary}[section]

\newtheorem{lemma}[cor]{Lemma}
\newtheorem{teo}[cor]{Theorem}

\newtheorem{rem}[cor]{Remark}

\newtheorem{ass}[cor]{Assumption}

\newcommand{\R}{\mathbb{R}}
\newcommand{\N}{\mathbb{N}}

\newcommand{\abVal}[1]{|#1|}
\newcommand{\norm}[1]{\|#1\|}
\newcommand{\im}{\mathrm{im}}
\newcommand{\DZweiTauNu}[1]{D^2\varphi^{\circ}(\nu_{\ast}^{#1}(1))\tau_{\ast}^{#1}(1)\cdot\tau_{\ast}^{#1}(1)}
\newcommand{\varphiDZweiNuTauTau}[2]{\varphi^{\circ}(\nu^{#1}_{\ast})(D^2\varphi^{\circ}(\nu_{\ast}^{#1})\tau_{\ast}^{#1}\cdot\tau_{\ast}^{#2})}

\newcommand{\DZweiNuTauTau}[2]{(D^2\varphi^{\circ}(\nu_{\ast}^{#1})\tau_{\ast}^{#1}\cdot\tau_{\ast}^{#2})}
\newcommand{\Id}{\mathrm{Id}}

\usepackage{tikz}
\usepackage{pgfplots}
\usetikzlibrary{fillbetween, calc}
\usetikzlibrary{arrows.meta, calc, bending, positioning, intersections}
\pgfplotsset{compat=newest}

\numberwithin{equation}{section}

\title{Stability analysis for the anisotropic curve shortening flow of planar networks}

\author{
Michael G\"o\ss{}wein\thanks{Fakult\"at f\"ur Mathematik, Universit\"at Duisburg-Essen,
Thea-Leymann-Stra\ss e 9,
45127 Essen, Germany, \url{michael.goesswein@uni-due.de}},
Matteo Novaga\thanks{Dipartimento di Mathematica, Universit\`{a} di Pisa, Largo Bruno Pontecorvo 5, 56127 Pisa, Italy, \url{matteo.novaga@unipi.it}}~ and 
Paola Pozzi\thanks{Fakult\"at f\"ur Mathematik, Universit\"at Duisburg-Essen,
Thea-Leymann-Stra\ss e 9,
45127 Essen, Germany, \url{paola.pozzi@uni-due.de}}
}

\begin{document}
 \setlength{\parindent}{0em} 
\maketitle

\begin{abstract}
In this article we study the anisotropic curve shortening flow for a planar network of three curves with fixed endpoints and which meet in a triple junction. We show that the anisotropic curvature energy fulfills a \L ojasiewicz-Simon gradient inequality and use this knowledge to derive stability results for the flow. Precisely, in our main theorem we show that for any initial data, which are $C^{2+\alpha}$-close to a (local) energy minimizer, the flow exists globally and converges to a possibly different energy minimum.
\end{abstract}

\bigskip

 
 
\section{Introduction}
For a regular, immersed curve $\Gamma\subset\R^2$ the length of $\Gamma$ is given by
\begin{align}
	E(\Gamma)=\int_{\Gamma}1\,ds,
\end{align}
where $s$ is the arc length parameter. The corresponding $L^2$-gradient flow is the so called curve shortening flow, which is also known as mean curvature flow in the case of higher dimensional surfaces. Originally the flow was suggested by Mullins \cite{MullinsIdealizedGrainBoundaries} to model the evolution of grain boundaries of heated polycrystalls. Afterwards the flow received a lot of interest both by mathematicians and physicists. Now, in material sciences it is very natural that the surface energy also depends on anisotropic effects. Then the energy above transforms to
\begin{align}\label{EquationIntroductionAnisoEnergy}
	E(\Gamma)=\int_{\Gamma}\varphi^{\circ}(\nu)ds,
\end{align}
where $\varphi^{\circ}: \R^2\to [0,\infty)$ is an anisotropy and $\nu$ the unit normal of $\Gamma$. The corresponding $L^2$-gradient flow induces an anisotropic curve shortening  
flow and in the last thirty years a lot of research on it was done, see for example \cite{TaylorMotionCurvesCrystallineCurvature, GageEvolvingPlaneCurvesByCurvatureInRelativeGeometries, BellettiniPaoliniAnisotropicMotionMC, ZhuAsymptoticBehaviorAnisoCurveFlows, CasellesChambolleAnisoCurvatureDrivenFlowConvexSets, PozziAnisoCurveShortFlowHigherCodim, PozziOnGradientFlowAnisoAreaFunctional, MercierNovagaPozziAnisotropicCurvatureFlowImmersedCurves, DeckelnickNuernbergUnconditonallyStableFiniteElementScheme} and references therein. Note that, from a mathematical point of view, the anisotropic curve shortening flow is also  a natural generalization of the curve shortening flow when considering Finsler spaces, see e.g. \cite{BellettiniPaoliniAnisotropicMotionMC}.

In this article we consider this energy on a network of three regular, immersed curves $\Gamma^1, \Gamma^2, \Gamma^3$ with fixed distinct endpoints $P^1, P^2, P^3$, and  meeting in a common triple junction $\Sigma$. This geometry is pictured in Figure \ref{FigureGeometricSituation}. 
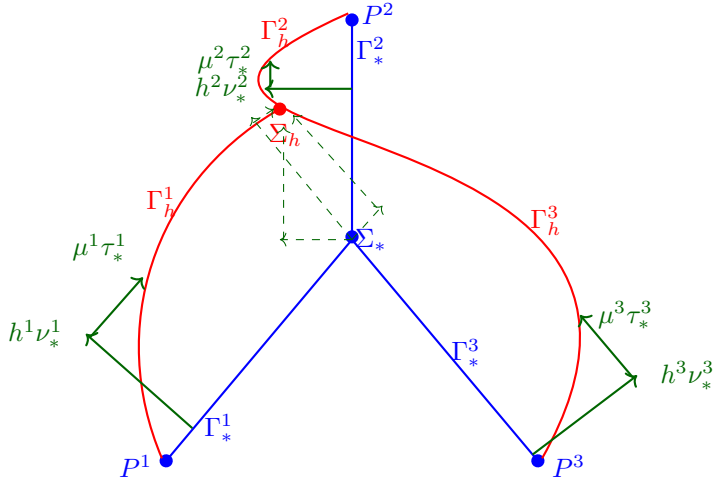
\begin{figure}\label{FigureGeometricSituation}
	\centering
	\begin{tikzpicture}[scale=0.5]
		\node[color=blue] at (0.5, 0.0) {$\Sigma_\ast$};
		\draw[color=blue, Circle-Circle, thick] (0, 6) node [right] {$P^2$} -- (0, -0.1);
		\node[color=blue] at (0.5, 5) {$\Gamma^2_\ast$};
		\draw[color=blue, Circle-, thick] (-5, -6) node [left] {$P^1$} -- (0, 0);
		\node[color=blue] at (3, -3) {$\Gamma^3_\ast$};
		\draw[color=blue, Circle-, thick] (5, -6) node [right] {$P^3$} -- (0, 0);
		\node[color=blue] at (-3.5, -5) {$\Gamma^1_\ast$};
		
		\draw[color=red, thick] (5, -5.8) .. controls (10.5, 4) and (-9, 2) .. (-0.1, 6);		
		\node[color=red] at (5.1, 0.5) {$\Gamma_h^3$};
		\node[color=red] at (-2, 5.5) {$\Gamma_h^2$};
		\draw[color=red, thick, -Circle] (-5, -5.8) to [bend left=42] (-1.75, 3.55) node [below=3pt] {$\Sigma_h$};
		\node[color=red] at (-5, 1) {$\Gamma_h^1$};

		\draw[color=green!40!black, thick, ->] (-4.2, -5) -- (-7, -2.5) node [left=5pt] {$h^1 \nu^1_\ast$};
		\draw[color=green!40!black, thick, ->] (-6.9, -2.6) -- (-5.5, -1) node [above left=3pt] {$\mu^1 \tau^1_\ast$};
		\draw[color=green!40!black, thick, ->] (4.75, -5.7) -- (7.5, -3.6) node [right=5pt] {$h^3 \nu^3_\ast$};
		\draw[color=green!40!black, thick, ->] (7.4, -3.67) -- (6, -2) node [right=3pt] {$\mu^3 \tau^3_\ast$};
		\draw[color=green!40!black, thick, ->] (0, 4) -- (-2.3, 4) node [left=2pt] {$h^2 \nu^2_\ast$};
		\draw[color=green!40!black, thick, ->] (-2.15, 4.04) -- (-2.15, 4.75) node [left=3pt] {$\mu^2 \tau^2_\ast$};
		

		\draw[color=green!40!black, opacity=.5, dashed, ->] (0, 0) -- (-1.9, 0) node [below=2pt] {}; 
		\draw[color=green!40!black, opacity=.5, dashed, ->] (-1.8, 0) -- (-1.8, 3); 
		\draw[color=green!40!black, opacity=.5, dashed, ->] (0, 0) -- (0.8, 0.9) node [right] {}; 
		\draw[color=green!40!black, opacity=.5, dashed, ->] (0.65, 0.85) -- (-1.55, 3.3); 
		\draw[color=green!40!black, opacity=.5, dashed, ->] (0, 0) -- (-2.65, 3.25) node [left] {}; 
		\draw[color=green!40!black, opacity=.5, dashed, ->] (-2.6, 3.1) -- (-2.05, 3.47); 
	\end{tikzpicture} 
	\caption{This graphic illustrates the geometric situation we consider in this article. Both the blue and the red lines have the form of the considered network consisting of three immersed curves. Hereby, the blue configuration illustrates the shape of an energy minimum of the anisotropic length energy and the red configuration the moving geometry itself. We will track the evolution of the latter by writing it as a kind of graph over the minimum configuration. This leads then to the green arrows. We will explain this in details in Section \ref{SubsectionParametrization} and the whole notation in Section \ref{SectionNotation}. }
\end{figure}

The study of the geometric evolution of networks received a lot of attention in the last years as it can be used as a model for many applications (see for instance \cite{MaNoPlSc16}, \cite{Kagaya} and references given in there). The network we consider is one of the typical minimal configurations, which appears often as subnetwork in more complicated situations. We denote the whole network by $\Gamma$ and compute its anisotropic length by

\begin{align}\label{EquationAnisoEnergyOnNetworks}
	E(\Gamma)=\sum_{i=1}^3\int_{\Gamma^i}\varphi^{\circ}(\nu^{i})ds^i,
\end{align}

where $s^i$ denotes the arc length parameter of the $i$-th curve. Now, let $\gamma^i, i=1,2,3,$ be regular parametrizations of $\Gamma^i$ such that $\gamma^i(0)=P^i, \gamma^i(1)=\Sigma, i=1,2,3$. Furthermore, let $\zeta^i: [0,1]\to\R^2, i=1,2,3,$ be smooth functions with $\zeta^i(0)=0, i=1,2,3,$ and $\zeta^i(1)=\zeta^j(1), i,j\in\{1,2,3\}$. For variations of type $\Gamma^i_{\varepsilon}=\im(\gamma^i+\varepsilon\zeta^i), i=1,2,3,$ we obtain that

\begin{align}\label{EquationVariationAnisotropicLength}
\frac{d}{d\varepsilon}\Big|_{\varepsilon=0} E(\Gamma_{\varepsilon})=-\sum_{i=1}^3\int_0^1(D^2\varphi^{\circ}(\nu^i)\tau^i\cdot\tau^i)\kappa^i\nu^i\cdot\zeta^i ds^i-\sum_{i=1}^3D\varphi^{\circ}(\nu^i(1))\cdot (\zeta^i(1))^{\bot},
\end{align}

where $\bot$ denotes the anticlockwise rotation by 90 degree. From \eqref{EquationVariationAnisotropicLength} one can derive a $L^2$-gradient flow with natural boundary conditions, which we will give in detail in Section \ref{SectionGeometricFlow} and which we will study more in detail in this work. Note that typically a factor simulating the mobility of the curves 
is included in the system, see e.g. \cite{AngenentGurtinMultiphaseThermomechanicsII}.

This article is a continuation of the work presented in \cite{KroenerNovagaPozziAnisoCurvFloImmeredNetworks}, 
 where both short time existence for motion by anisotropic curvature of the network $\Gamma$ and the behavior at the maximal existence time are studied: in particular there the authors show that if the maximal time of existence of the flow is finite, then either the length of one of the curves goes to zero or the $L^{2}$-norm of the anisotropic curvature blows up. In this paper we show stability of local energy minima of \eqref{EquationAnisoEnergyOnNetworks}. Precisely, we show that the flow of any network, that is $C^{2+\alpha}$-close to a local energy minimum, exists globally in time and converges to a possible different local energy minimum with the same energy. Our main result is given in detail in Theorem \ref{TheoremStabilityAnMCF}. 
In Remark~\ref{rem5.3} we point out that due to our elementary network setting and exploiting the assumptions on the anisotropy map $\varphi^{\circ}$,  the energy landscape for our functional is quite simple, in the sense that if a local minimum for $E$ is non-degenerate (i.e. it has a triple junction point that is distinct from $P_{i}$, $i=1,2,3$)  then it is in fact the unique minimum for $E$,  to which  the flow converges  for  sufficiently close initial data.
 Analogous versions of our main result restricted to  the isotropic setting are presented in \cite{KL2001} and \cite{PludaPozzetaLojasiewiczsimon}.

Our analysis relies mainly on the application of the so called \L ojasiewicz-Simon gradient inequality. Chill \cite{Chill2003OnTheLojasiewiczSimon} established a very general framework to prove such an inequality. To make the application easier we rely on a simplified framework from \cite{FeehannMaridakisLojasiewiczSimon}. This technique has been applied successfully to many other geometric flows, e.g. 
\cite{ChillFasangovaSchaetzleWillmoreBlowupsNeverCompact, AcquaPozziSpenerLSIOpenElasticCurve, GarckeGoessweinNonlinearStabilityDoubleBubbleSDF, RuppSpenerExistenceElasticFlowClampedCurves, PludaPozzetaLojasiewiczsimon}.  
Indeed, our result and strategy is very similar to \cite{PludaPozzetaLojasiewiczsimon}. In \cite{PludaPozzetaLojasiewiczsimon} the authors restrict to the (isotropic) curve shortening  flow while our article generalizes to the anisotropic  setting. In particular the verification of the prerequisites of the \L ojasiewicz-Simon gradient inequality is more technical due to the anisotropic angle conditions. On the other hand \cite{PludaPozzetaLojasiewiczsimon} study general network situations. Concerning the proofs we want to note that we use a different and elegant approach in the construction of a suitable graph parametrizations of the evolving geometry. The proof of Lemma \ref{LemmaExistenceRefFramGraphPara} works directly on the function spaces while the construction in \cite[Proposition 3.4]{PludaPozzetaLojasiewiczsimon} is done pointwise. 

The paper is organized as follows. In Section \ref{SectionPrelim} we will first clarify the notation, then give some facts about anisotropies and finally sum up the needed results from \cite{KroenerNovagaPozziAnisoCurvFloImmeredNetworks}. In Section \ref{SectionParaAndLoja} we will first introduce a way to track the evolution of our network as graph over a reference configuration and show that it indeed is possible to parametrize all networks that are close enough in a certain sense (Section \ref{SubsectionParametrization}). Then we will calculate the first and second variation of the anisotropic length energy (Section \ref{SectionVariations}) and use these results to prove a \L ojasiewicz-Simon gradient inequality (Section \ref{SectionLoja}).
In Section \ref{SectionAnalysisH} we verify an existence and smoothing result if our flow is written in the graph form from Section \ref{SectionParaAndLoja}.
 Finally, we will carry out the stability analysis and prove our main result Theorem~\ref{TheoremStabilityAnMCF} in Section~\ref{SectionStabAna}.

\section{Preliminaries}\label{SectionPrelim}
\subsection{Notation}\label{SectionNotation}
We will consider three curves $\Gamma^1(t),\Gamma^2(t)$ and $\Gamma^3(t)$ in $\R^2$ meeting in a common triple junction $\Sigma(t)$ and having fixed outer boundary points. The time evolving parametrizations of these curves over the interval $[0,1]$ will be denoted by $\gamma^i$. We will use $x$ as space and $t$ as time parameter. The parametrizations are chosen such that $\gamma^i(t,1)=\Sigma(t)$ for $i=1,2,3, t\ge 0$. In general, the upper index will always refer to the specific curve we are on. We will omit the upper index, if it is either clear on which curve we are or to refer to the whole geometry. Furthermore, we use the typical geometric quantities,  such as  the unit tangential vectors $\tau^i$, the unit normal vectors $\nu^i$, the curvature $\kappa^i$ and the arc length parameter $s^i$. Hereby, the unit normal vectors arise from the unit tangential vectors by anticlockwise $90^\circ$-rotation. Such rotation will be denoted by upper index $\bot$. For $(i,j,k)=(1,2,3), (2,3,1), (3,1,2)$ we will denote by $\theta^k(t)$ the angle between $\nu^i(t,1)$ and $\nu^j(t,1)$ at the triple junction. 
Note that these angles might be not constant in time. In case $\theta^k(t) \in (0, \pi)$, $k=1,2,3$, they are related by $\theta^{1}(t)+ \theta^{2}(t)+ \theta^{3}(t) =2\pi$ and suitable $\tilde{\alpha}^i(t)\in\R^+, i=1,2,3,$ fulfilling,
\begin{align}\label{EquationYoungsModulus}
	\frac{\sin\theta^1(t,1)}{\tilde{\alpha}^1(t)}=\frac{\sin\theta^2(t,1)}{\tilde{\alpha}^2(t)}=\frac{\sin\theta^3(t,1)}{\tilde{\alpha}^3(t)},
\end{align}
which can be shown to be equivalent to
\begin{align}\label{EquationForceBalanceTripleJunction}
	\sum_{i=1}^3\tilde{\alpha}^i(t)\nu^i(t,1)=0.
\end{align}
A proof of the equivalence of \eqref{EquationYoungsModulus} and \eqref{EquationForceBalanceTripleJunction} can be found in Lemma \ref{LemmaEquivalenceForceBalanceYoungModulus}.
 Note that \eqref{EquationYoungsModulus}
 typically appears from Young's modulus, which is a force balance at the triple junction, see, e.g., \cite{BronsardReitichOnThreePhaseBoundaryMotion} for a mathematical motivation and \cite{SMITHCyrilStanleyGrainsPhasesInterfacesInterpretationofMicrostructures} for a physical modelling. There, the $\tilde{\alpha}^i$ are related to the different energy densities of the surfaces. Furthermore, note that our assumption on the $\theta^i$ implies that all $\tilde{\alpha}^i$ have the same sign. So we may indeed choose them all positive.

 For the definition of the anisotropic flow we will denote the Wulff-shape inducing anisotropy by $\varphi$ and the corresponding polar norm by $\varphi^{\circ}$. Then $E(\Gamma)$ is the corresponding anisotropic curvature energy given by
\begin{align*}
	E(\Gamma)=\sum_{i=1}^3 \int_0^1 \varphi^{\circ}(\nu^i)ds^i.
\end{align*} 
In the main part of the paper we will track the evolution of the considered curves as graph over a given energy minimum $\Gamma_{\ast}$. Hereby, $h$ will be referring to the normal and $\mu$ to the tangential part of this parametrization. Additionally, lower index $h$ will refer to the evaluation of a specific quantity for the resulting curve due to this choice of $h$. For example, $\Gamma_h^i$ will be the $i$-th curve parameterized using $h$ as normal part and $\kappa_h^i$ will refer to the curvature on $\Gamma_h^i$. With some abuse of notation $E(h)$ will denote $E(\Gamma_h)$. Finally, a lower index $\ast$ will refer to a quantity in the reference geometry. For example, $\kappa_{\ast}$ will denote the curvature operator of $\Gamma_{\ast}$, $ds_{\ast}$ the length element on $\Gamma_{\ast}$, and so on.

\subsection{The geometric flow and known results}\label{SectionGeometricFlow}
In this section we will give a precise definition of the anisotropic curve shortening flow we consider. Additionally, for the reader's convenience, we will state known results from \cite{KroenerNovagaPozziAnisoCurvFloImmeredNetworks} concerning anisotropies and our flow. Note that we are being rather concise here. A more detailed introduction can be found in \cite[Section 2]{KroenerNovagaPozziAnisoCurvFloImmeredNetworks}.

We will begin with some basic definitions and properties of anisotropies, see for example \cite{BellettiniPaoliniAnisotropicMotionMC}.
\begin{defi}[Anisotropy]\
	\begin{itemize}
		\item[i.)] An anisotropy is a norm $\varphi:\R^2\to [0,\infty)$. We say that $\varphi$ is smooth if $\varphi\in C^{\infty}(\R^2\backslash\{0\})$ and $\varphi$ is elliptic if $\varphi^2$ is uniformly convex, i.e., there exists $C>0$ such that 
		\begin{align}\label{EquationEllipticAniso}
		D^2(\varphi^2)\ge C\, \Id	
		\end{align}
	in the distributional sense.
		\item[ii.)] The set $W_{\varphi}:=\{\varphi\le 0\}$ is called Wulff shape. We say that $\varphi$ is crystalline if $W_{\varphi}$ is a polygon.
		\item[iii.)] For an anisotropy $\varphi$ we introduce the polar norm $\varphi^{\circ}$ relative to $\varphi$ by
		\begin{align}\label{EquationPolarNorm}
			\varphi^{\circ}(x)=\sup\{\zeta\cdot x|\varphi(\zeta)\le 1\}.
		\end{align}
	\end{itemize}
\end{defi}
Note that $\varphi$ is smooth and elliptic if and only if $\varphi^{\circ}$ is smooth and elliptic.
The following formulas will be essential in a lot of computations we do in this article.
\begin{lemma}[Properties of anisotropies]\ \label{lemma2.1} \\
	Let $\varphi$ be a sufficiently smooth  elliptic anisotropy with ellipticity constant $C$. Furthermore, let $\nu,\tau\in \R^2$ be unit vectors with $\nu\cdot\tau=0$ and $p\in\R^2$ with $p\neq 0$. Then we have the following:
	\begin{itemize}
		\item[i.)] $D^2\varphi(\nu)\tau\cdot\tau\ge \tilde{C}$ with $\tilde{C}:=C(2\max\{\varphi(\tilde{\nu})|\tilde{\nu}\in S^1\})^{-1}$.
		\item[ii.)] $D\varphi(p)\cdot p = \varphi(p)$.
		\item[iii.)] $D^2\varphi(p)p=0$.
	\end{itemize}
\end{lemma}
\begin{proof}
	The first fact is proven in \cite[Remark 1]{MercierNovagaPozziAnisotropicCurvatureFlowImmersedCurves}. The other two results follow directly from the homogeneity property of a norm.
\end{proof}
We will need the following quantity for the definition of our flow.
\begin{defi}[Anisotropic curvature]\ \\
	Let $\varphi$ be a sufficiently smooth, elliptic anisotropy and $\Gamma$ a curve in $\R^2$ with the usual notation of the geometric quantities. Additionally, denote by 
	\begin{align}\label{EquationCahnHoffmannVector}
		N:=D\varphi^{\circ}(\nu)
	\end{align}
the Cahn-Hoffmann vector. Then we define the (scalar) anisotropic curvature on $\Gamma$ by
	\begin{align}
		\kappa_{\varphi}:=-N_s\cdot\tau=D^2\varphi^{\circ}(\nu)\tau\cdot\tau\kappa.
	\end{align}
\end{defi}
Now we are able to define the anisotropic curve shortening flow for a triple junction network $\Gamma(t)$. As our solutions will be of class $C^{2+\alpha}$ for some H\"older coefficient $\alpha\in(0,1)$, the initial network $\Gamma_0$ needs to fulfil some compatibility conditions to guarantee the existence of a solution of such regularity. We sum up these compatibility conditions in the following definition before we actually define solutions to our flow. For the reader's convenience we include some basic facts about parabolic H\"older spaces in Appendix \ref{AppendixHoelder}.
\begin{defi}[Geometrically admissible networks]\ \label{def-gan} \\
	Let $P^i\in \R^2, i=1,2,3,$ and $\varphi^\circ$ an elliptic, smooth anisotropy. A network $\Gamma_0$ is called (geometrically) admissible if there exists regular parametrizations $u^i_0\in C^{2+\alpha}([0,1],\R^2), i=1,2,3$ such that $\Gamma^i_0=\im(u^i_0), i=1,2,3,$ and there holds
	\begin{align}\label{EquationCompatibilityConditions}
		\begin{cases}
			u^i_0(0)=P^i & i=1,2,3,\\
			u^1_0(1)=u^2_0(1)=u^3_0(1), & \\
			\sum_{i=1}^3D\varphi^{\circ}(\nu_0^i(1))=0,\\
			\kappa^i_{\varphi}(0)=0 & i=1,2,3,\\
			\kappa^i_{\varphi}(1)\varphi^{\circ}(\nu^i_0(1))\nu^i_0(1)+\lambda^i_0(1)\tau^i_0(1)=\kappa^j_{\varphi}(1)\varphi^{\circ}(\nu^j_0(1))\nu^j_0(1)+\lambda^j_0(1)\tau^j_0(1) & i,j\in \{1,2,3\}
.		\end{cases}
	\end{align}
Hereby, the $\lambda^i_0$ are geometrical, curvature dependent quantities on $\Gamma^i_0$ , which are given in detail in \cite[Section 2.0.4]{KroenerNovagaPozziAnisoCurvFloImmeredNetworks}.
\end{defi}
\begin{defi}[Geometric solution and special flow]\  \label{def-geomflow}\\
	Let $T>0, P^i\in \R^2, i=1,2,3,$ and $\varphi^\circ$ an elliptic, smooth anisotropy. Furthermore, let $\Gamma_0$ be an admissible network with respect to these data and the corresponding parametrizations of $\Gamma_0$ be given by $u_0^i, i=1,2,3$. Then we call $(u^i)_{i=1,2,3}$ with $u^i\in C^{\frac{2+\alpha}{2},2+\alpha}([0,T)\times[0,1],\R^2)$ a geometric solution of the anisotropic mean curvature flow, if
	\begin{align}\label{EquationGeometricProblem}
		\begin{cases}
			(u_t^i\cdot \nu^i)\nu^i=\varphi^{\circ}(\nu^i)(D^2\varphi^{\circ}(\nu^i)\tau^i\cdot\tau^i)\kappa^i\nu^i & \text{on }(0,T)\times(0,1), i=1,2,3,\\
			u^i(t,0)=P^i & \forall t\in(0,T), i=1,2,3,\\
			u^1(t,1)=u^2(t,1)=u^3(t,1) &\forall t\in(0,T),\\
			\sum_{i=1}^3D\varphi^{\circ}(\nu^i(t,1))=0 &\forall t\in(0,T),\\
			u^{i}(0,x)=u_0^{i}(x)  \text{ up to reparametrization} & \forall x\in [0,1], i=1,2,3.
		\end{cases}
	\end{align}
	Up to reparametrization means that $u^{i}(0,x)=u_{0}^{i}(\phi^{i}(x))$ for some orientation preserving diffeomorphism $\phi^{i} \in C^{2+\alpha} ([0,1], [0,1])$.
    A geometric solution, which additionally fulfills, that
   
    \begin{align}\label{EquationSpecialFLowCondition}
    	(u_t^i\cdot \tau^i)\tau^i&=\varphi^{\circ}(\nu^i)(D^2\varphi^{\circ}(\nu^i)\tau^i\cdot\tau^i)\frac{u_{xx}^{i}}{|u_x^{i}|^2}\cdot\tau^i  \tau^{i}& &\text{on }(0,T)\times(0,1), i=1,2,3,
    \end{align}
    
    is called a special flow.
\end{defi}
\begin{rem}[Motivation of the solution and the boundary conditions]\label{RemarkMotivationSolution}\
	\begin{itemize}
		\item[i.)]  The boundary conditions $\eqref{EquationGeometricProblem}_2-\eqref{EquationGeometricProblem}_4$ are natural for the geometric gradient flow. The first two just fix the motion of the boundary points of the curves. The third one follows by considering the variation of $E$, c.f. \eqref{EquationVariationAnisotropicLength}, if one wants to establish a natural $L^2$-gradient flow structure.
		\item[ii.)]Note that the fact that $\eqref{EquationGeometricProblem}_1$ only fixes the motion in normal direction is normal for a PDE system describing a geometric evolution. Tangential movements will not change the geometry and thus have no influence on the geometric evolution law. But as we track the geometric evolution using diffeomorphisms, we are not considering the geometric problem itself anymore but the motion of particles of the curves. Such  equations necessarily needs additional information for the tangential motion not to   be  degenerated. To fix this we have to make a choice for the tangential part. One - but not the only - possible choice is \eqref{EquationSpecialFLowCondition}, which gives us a well-defined system.\\
		To see this from the point of view of partial differential equations note that the special flow gives us that
	    \begin{align*}
	    	u^i_t=\varphi^{\circ}(\nu^i)(D^2\varphi^{\circ}(\nu^i)\tau^i\cdot\tau^i)\frac{u^i_{xx}}{|u_x|^2}.
    \end{align*}
Due to Lemma \ref{lemma2.1}i.) the factor before $u^i_{xx}$ has a positive sign and therefore the whole equation has the structure of quasilinear, parabolic equation. In the isotropic case this is very clear to see, as the special flow reduces to $u_{t}=\frac{u_{xx}}{|u_x|^2}$. 
		\item[iii.)] Observe that the positive factor $\varphi^{\circ}(\nu^i)$ in $\eqref{EquationGeometricProblem}_1$ does not come from the gradient flow structure with respect to the anisotropic surface energy $E$. Still such a factor is natural to be included as it models the mobility, see e.g. \cite{AngenentGurtinMultiphaseThermomechanicsII}. 
		\item[iv.)] Note that up to reparametrization all geometric solutions coincide with the special flow solution, cf. \cite[Lemma 4.1]{KroenerNovagaPozziAnisoCurvFloImmeredNetworks}.
		\item[v.)] The compatibility conditions \eqref{EquationCompatibilityConditions} are a direct consequence from the boundary conditions $\eqref{EquationGeometricProblem}_2-\eqref{EquationGeometricProblem}_4$ for solutions of class $C^{\frac{2+\alpha}{2},2+\alpha}$. The first three are just the boundary conditions postulated for the initial data. The others appear due to the fact that we can differentiate $\eqref{EquationGeometricProblem}_2$ and $\eqref{EquationGeometricProblem}_3$ in time. As the precise formula of the $\lambda^i$ will be not important for our analysis, we do not provide the formula in this work.
	\end{itemize}
\end{rem}

In \cite{KroenerNovagaPozziAnisoCurvFloImmeredNetworks} the following short time existence result was proven.
\begin{teo}[Short time existence for  the geometric flow]\label{TheoremSTE}\ \\
	Let $P^i\in\R^2, i=1,2,3, \alpha\in(0,1)$ and $\varphi$ (resp. $\varphi^{\circ}$) be a smooth, elliptic anisotropy. Furthermore, let $u^i_0\in C^{2+\alpha}([0,1],\R^2), i=1,2,3$ be regular maps fulfilling \eqref{EquationCompatibilityConditions}. Then there exists a $T_{STE}>0$ and regular maps $u^i\in C^{\frac{2+\alpha}{2},2+\alpha}([0,T_{STE}]\times[0,1],\R^2)  \, i=1,2,3,$ such that \eqref{EquationGeometricProblem} are fulfilled.
	 Additionally, we have that $u^i\in C^{\infty}((0,T_{STE}]\times[0,1],\R^2), i=1,2,3$.
\end{teo}
\begin{proof}
The proof is given in Theorem~\cite[Theorem 4.1]{KroenerNovagaPozziAnisoCurvFloImmeredNetworks} and hinges on the short time existence of the special flow (recall \eqref{EquationSpecialFLowCondition}  and see \cite[Theorem 3.1]{KroenerNovagaPozziAnisoCurvFloImmeredNetworks} and \cite[Corollary 3.1]{KroenerNovagaPozziAnisoCurvFloImmeredNetworks}).
\end{proof}

\section{Parametrization and \L ojasiewicz-Simon inequality}\label{SectionParaAndLoja}
\subsection{Parametrization as graphs over reference frames}\label{SubsectionParametrization}
In order to be able to prove a \L ojasiewicz-Simon inequality (short LSI) we have to eliminate the typical tangential degeneracy of the flow. For this we use an idea introduced in \cite{depner2013linearized}. Precisely, we want to track the evolution of $\Gamma(t)$ as graph over a reference frame $\Gamma_{\ast}$, 
	for which we will fix some assumptions for the rest of this article. 
\begin{ass}[Assumptions for $\Gamma_{\ast}$]\label{RemarkAssumstionsGammaast}\ \\
	In the rest of this article $\Gamma_{\ast}$ will be a fixed local minimum of $E$. Hereby, we say that $\Gamma_{\ast}$ - parametrized by some $\gamma_{\ast}$ - is a local minimum, if there is a $\delta>0$ such that for all triple junction networks $\tilde{\Gamma}$ with the same fixed endpoints parametrized by some $\tilde{\gamma}$ with $\|\gamma_{\ast}-\tilde{\gamma}\|_{H^2}\le \delta$ we have that $E(\Gamma_{\ast})\le E(\tilde{\Gamma})$. In particular, $\Gamma_{\ast}$ consists of three curves $\Gamma^i_{\ast}$ with zero curvature, i.e., three straight lines. The three lines start in the points $P^{i}$ and then meet in a common triple junction where the normal vectors $\nu^i_{\ast}(1)$ fulfill  $\eqref{EquationGeometricProblem}_4$. As the normal and tangent vector are constant on each $\Gamma^i_{\ast}$, we will omit the space variable for them.\\
	 Additionally, we require the contact angles $\theta^i_{\ast}$ to fulfill
	\begin{align}\label{EquationAngleRestriction}
		\forall i=1,2,3: \theta^i_{\ast}\in(0,\pi),\quad \theta_{\ast}^1+\theta_{\ast}^2+\theta_{\ast}^3=2\pi.
	\end{align}
     Then, by \eqref{EquationYoungsModulus} we can find positive $\tilde{\alpha}^i_{\ast}\in\R^+, i=1,2,3,$ such that
     \begin{align}\label{EquationDefinitionOfAlphaTildeAst}
     	\sum_{i=1}^{3} \tilde{\alpha}^{i}_{*}\nu^{i}_{*}=0.
     \end{align}
     We fix a triplet $(\tilde{\alpha}_{\ast}^1, \tilde{\alpha}_{\ast}^2,\tilde{\alpha}_{\ast}^3)$. Finally, we fix for $i=1,2,3$ any regular parametrizations $\gamma_{\ast}^i$ of $\Gamma_{\ast}^i$ such that 
     \begin{align}\label{EquationConditionParaOfGammaAst}
     	(\gamma_{\ast}^i)'(1)=\tau_{\ast}^i.
     \end{align}
\end{ass}

\begin{rem}Note that by the assumptions on the anisotropy map equation \eqref{EquationAngleRestriction} is automatically realized. To see this first of all observe that the assumption that $\varphi^{\circ}$  is an elliptic smooth anisotropy yields that the unit balls $\partial B_{\varphi^{\circ}}$ and $\partial B_{\varphi}$ are strictly convex, i.e. they contain no straight segments (cf. for instance \cite[\S.~2]{PozziOnGradientFlowAnisoAreaFunctional}).
Now let $\nu^{i} \in \R^{2}$ be  Euclidean unit vectors $|\nu^{i}|=1$, such that  $\eqref{EquationGeometricProblem}_4$ is fulfilled, i.e.
$$ \sum_{i=1}^{3}D \varphi^{\circ}(\nu^{i}) =0$$
 and set $\xi^{i}=\frac{\nu^{i}}{\varphi^{\circ}(\nu^{i})}$. Then $\varphi^{\circ}(\xi^{i})=1$ for $i=1,2,3$. Moreover $D \varphi^{\circ}(\nu^{i})= D \varphi^{\circ}(\xi^{i})$ is normal to $\partial B_{\varphi^{\circ}}$ in the point $\xi^{i} \in \partial B_{\varphi^{\circ}}$.
Furthermore it holds $\varphi(D \varphi^{\circ}(\nu^{i}))=1$.
Due to the strict convexity of the balls the correspondence
$$  \partial B_{\varphi^{\circ}} \ni \xi^{i}  \to  D \varphi^{\circ}(\nu^{i})\in \partial B_{\varphi}$$
is bijective.\\
If the angle $\theta^{3}$ between the vectors $\nu^{1}$ and $\nu^{2}$ (or equivalently between $\xi^{1}$ and $\xi^{2}$) were equal to $\pi$, then, using the symmetry of the unit ball, we infer $\xi^{1}=-\xi^{2}$, and therefore also $D \varphi^{\circ}(\xi^{1})=-D \varphi^{\circ}(\xi^{2})$, that is $D \varphi^{\circ}(\nu^{1})=-D \varphi^{\circ}(\nu^{2})$.
But then we obtain a contradiction since
$$ 0= \sum_{i=1}^{3}D \varphi^{\circ}(\nu^{i}) = D \varphi^{\circ}(\nu^{3}) \neq 0.$$
If the angle $\theta^{3}$ between the vectors $\nu^{1}$ and $\nu^{2}$ (or equivalently between $\xi^{1}$ and $\xi^{2}$) were equal to $0$, then $\xi_{1}=\xi_{2}$ and hence $D \varphi^{\circ}(\nu^{1})=D \varphi^{\circ}(\nu^{2})$.
From $ 0= \sum_{i=1}^{3}D \varphi^{\circ}(\nu^{i})$ it follows that $2 D\varphi^{\circ}(\nu^{1})= - D\varphi^{\circ}(\nu^{3})$
 yielding again a contradiction since then
 $$2 = \varphi (2 D\varphi^{\circ}(\nu^{1})) = \varphi (- D\varphi^{\circ}(\nu^{3}))=1.$$
Finally, if the angle $\theta^{3}$ between the vectors $\nu^{1}$ and $\nu^{2}$ (or equivalently between $\xi^{1}$ and $\xi^{2}$)  lied in $(\pi, 2\pi)$, then using the strict convexity of $\partial B_{\varphi^{\circ}}$ we infer that the vectors $D\varphi^{\circ}(\nu^{i})$ (which are normal to
$\partial B_{\varphi^{\circ}}$ in $\xi^{i}$) point in the same half plane contradicting the fact that
$ 0= \sum_{i=1}^{3}D \varphi^{\circ}(\nu^{i})$.
\end{rem}

Note that this section and Section \ref{SectionAnalysisH} do not require $\Gamma_{\ast}$ to be a local minimum. But the proof of our main result Theorem \ref{TheoremStabilityAnMCF} uses the property of local decrease of the energy.

 Now, we can define curves $\Gamma^i_{h,\mu}$ by 
\begin{align}\label{EquationNormalGraphs}
	\Gamma^i_{h,\mu}:=\im(\gamma^i_{h,\mu}: [0,1]\to\R^2),\quad \gamma^i_{h,\mu}(x):&=\gamma^i_{\ast}(x)+h^i(x)\nu^i_{\ast}(x)+\mu^i(x)\tau_{\ast}^i(x)\\
	&=\gamma^i_{\ast}(x)+h^i(x)\nu^i_{\ast}+\mu^i(x)\tau_{\ast}^i.\notag
\end{align}
    
Hereby, $h$ and $\mu$ are scalar functions tracking the displacement 
in normal and tangential direction. Clearly, if we want to guarantee the preservation of the triple junction at $x=1$, we will need some boundary conditions at $x=1$. With this in mind we give 
the following result.
\begin{lemma}[Relation between normal and tangential part]\label{LemmaRelationNormalTangentialPart}\ \\
Let $\Gamma_{\ast}, \tilde{\alpha}^i_{\ast}$ be as in Assumption~\ref{RemarkAssumstionsGammaast}. Consider for $i=1,2,3$ regular curve parametrizations $\Phi^i:[0,1]\to\R^2$. Furthermore let $N^i,T^i\in \R$ for $i=1,2,3$ be such that
\begin{align}\label{EquationHelpEquationNormalGraphs}
	\Phi^i(1)-\Sigma_{\ast}=N^i\nu_{\ast}^i+T^i\tau_{\ast}^i.
\end{align}
Hereby, the unique existence of such $N^i,T^i$ is clear due to the fact that for all $i=1,2,3$ the pair $(\nu_{\ast}^i,\tau_{\ast}^i)$ forms a basis of $\R^2$. Then the triple junction condition
\begin{align}
	\Phi^1(1)=\Phi^2(1)=\Phi^3(1)
\end{align}
is equivalent to the conditions
\begin{align}\label{EqRelationNormalPart}
	0&=\sum_{i=1}^3 \tilde{\alpha}_{\ast}^iN^i, 
	\\ \label{EqRelationNormalTangentialPart}
	\begin{pmatrix}T^1 \\ T^2 \\ T^3\end{pmatrix}&=\mathcal{I}\begin{pmatrix}N^1 \\ N^2 \\ N^3\end{pmatrix},\quad \mathcal{I}=\begin{pmatrix}
		0 & \frac{c^2}{s^1} & -\frac{c^3}{s^1} \\
		-\frac{c^1}{s^2} & 0 & \frac{c^3}{s^2} \\
		\frac{c^1}{s^3} & -\frac{c^2}{s^3} & 0
	\end{pmatrix},
\end{align}
with $s^i=\sin\theta^i_{\ast}$ and $c^i=\cos\theta^i_{\ast}$. 
\end{lemma}

\begin{proof}
	The proof is carried out in \cite[Lemma 2.3]{depner2013linearized} for the case where the $\Phi^i$ are of the form as in \eqref{EquationNormalGraphs}. The proof there only relies on the geometric situation at the triple junction, i.e., what is the motion of $\Sigma_{\ast}$ in direction of $\nu_{\ast}^i(1)$ and $\tau_{\ast}^i(1)$ for i=1,2,3. Due to \eqref{EquationHelpEquationNormalGraphs} we retrieve this information from the $\varPhi^i$. Therefore, we can apply the same proof in the general situation.\\
	Note that the conditions \eqref{EqRelationNormalPart} and \eqref{EqRelationNormalTangentialPart} simply arise from forming a linear equation system out of the triple junction preservation condition
	$\varPhi^i(1)=\varPhi^j(1)$ 
	for $i,j\in\{1,2,3\}$.
\end{proof}
 \begin{rem}[On the Matrix $\mathcal{I}$ in \eqref{EqRelationNormalTangentialPart}]\label{RemarkOnMatrixI}\ \\ A straightforward calculation shows that $\det(\mathcal{I})=0$, so $\mathcal{I}$ has a kernel. This might be irritating at first as a triple junction preserving motion will need a tangential part, if the normal motion is unequal to zero. To see that this is indeed true, assume that $(N^1, N^2, N^3)^T\in\R^3$ fulfills  \eqref{EqRelationNormalPart}. Observe that vectors solving \eqref{EqRelationNormalPart} are in the hyperplane spanned by $v=(0, -(\tilde{\alpha}^2_{\ast})^{-1}, (\tilde{\alpha}^3_{\ast})^{-1})^T$ and $w=(-(\tilde{\alpha}^1_{\ast})^{-1}, 0, (\tilde{\alpha}^3_{\ast})^{-1})^T$. Then, we have
\begin{align*}
	\mathcal{I}v=\begin{pmatrix}-\frac{c^2}{s^1\tilde{\alpha}^2_{\ast}}-\frac{c^3}{s^1\tilde{\alpha}^3_{\ast}} \\ \frac{c^3}{s^2\tilde{\alpha}^3_{\ast}}\\ \frac{c^2}{s^3\tilde{\alpha}^2_{\ast}} \end{pmatrix}, \qquad \mathcal{I}w=\begin{pmatrix}-\frac{c^3}{s^1\tilde{\alpha}^3_{\ast}} \\\frac{c^1}{s^2\tilde{\alpha}^1_{\ast}}+\frac{c^3}{s^2\tilde{\alpha}^3_{\ast}}\\ -\frac{c^1}{s^3\tilde{\alpha}^1_{\ast}} \end{pmatrix}.
\end{align*}
We see directly that neither $\mathcal{I}v=0$ nor $\mathcal{I}w=0$ holds, as this would imply $c^2=c^3=0$ resp. $c^1=c^3=0$ and thus $\theta^2, \theta^3\in\{\frac{\pi}{2},\frac{3\pi}{2}\}$ resp. $\theta^1, \theta^3\in\{\frac{\pi}{2},\frac{3\pi}{2}\}$. But this would contradict \eqref{EquationAngleRestriction}. Now assume that there is a $d\neq 0$ with $d\mathcal{I}v=\mathcal{I}w$. 
The discussion of the system so obtained yields $c_{2}=c_{3}=0$. So again, we have that $\theta^2,\theta^3\in \{\frac{\pi}{2},\frac{3\pi}{2}\}$, which contradicts \eqref{EquationAngleRestriction}. In total, we see that for any normal part fulfilling \eqref{EqRelationNormalPart} the tangential part given by \eqref{EqRelationNormalTangentialPart} does not vanish. 

\end{rem}
 Lemma ~\ref{LemmaRelationNormalTangentialPart} tells us in particular that the preservation of the triple junction eliminates the tangential degree of freedom at $x=1$. Indeed, we want to eliminate the tangential degree of freedom also for all $x\in[0,1)$. The reason for this is that otherwise we will have problems in the verification of the \L ojasiewicz-Simon gradient inequality in Section \ref{SectionLoja}. Now the idea is to use \eqref{EqRelationNormalTangentialPart} for every $x\in[0,1]$ to get $\mu$ as function in $h$.  Precisely, we  proceed as in the following definition.
 \begin{defi}[Definition of tangential component $\mu=\mu(h)$]\label{defmufinal}\ \\
 For any $h\in (C([0,1],\R))^3$ fulfilling \eqref{EqRelationNormalPart} we define  $\mu=\mu(h)\in (C([0,1],\R))^3$ by
  	\begin{align}\label{EquationCOnstructionTangentialPart}
  		\forall x\in[0,1], i=1,2,3: \mu(h)(x)=\mathcal{I}\begin{pmatrix}
  			h^1(x) \\ h^2(x) \\ h^3(x)\end{pmatrix}
  	\end{align} 
	where $\mathcal{I}$ is as in \eqref{EqRelationNormalTangentialPart}.
	\end{defi} 
  Note that the fixed point boundary condition at $x=0$ is then equivalent to $h=0$ and $h=0$ implies $\mu(h)=0$. The curve resulting from a specific choice of $h$ together with $\mu(h)$ will be denoted by $\Gamma^i_{h}$ and, if not stated otherwise, $\mu$ will always be this specific choice depending on $h$. Now, we have to see that this kind of parametrization is universal in the sense that every network $\Gamma$, which is $C^{2+\alpha}$-close to $\Gamma_{\ast}$, can indeed (up to reparametrization) be written in the form \eqref{EquationNormalGraphs}. This result is stated in the following lemma.

\begin{lemma}[Existence of reference frame graph parametrization]\label{LemmaExistenceRefFramGraphPara}\ \\
	Let $\Gamma_{\ast}, \gamma^i_{\ast},\tilde{\alpha}^i_{\ast}, i=1,2,3$ be as in Assumption \ref{RemarkAssumstionsGammaast} and $\alpha\in(0,1)$. Then, there exists $\sigma(\Gamma_{\ast}, \alpha)>0$ such that any regular parametrizations $(\gamma^i)_{i=1,2,3}$ with 
	\begin{align}
		\gamma^i\in C^{2+\alpha}([0,1], \R^{2}),  \quad \gamma^i(0)=\gamma^i_{\ast}(0),\quad \gamma^1(1)=\gamma^2(1)=\gamma^3(1)
	\end{align}
	and
	\begin{align}
		\sum_{i=1}^3\|\gamma_{\ast}^i-\gamma^i\|_{C^{2+\alpha}([0,1],\R^2)}\le \sigma(\Gamma_{\ast}, \alpha),
	\end{align} 	
	there exists for $i=1,2,3$ functions $h^i,\mu^i\in C^{2+\alpha}([0,1], \R)$ fulfilling  \eqref{EquationCOnstructionTangentialPart} and reparametrizations $\varPhi^i: [0,1]\to[0,1]$ of class $C^{2+\alpha}$ such that 
	\begin{align}\label{EquationHReparametrizingGamma}
		\forall x\in[0,1], i=1,2,3: (\gamma^i\circ\varPhi^i)(x)=\gamma^i_{\ast}(x)+h^i(x)\nu^i_{\ast}+\mu^i(x)\tau_{\ast}^i.
	\end{align}
	Moreover, for any $\delta>0$ there is $\tilde{\sigma}\in(0,\sigma(\Gamma_{\ast}, \alpha))$ such that
	\begin{align}\label{EquationReparametrizationToNormalGraphSmallnesH}
		\sum_{i=1}^3\|\gamma_{\ast}^i-\gamma^i\|_{C^{2+\alpha}([0,1],\R^2)}\le\tilde{\sigma}\Rightarrow \sum_{i=1}^3\|h^i\|_{C^{2+\alpha}([0,1])}+\|\varPhi^i(x)-\mathrm{Id}_{[0,1]}\|_{C^{2+\alpha}([0,1])}\le \delta.
	\end{align}
\end{lemma}
\begin{proof}
	The proof is very similar to the proof of \cite[Lemma 4.1]{AcquaPozziSpenerLSIOpenElasticCurve}. The main difference is the triple junction geometry and - as a consequence - that we have a tangential part in the sought reparametrization. Nevertheless, the main idea is to construct the reparametrization using the implicit function theorem (cf. Theorem~\ref{TheoremImplicitFunction}). \\
	Before we start with the proof itself we want to give a short motivation for it. For technical reasons we will loose one order of differentiability in the construction of $\varPhi$ and $h$. Afterwards we can regain it by an implicit differentiation. To that end it is convenient that $\varPhi$ and $h$ do not appear as free independent parameters. Therefore, we have to find a formulation which only needs $\varPhi$ as a independent parameter. By rearranging \eqref{EquationHReparametrizingGamma} we see that for its solutions we have that
	\begin{align}\label{EquationRelationVarphiH}
		\left((\gamma^i\circ\varPhi^i)(x)-\gamma^i(x)\right)\cdot \nu_{\ast}^i=h^i, \quad \left((\gamma^i\circ\varPhi^i)(x)-\gamma^i(x)\right)\cdot \tau_{\ast}^i=\mu^i.
	\end{align} 
    Additionally, we want the relationship \eqref{EquationCOnstructionTangentialPart} to be fulfilled. Plugging these two facts together allows us to to find a suitable functional analytic setting. We just choose a functional, which tests if \eqref{EquationCOnstructionTangentialPart} is fulfilled for $h$ and $\mu$ given by \eqref{EquationRelationVarphiH}.\\
	To begin with we have to introduce some functions spaces. We set 
	\begin{align*}
		X&:=\{(\gamma^i_P)_{i=1,2,3}\in C^{2+\alpha}([0,1],\R^2)|(\forall i=1,2,3: \gamma^i_P(0)=0)\wedge \gamma^1_P(1)=\gamma^2_P(1)=\gamma^3_P(1)\},\\
		Y&:=\{(\varPhi^i_P)_{i=1,2,3}\in(C^{1+\alpha}([0,1],\R))^3| \forall i=1,2,3, x=0,1: \varPhi^i_P(x)=0\},\\
		D&:=\{g\in C^{1+\alpha}([0,1],\R^3)|g(0)=0\wedge g(1)=0\}.
	\end{align*}

	Before we  can actually define the function to use the implicit function theorem on, we have to choose suitable subsets in the spaces above. This is to ensure that the whole construction will be well-defined.
	For this let $0\in  U\subset X$ be an open neighborhood of zero such that $\gamma_{\ast}^i+\gamma^i_P$ is immersed and regular for all $i=1,2,3$ and $(\gamma_P^i)_{i=1,2,3}\in U$. 
	Furthermore, let $0\in V\subset Y$ be an open neighborhood such that \begin{align}\label{EquationPropertionPerturbationParametrization}\|(\varPhi_P^i)'\|_{\infty}<\frac{1}{2}\end{align}
	for all $i=1,2,3$ and $(\varPhi_P^i)_{i=1,2,3}\in V$.  The reason we want to have \eqref{EquationPropertionPerturbationParametrization} is, that it guarantees that for all $(\varPhi_P^i)_{i=1,2,3}\in V$ we have that $\mathrm{Id}_{[0,1]}+\varPhi_P^i$ is a $C^1$-diffeomorphism  for all $i=1,2,3$. To see this note that with \eqref{EquationPropertionPerturbationParametrization} we have pointwise that $|(\mathrm{Id}_{[0,1]}+\varPhi_P^i)'|\ge\frac{1}{2}>0$. Thus, $\Id_{[0,1]}+\varPhi_P^i$ is strictly increasing on $[0,1]$ and with $(\Id_{[0,1]}+\varPhi_P^i)(0)=0$ and $(\Id_{[0,1]}+\varPhi_P^i)(1)=1$, which hold due to the boundary conditions in $Y$, it follows bijectivity. Additionally, the inverse function theorem yields the fact that we have a $C^1$-diffeomorphism.\\
	Now we define the map
	\begin{align*}
		F: U\times V&\to D,\\
		(\gamma_P, \varPhi_P)&\mapsto  \left([(\gamma_{\ast}^i+\gamma_P^i)\circ(\Id_{[0,1]}+\varPhi_P^i)-\gamma_{\ast}^i]\cdot\tau_{\ast}^i\right)_{i=1,2,3}\\
		&-\mathcal{I}\left([(\gamma_{\ast}^i+\gamma_P^i)\circ(\Id_{[0,1]}+\varPhi_P^i)-\gamma_{\ast}^i]\cdot\nu_{\ast}^i\right)_{i=1,2,3},
	\end{align*}
    where $\mathcal{I}$ is the matrix from \eqref{EqRelationNormalTangentialPart}.
	As we already saw before, the first summand is just the tangential part $\mu$ and the vector after $\mathcal{I}$ is the normal part $h$,  which we would regain from the choice $\Id_{[0,1]}+\varPhi_P$ as reparametrization. Thus, $F$ checks if the pair $(h,\mu)$ resulting from the reparametrization $\Id_{[0,1]}+\varPhi_P$ fulfills \eqref{EquationCOnstructionTangentialPart}. 
	Note that the right-hand side is indeed of class $C^{1+\alpha}$. To see this observe that the composition of $C^{k+\alpha}$-functions is again in $C^{k+\alpha}$ for $k\ge 1$, cf. \cite[Remark B.4]{Dall'AcquaAnnaLinPozziElasticFlowShortTimeExistence}.
	 Thus, we obtain that $(\gamma_{\ast}^i+\gamma_P)\circ(\Id_{[0,1]}+\varPhi_P)\in C^{1+\alpha}$. 
	Additionally, the right-hand side fulfills the boundary conditions at $x=0$ and $x=1$ included in the definition of the space $D$. For this note that in both points we have due to the boundary values of $\varPhi_P$ that 
	\begin{align*}
		F(\gamma_{P}, \varPhi_P)=\left(\gamma^i_P\cdot\tau_{\ast}^i\right)_{i=1,2,3}-\mathcal{I}\left(\gamma^i_P\cdot\nu_{\ast}^i\right)_{i=1,2,3}.
	\end{align*}
    Now, $F(\gamma_{P},\varPhi_P)(0)=0$ is clear due to $\gamma_P(0)=0$. On the other hand, $F(\gamma_{P},\varPhi_P)(1)=0$ follows from Lemma \ref{LemmaRelationNormalTangentialPart}. In total, we see that $F$ is well-defined.
	Furthermore, $F$ is clearly continuous. We now have to study the Fr\'echet-derivative with respect to the second component, which we will denote in the following by $\partial_{2}$. This reduces to the Fr\'echet-derivative of $(\gamma_{\ast}^i+\gamma_P^i)\circ(\Id_{[0,1]}+\varPhi_P^i)$. For this we observe for $(\gamma_P,\varPhi_P)\in U\times V, \tilde{\varPhi}_P\in Y$ small enough and $i=1,2,3$ that we have pointwise
	\begin{align*}
		&(\gamma_{\ast}^i+\gamma^i_P)\circ(\Id_{[0,1]}+\varPhi_P^i+\tilde{\varPhi}_P^i)-(\gamma_{\ast}^i+\gamma^i_P)\circ(\Id_{[0,1]}+\varPhi_P^i)-\tilde{\varPhi}_P^i(\partial_x(\gamma_{\ast}^i+\gamma^i_P))\circ(\Id_{[0,1]}+\varPhi_P^i)\\
		&=\int_0^1 \tilde{\varPhi}_P^i[\partial_x(\gamma_{\ast}^i+\gamma_P^i)\circ(\Id_{[0,1]}+\varPhi_P^i+t\tilde{\varPhi}_P^i)-\partial_x(\gamma_{\ast}^i+\gamma_P^i)\circ(\Id_{[0,1]}+\varPhi_P^i)]dt.
	\end{align*}
    Using the monotonicity of the integral we conclude that
    \begin{align*}
    &\frac{\|(\gamma_{\ast}^i+\gamma^i_P)\circ(\Id_{[0,1]}+\varPhi_P^i+\tilde{\varPhi}_P^i)-(\gamma_{\ast}^i+\gamma^i_P)\circ(\Id_{[0,1]}+\varPhi_P^i)-\tilde{\varPhi}_P^i(\partial_x(\gamma_{\ast}^i+\gamma^i_P))\circ(\Id_{[0,1]}+\varPhi_P^i)\|_{C^{1+\alpha}([0,1],\R)}}{\|\tilde{\varPhi}_P^i\|_{C^{1+\alpha}([0,1],\R)}}\\
    &\le\int_0^1\|\partial_x(\gamma_{\ast}^i+\gamma_P^i)\circ(\Id_{[0,1]}+\varPhi_P+t\tilde{\varPhi}_P)-\partial_x(\gamma_{\ast}^i+\gamma_P^i)\circ(\Id_{[0,1]}+\varPhi_P)\|_{C^{1+\alpha}([0,1],\R)}dt.
    \end{align*} 
    Using again \cite[Remark B4]{Dall'AcquaAnnaLinPozziElasticFlowShortTimeExistence} we see that the integrand tends to zero uniformly in $t$ for $\|\varPhi^i_P\|_{C^{1+\alpha}([0,1],\R)}\to 0$.
	This shows that $\partial_{2}F$ exists on $U\times V$ and is given by
	\begin{align*}
		&\partial_{2}F(\gamma_P, \varPhi_P): Y \to D,\\
		&\tilde{\varPhi}_P\mapsto \left([\tilde{\varPhi}_P^i\partial_x(\gamma_{\ast}^i+\gamma_P^i)\circ(\Id_{[0,1]}+\varPhi_P^i)]\cdot\tau_{\ast}^i\right)_{i=1,2,3}-\mathcal{I}\left([\tilde{\varPhi}_P^i\partial_x(\gamma_{\ast}^i+\gamma_P^i)\circ(\Id_{[0,1]}+\varPhi_P^i)]\cdot\nu_{\ast}^i\right)_{i=1,2,3}.
	\end{align*}
	Clearly, $\partial_{2}F$ is continuous in $(0,0)$ and we have that
	\begin{align*}
		\partial_{2}F(0, 0)\tilde{\varPhi}_P &=  \left([\tilde{\varPhi}_P^i\partial_x\gamma_{\ast}^i\circ\Id_{[0,1]}]\cdot\tau_{\ast}^i\right)_{i=1,2,3}-\mathcal{I}\left([\tilde{\varPhi}_P^i\partial_x\gamma_{\ast}^i\circ\Id_{[0,1]}]\cdot\nu_{\ast}^i\right)_{i=1,2,3}\\
		&=\left(\tilde{\varPhi}_P^i(\partial_x\gamma_{\ast}^i\cdot\tau_{\ast}^i)\right)_{i=1,2,3}=\mathrm{diag}(\partial_x\gamma_{\ast}^1\cdot\tau_{\ast}^1, \partial_x\gamma_{\ast}^2\cdot\tau_{\ast}^2, \partial_x\gamma_{\ast}^3\cdot\tau_{\ast}^3)\tilde{\varPhi}_P.
	\end{align*}
    Due to Assumption \ref{RemarkAssumstionsGammaast} we have that $\mathrm{diag}(\partial_x\gamma_{\ast}^1\cdot\tau_{\ast}^1, \partial_x\gamma_{\ast}^2\cdot\tau_{\ast}^2, \partial_x\gamma_{\ast}^3\cdot\tau_{\ast}^3)^{-1}$ exists and is smooth. Therefore, we see immediately that $\partial_{2}F(0, 0)$ is bijective. \\
	Finally, we see that 
	$$F(0,0)=\left( (\gamma^i_{\ast}-\gamma^i_{\ast})\cdot\tau_{\ast}^i\right)_{i=1,2,3}-\mathcal{I}\left( (\gamma^i_{\ast}-\gamma^i_{\ast})\cdot\nu_{\ast}^i\right)_{i=1,2,3}=0.$$ In total we have shown that all prerequisites to apply the implicit function theorem (see Theorem~\ref{TheoremImplicitFunction}) are valid. This gives us the existence of $\sigma>0$ and $r>0$ such that $B_{\sigma}(0)\times B_{r}(0)\subset U\times V $ and that for any $\gamma_P\in U$ with $\|\gamma\|_X\le \sigma$ we have exactly one $\varPhi_P(\gamma)\in V$ with $\varPhi_P\in B_{r}(0)$ such that $F(\gamma_P,\varPhi_P(\gamma))=0$. Consequently, we can write any family of curves of the form $\gamma_{\ast}+\gamma_P$ with $\gamma\in B_{\sigma}(0)$ in the form \eqref{EquationNormalGraphs} by choosing $\Id_{[0,1]}+\varPhi_P(\gamma)$ as reparametrization and
	\begin{align}\label{EquationDerivingH}
		h=[(\gamma_{\ast}+\gamma_P)\circ(\Id_{[0,1]}+\varPhi_P)-\gamma_{\ast}]\cdot\nu_{\ast}.
	\end{align} 
    Now we have to verify the $C^{2+\alpha}$-regularity of $\varPhi_P$, which then directly implies the $C^{2+\alpha}$-regularity of $h$. This is done by implicit differentiation, i.e., differentiating the identity $F(\gamma_P, \varPhi_P(\gamma_P))=0$ with respect to the space variable. This yields that
    \begin{align*}
    	0&=\left([(1+\partial_x\varPhi_P^i(\gamma_P))\partial_x(\gamma_{\ast}^i+\gamma_P^i)\circ(\Id_{[0,1]}+\varPhi_P^i(\gamma_P))-\partial_x\gamma_{\ast}^i]\cdot\tau_{\ast}^i\right)_{i=1,2,3}\\
    	&-\mathcal{I}\left([(1+\partial_x\varPhi_P^i(\gamma_P))\partial_x(\gamma_{\ast}^i+\gamma_P^i)\circ(\Id_{[0,1]}+\varPhi_P^i(\gamma_P))-\partial_x\gamma_{\ast}^i]\cdot\nu_{\ast}^i\right)_{i=1,2,3}.
    \end{align*}
This can be rearranged to
\begin{align}\label{EquationImpliciteDifferentiationResult}
	A(\gamma_P) (1+\partial_x\varPhi_P^i(\gamma_P))_{i=1,2,3}=(\partial_x\gamma_{\ast}^i\cdot\tau_{\ast}^i)_{i=1,2,3}-\mathcal{I}\left(\partial_x\gamma^i_{\ast}\cdot\nu_{\ast}^i\right)_{i=1,2,3}=(\partial_x\gamma_{\ast}^i\cdot\tau_{\ast}^i)_{i=1,2,3},
\end{align}
with
\begin{align*}
	A(\gamma_P)&=\mathrm{diag}(\left([\partial_x(\gamma_{\ast}^i+\gamma_P^i)\circ(\Id_{[0,1]}+\varPhi^i_P(\gamma_P))]\cdot\tau_{\ast}^i\right)_{i=1,2,3})\\
	&-\mathcal{I}\mathrm{diag}(\left([\partial_x(\gamma_{\ast}^i+\gamma_P^i)\circ(\Id_{[0,1]}+\varPhi^i_P(\gamma_P))]\cdot\nu_{\ast}^i\right)_{i=1,2,3}).
\end{align*}
Note that 
\begin{align*}
	A(0)=\mathrm{diag}(\left(\partial_x\gamma^i_{\ast}\cdot\tau_{\ast}^i\right)_{i=1,2,3}),
\end{align*}
which is invertible due to Assumption \ref{RemarkAssumstionsGammaast}. Thus, for $\gamma^i_P$ and $\varPhi^i_P(\gamma_P)$ close enough to $0$ in the $C^{1+\alpha}$-norm we see that $A(\gamma_P)^{-1}$ exists and has entries in $C^{1+\alpha}$ using Cramer's rule and the general properties of H\"older continuous functions . Then \eqref{EquationImpliciteDifferentiationResult} shows that $\partial_x\varPhi_P$ is in $C^{1+\alpha}$ and thus $\varPhi_P$ is in $C^{2+\alpha}$.\\ 
    The second part of the claim follows by the fact that due to the regularity of $F$ the mapping $\gamma\to \phi(\gamma_P)$ is continuous, \eqref{EquationImpliciteDifferentiationResult} and \eqref{EquationDerivingH}. The precise calculations are similar to \cite[Appendix B.5]{AcquaPozziSpenerLSIOpenElasticCurve}.
\end{proof}

\subsection{Variational formulas}\label{SectionVariations}
In this section we want to study the anisotropic energy $E(\cdot)$ as a functional on function spaces suitable for the LSI approach we want to use later in the stability analysis. Precisely, we consider the functional
\begin{align*}
	E: h\mapsto E(\Gamma_{h,\mu}),
\end{align*}
 where $\Gamma_{h,\mu}$  is the network given by \eqref{EquationNormalGraphs} (with  $\Gamma_{\ast}$ a reference geometry fulfilling Assumption \ref{RemarkAssumstionsGammaast}), $h=(h_i)_{i=1,2,3}$ is a triplet of sufficiently smooth  height functions and $\mu=\mu(h)$ is the tangential part according to Definition~\ref{defmufinal}.
 
Note that as already mentioned in Section \ref{SectionNotation} we will usually denote $\Gamma_{h,\mu}$ by $\Gamma_h$. Also, we will denote by a lower index $h$ a geometric quantity of the curve $\Gamma_h$. For example, $\nu_h$ will denote the unit normal of $\Gamma_{h}$. 

We will calculate the first and second variation of $E$, which are essential for the application \cite[Theorem 2]{FeehannMaridakisLojasiewiczSimon} (and also the original work \cite{Chill2003OnTheLojasiewiczSimon}, which is the foundation to most works on this topic). For the reader's convenience we state \cite[Theorem 2]{FeehannMaridakisLojasiewiczSimon} in the Appendix as Theorem \ref{TheoremFeehanMaridakisTheo2}. Note that we changed the notation from the original article in the way we will use it in this article. 
We set
\begin{align}
	V&:=\left\{h\in H^2([0,1])^3|\sum_{i=1}^3 \tilde{\alpha}^i_{\ast}h^i(1)=0, \forall i=1,2,3: h^i(0)=0\right\},\label{EquationLSISettingV} \\
	W&:=L^2([0,1])^3\times \R^2,\label{EquationLSISettingW}\\
	V&\subset W: ((h^i)_{i=1,2,3})\mapsto  ((h^i)_{i=1,2,3}, 0, 0),\label{EquationLSISettingVToW} \\
	W&\subset V': ((u^i)_{i=1,2,3}, a^1, a^2)\mapsto \left( (v^i)_{i=1,2,3}\mapsto\sum_{i=1}^3\int_0^1 u^iv^i dx+a^1v^1(1)+a^2v^2(1) \right).\label{EquationLSISettingWToVPrime}
\end{align}
Before we move on, we want to give some more explanations on the four lines above. The real Banach space $V$ gives us the domain for the height functions (and thus the considered curves $\Gamma_h$). As the application of Theorem \ref{TheoremFeehanMaridakisTheo2} basically involves solving an elliptic PDE of second order on $\Gamma_{\ast}$, we will need boundary conditions, which fit to the boundary conditions of our flow. Therefore, we encode $\eqref{EquationGeometricProblem}_2-\eqref{EquationGeometricProblem}_3$ in the definition of the space $V$ for the height function $h$ (see the results from Lemma~\ref{LemmaRelationNormalTangentialPart} and Remark~\ref{RemarkOnMatrixI}). The condition $\eqref{EquationGeometricProblem}_4$ cannot be encoded directly into $V$ and will therefore appear as a part of the gradient of $E$.   This is the reason why we have the factor $\R^2$ in the space $W$. 
Now, the embedding $W\subset V'$ defines us in which way we consider the gradient of the energy functional $E$ (see for this the definition of gradient maps in Definition \ref{DefinitionGradientMap}) and the first factor of $W$ determines how the gradient is measured. As we saw in the introduction that our flow is a $L^2$-gradient flow of the anisotropic length functional, these are reasonable choices. To fully understand the choices one has to go through the calculations in Section \ref{SectionStabAna}, especially \eqref{EquationStabAnaMainCalcu}. Finally, we want to mention that we choose the above spaces in order to apply Theorem~\ref{TheoremFeehanMaridakisTheo2} easily. To work with the precise gradient structure of our flow we will need a more geometric version of Theorem \ref{TheoremLSIanalytic}. We will prove this version in the beginning of Section \ref{SectionStabAna}.

Now with the setting  having been introduced we can calculate the variations of $E: V\to \R$ and write them as elements of $W$, i.e., as $\mathcal{M}$,  whereby $\mathcal{M}$ is a gradient map as in Definition \ref{DefinitionGradientMap}.  To clarify the notation we note that when we write the first variation of $E$, we have that
\begin{align}\label{EquationFirstVariationNormalDirection}
	E'(h_0)h_1=\frac{d}{d\varepsilon}\Big|_{\varepsilon=0}E(\gamma_{h_0}+\varepsilon h_1\nu_{\ast}+\varepsilon\mu_1\tau_{\ast}),
\end{align}
with $\gamma_{h_0}=\gamma_{\ast}+h_0\nu_{\ast}+\mu_0\tau_{\ast}$, and  $\mu_{1}=\mu(h_{1})$ resp. $\mu_{0}=\mu(h_{0})$ according to Definition~\ref{defmufinal}. 
\begin{lemma}[First Variation of $E$]\label{LemmaFirstVariE}\ \\
	There is an open neighborhood $U\subset V$ of $0$ such that for any $h_0, h_1\in U$ we have that
	\begin{align}\label{EquFirstVariationE}
		E'(h_0)h_1=&-\sum_{i=1}^3\int_0^1(D^2\varphi^{\circ}(\nu_{h_0}^i)\tau_{h_0}^i\cdot\tau_{h_0}^i)\kappa_{h_0}^i\nu_{h_0}^i\cdot (h_1^i\nu_{\ast}^i+\mu^i(h_1)\tau_{\ast}^i)\,ds_{h_0}^{i}\\
		&-\sum_{i=1}^3D\varphi^{\circ}(\nu_{h_0}^i(1))\cdot (h_1^i(1)\nu_{\ast}^i(1)+\mu^i(h_1)(1)\tau_{\ast}^i(1))^{\bot}.\notag
	\end{align}
	For the corresponding gradient $\mathcal{M}$ with respect to the embedding $W\subset V'$ 
	yielding $$ E'(h_0)h_1=\langle h_{1},\mathcal{M}(h_0)  \rangle_{V \times V'}=
	\sum_{i=1}^{3} \int_{0}^{1} h^{i}_{1} u^{i} ds_{\ast}^{i} + a^{1} h^{1}_{1}(1) + a^{2}h^{2}_{1}(1)$$
	we have that \begin{align}\label{EquationGradientOfE}
		\mathcal{M}(h_0)=((u^i)_{i=1,2,3}, a^1,a^2) 
	\end{align}
	with
	\begin{align*}
		u^i&=\big[-(D^2\varphi^{\circ}(\nu_{h_0}^i)\tau_{h_0}^i\cdot\tau_{h_0}^i)\kappa_{h_0}^i(\nu_{h_0}^i\cdot\nu_{\ast}^i)+(D^2\varphi^{\circ}(\nu_{h_0}^{i+1})\tau_{h_0}^{i+1}\cdot\tau_{h_0}^{i+1})\kappa_{h_0}^{i+1}(\nu_{h_0}^{i+1}\cdot\tau_{\ast}^{i+1})\frac{c^i}{s^{i+1}}\\
		&-(D^2\varphi^{\circ}(\nu_{h_0}^{i+2})\tau_{h_0}^{i+2}\cdot\tau_{h_0}^{i+2})\kappa_{h_0}^{i+2}(\nu_{h_0}^{i+2}\cdot\tau_{\ast}^{i+2})\frac{c^i}{s^{i+2}}\big]J_h^i,\\
		a^1&= D\varphi^{\circ}(\nu^1_{h_0})\cdot\tau_{\ast}^1+\frac{c^1}{s^{2}}D\varphi^{\circ}(\nu_{h_0}^{2})\cdot\nu_{\ast}^{2}-\frac{c^1}{s^{3}}D\varphi^{\circ}(\nu_{h_0}^{3})\cdot\nu_{\ast}^{3}\\
		&-\frac{\tilde{\alpha}_{\ast}^1}{\tilde{\alpha}_{\ast}^3}\left(D\varphi^{\circ}(\nu^3_{h_0})\cdot\tau_{\ast}^3+\frac{c^3}{s^{1}}D\varphi^{\circ}(\nu_{h_0}^{1})\cdot\nu_{\ast}^{1}-\frac{c^3}{s^{2}}D\varphi^{\circ}(\nu_{h_0}^{2})\cdot\nu_{\ast}^{2}\right), \qquad \text{evaluated at $x=1$}\\
		a^2&=D\varphi^{\circ}(\nu^2_{h_0})\cdot\tau_{\ast}^2+\frac{c^2}{s^{3}}D\varphi^{\circ}(\nu_{h_0}^{3})\cdot\nu_{\ast}^{3}-\frac{c^2}{s^{1}}D\varphi^{\circ}(\nu_{h_0}^{1})\cdot\nu_{\ast}^{1} \\
		&-\frac{\tilde{\alpha}_{\ast}^2}{\tilde{\alpha}_{\ast}^3}\left(D\varphi^{\circ}(\nu^3_{h_0})\cdot\tau_{\ast}^3+\frac{c^3}{s^{1}}D\varphi^{\circ}(\nu_{h_0}^{1})\cdot\nu_{\ast}^{1}-\frac{c^3}{s^{2}}D\varphi^{\circ}(\nu_{h_0}^{2})\cdot\nu_{\ast}^{2}\right), \qquad \text{evaluated at $x=1$}.
	\end{align*}
	Hereby, all appearing upper indices are to be understood modulo $3$ plus $1$, i.e., $4$ equals $1$, $5$ equals $2$ and $6$ equals $3$. Also, we denote by $J_h^i$ the length element of $\Gamma^i_h$, i.e.,
	\begin{align*}
		ds^i_h=J_h^idx.
	\end{align*}

\end{lemma}
\begin{proof}
	The first formula \eqref{EquFirstVariationE} is a direct consequence from \cite[P. 7]{KroenerNovagaPozziAnisoCurvFloImmeredNetworks} applied to the specific variation \eqref{EquationFirstVariationNormalDirection}. Now to rewrite this as a gradient with respect to $W\subset V'$, we have to write the $\mu^i(h_1)$ in terms of the $h_1^i$. Using \eqref{EqRelationNormalTangentialPart} we have that
	\begin{align*}
		\mu^1(h_1)\tau_{\ast}^1&=\left(\frac{c^2}{s^1}h_1^2-\frac{c^3}{s^1}h_1^3\right)\tau_{\ast}^1,\\
		\mu^2(h_1)\tau_{\ast}^2&=\left(\frac{c^3}{s^2}h_1^3-\frac{c^1}{s^2}h_1^1\right)\tau_{\ast}^2,\\
		\mu^3(h_1)\tau_{\ast}^3&=\left(\frac{c^1}{s^3}h_1^1-\frac{c^2}{s^3}h_1^2\right)\tau_{\ast}^3.
	\end{align*}
In other words
	\begin{align*}
	\mu^i(h_1)\tau_{\ast}^i&=\left(\frac{c^{i+1}}{s^i}h_1^{i+1}-\frac{c^{i+2}}{s^i}h_1^{i+2}\right) \tau_{\ast}^i \qquad  i=1,2,3.
	\end{align*}
	Then the result for the $u^i$ follows by sorting the terms from the $\mu^i(h_1)$ by $h_1^i$. For the boundary terms it is a bit more complicated as we also have to rewrite $h_1^3$ in terms of $h_1^1$ and $h_1^2$ as we can only have two components here. To begin with we observe - again using \eqref{EqRelationNormalTangentialPart} and the facts that $(\nu_{\ast}^i)^{\bot}=-\tau_{\ast}^i$ and $(\tau_{\ast}^i)^{\bot}=\nu_{\ast}^i$ - that the boundary terms at $x=1$ - so the second line in \eqref{EquFirstVariationE} - equals
	\begin{align*}
	&-\sum_{i=1}^3D\varphi^{\circ}(\nu_{h_0}^i(1))\cdot (h_1^i(1)\nu_{\ast}^i(1)+\mu^i(h_1)(1)\tau_{\ast}^i(1))^{\bot}\\
	&	=\sum_{i=1}^3h_1^i\left( D\varphi^{\circ}(\nu^i_{h_0})\cdot\tau_{\ast}^i+\frac{c^i}{s^{i+1}}D\varphi^{\circ}(\nu_{h_0}^{i+1})\cdot\nu_{\ast}^{i+1}-\frac{c^i}{s^{i+2}}D\varphi^{\circ}(\nu_{h_0}^{i+2})\cdot\nu_{\ast}^{i+2}\right)\Big|_{x=1}.
	\end{align*}
	To rewrite $h_1^3$ we observe that \begin{align}\label{EquRewrtingH3}
		h_1^3=-\frac{\tilde{\alpha}_{\ast}^1}{\tilde{\alpha}_{\ast}^3}h_1^1-\frac{\tilde{\alpha}_{\ast}^2}{\tilde{\alpha}_{\ast}^3}h_1^2.
	\end{align}
	Again resorting terms gives the sought formulas.
\end{proof}
Before we also give formulas for the second variation of $E$ in $0$ we recall that for $h_0,h_1\in V$ we have that
\begin{align*}
	E''(0)h_0h_1=\frac{d}{d\varepsilon}\frac{d}{d\eta}\Big|_{\varepsilon=0,\eta=0}E(\gamma_{\ast}+\varepsilon h_0\nu_{\ast}+\eta h_1\nu_{\ast}+\varepsilon\mu_0\tau_{\ast}+\eta\mu_1\tau_{\ast}).
\end{align*} 
again with  $\mu_{1}=\mu(h_{1})$ resp. $\mu_{0}=\mu(h_{0})$ according to Definition~\ref{defmufinal}.
Note that due to Definition~\ref{DefinitionGradientMap} we can use $E''(0)$ to calculate $\mathcal{M}'(0)$.
\begin{lemma}[Second Variation of $E$ in $0$]\label{LemmaSecondVariE}\ \\
	For any $h_0, h_1\in V$ we have that
		\begin{align}\label{EquSecondVariationE}
			E''(0)h_0h_1=-&\sum_{i=1}^3\int_0^1 (D^2\varphi^{\circ}(\nu_{\ast}^i))\tau_{\ast}^i\cdot\tau_{\ast}^i)\left(\frac{(h_0^i)''}{|(\gamma_{\ast}^i)'|^2}-\frac{(h_0^i)'\langle(\gamma_{\ast}^{i})'',\tau^{i}_{\ast}\rangle}{|(\gamma_{\ast}^i)'|^3}\right)h_1^ids_{\ast}^i\\
			+&\sum_{i=1}^3(h_0^i)'(1)h_1^i(1)\left(D^2\varphi^{\circ}(\nu_{\ast}^i(1))\tau_{\ast}^i(1)\cdot\tau_{\ast}^i(1)\right).\notag
	\end{align}
	As a consequence we obtain for  $\mathcal{M}'(0)$ 
	$$ E''(0)h_0h_1=\langle h_{1},\mathcal{M}'(0)(h_0)  \rangle_{V \times V'}=
	\sum_{i=1}^{3} \int_{0}^{1} h^{i}_{1} u^{i} ds_{\ast}^{i} + a^{1} h^{1}_{1}(1) + a^{2}h^{2}_{1}(1)$$
     that
	\begin{align}\label{EquationGradientDerivative}
		\mathcal{M}'(0)(h_0)=((u^i)_{i=1,2,3}, a^1,a^2)
	\end{align}
	with
		\begin{align*}
			u^i&=-(D^2\varphi^{\circ}(\nu_{\ast}^i))\tau_{\ast}^i\cdot\tau_{\ast}^i)\left(\frac{(h_0^i)''}{|(\gamma_{\ast}^i)'|^2}-\frac{(h_0^i)'\langle(\gamma_{\ast}^{i})'',\tau_{\ast}^{i}\rangle}{|(\gamma_{\ast}^i)'|^3}\right)J^i_{\ast},\\
			a^1&=(h_0^1)'(1)\DZweiTauNu{1}-\frac{\tilde{\alpha}_{\ast}^1}{\tilde{\alpha}_{\ast}^3}(h_0^3)'(1)\DZweiTauNu{3},\\
			a^2&=(h_0^2)'(1)\DZweiTauNu{2}-\frac{\tilde{\alpha}_{\ast}^2}{\tilde{\alpha}_{\ast}^3}(h_0^3)'(1)\DZweiTauNu{3}.
		\end{align*}
		Hereby, the $J^i_{\ast}$ denote the length element of $\Gamma^i_{\ast}$, i.e.,
		\begin{align*}
			ds^i_{\ast}=J^i_{\ast}dx.
	\end{align*} 
\end{lemma}
\begin{proof}
	Before we calculate $E''(0)$ we observe that due to the homogeneity of the norm we have for a network $\Gamma$ parametrized by $(\gamma^i)_{i=1,2,3}$ that
		\begin{align*}
			E(\Gamma)=\sum_{i=1}^3\int_{\Gamma^i}\varphi^{\circ}(\nu^i)ds^i=\sum_{i=1}^3\int_0^1\varphi^{\circ}((\gamma^i_x)^{\bot}|(\gamma^i)_x|^{-1})|\gamma^i_x|dx=\sum_{i=1}^3\int_0^1\varphi^{\circ}((\gamma^i_x)^{\bot})dx.
		\end{align*}
	Using this identity and the fact that $\tau_{\ast}^i$ and $\nu_{\ast}^i$ are constant on $\Gamma^i_{\ast}$ for $i=1,2,3$, we see directly that
	\begin{align*}
		E''(0)h_0h_1&=\sum_{i=1}^3\int_0^1D^2\varphi^{\circ}((\gamma_{\ast}^i)_x^{\bot})((h_0^i)'\nu^i_{\ast}+(\mu^i(h_0))'\tau_{\ast}^i)^{\bot}((h_1^i)'\nu^i_{\ast}+(\mu^i(h_1))'\tau^i_{\ast})^{\bot}dx\\
		&=\sum_{i=1}^3\int_0^1|(\gamma^i_{\ast})_x|^{-1}D^2\varphi^{\circ}(\nu^i_{\ast})((h_0^i)'\nu^i_{\ast}+(\mu^i(h_0))'\tau_{\ast}^i)^{\bot}((h_1^i)'\nu^i_{\ast}+(\mu^i(h_1))'\tau^i_{\ast})^{\bot}dx\\
		&=\sum_{i=1}^3\int_0^1|(\gamma^i_{\ast})_x|^{-1}(D^2\varphi^{\circ}(\nu^i_{\ast})\tau_{\ast}^i\cdot\tau_{\ast}^i)(h_0^i)'(h_1^i)'dx.
	\end{align*}
Here, we used in the second line the homogeneity of $\varphi^{\circ}$ and in the third line Lemma \ref{lemma2.1}iii.). Now using integration by parts we see that
\begin{align*}
		E''(0)h_0h_1=-&\sum_{i=1}^3\int_0^1 (D^2\varphi^{\circ}(\nu_{\ast}^i))\tau_{\ast}^i\cdot\tau_{\ast}^i)\left(\frac{(h_0^i)''}{|(\gamma_{\ast}^i)'|}-\frac{(h_0^i)'\langle(\gamma_{\ast}^{i})'',\tau_{\ast}^{i}\rangle}{|(\gamma_{\ast}^i)'|^2}\right)h_1^idx\\
	+&\sum_{i=1}^3\left[|(\gamma^i_{\ast})_x|^{-1}(h_0^i)'h_1^i\left(D^2\varphi^{\circ}(\nu_{\ast}^i)\tau_{\ast}^i\cdot\tau_{\ast}^i\right)\right]_0^1\\
	=-&\sum_{i=1}^3\int_0^1 (D^2\varphi^{\circ}(\nu_{\ast}^i))\tau_{\ast}^i\cdot\tau_{\ast}^i)\left(\frac{(h_0^i)''}{|(\gamma_{\ast}^i)'|^2}-\frac{(h_0^i)'\langle(\gamma_{\ast}^{i})'',\tau_{\ast}^{i}\rangle}{|(\gamma_{\ast}^i)'|^3}\right)h_1^ids^i_{\ast}\\
	+&\sum_{i=1}^3(h_0^i)'(1)h_1^i(1)\left(D^2\varphi^{\circ}(\nu_{\ast}^i(1))\tau_{\ast}^i(1)\cdot\tau_{\ast}^i(1)\right).
\end{align*}
Hereby, we used in the second step the boundary conditions for $h_1^i$ due to the definition of $V$ and \eqref{EquationConditionParaOfGammaAst}. In total we showed 
\eqref{EquSecondVariationE}.

	Concerning the second part of the claim, 
	the formula for the $u^i$ is trivial. For the boundary terms we will use the abbreviation
	\begin{align}\label{EquAbbrevD2}
		D^i:=\DZweiTauNu{i}.
	\end{align}
	Using \eqref{EquRewrtingH3} we can rewrite the boundary terms to
	\begin{align*}
		(h_0^1)'(1)h_1^1(1)D^1+(h_0^2)'(1)h_1^2(1)D^2+(h_0^3)'(-\frac{\tilde{\alpha}_{\ast}^1}{\tilde{\alpha}_{\ast}^3}h_1^1(1)-\frac{\tilde{\alpha}_{\ast}^2}{\tilde{\alpha}_{\ast}^3}h_1^2(1))D^3.
	\end{align*}
	Sorting the terms by $h_1^1(1)$ and $h_1^2(1)$ gives the wished result. 
\end{proof}

\subsection{\L ojasiewicz-Simon gradient inequality}\label{SectionLoja}
In this section we will show the prerequisites to apply Theorem \ref{TheoremFeehanMaridakisTheo2}. These are analyticity of $\mathcal{M}: U\to W$ for some neighborhood $0\in U\subset V$ and that $\mathcal{M}'(0): V\to W$ is a Fredholm operator of index $0$. We will do this in Lemma \ref{LemmaAnalyticityEPrime} and Lemma \ref{LemmaFredholmPropEPrimePrime}. Before we tackle the first result, we shortly recall the definition of analytic operators (see, e.g., \cite[Definition 8.8]{ZeidlerBook1}) and some basic properties. 
\begin{defi}[Analytic operator]\label{DefinitionAnalyticOperator}\ \\
	Let $X,Y$ be Banach spaces and $U\subset X$ open. A map $f: U\to Y$ is analytic in $x_0\in U$ if there is a sequence $(a_k)_{k\in\N_0}$, where $a_k: X^k\to Y$ is a $k$-linear, symmetric and continuous map for each $k\in\N_0$, such that on a neighborhood of $x_0$ we have that
	\begin{align*}
		\sum_{k=0}^{\infty}\|a_k\|\ \|x-x_0\|_X^k\text{ converges and } f(x)=\sum_{k=0}^{\infty}a_k(x-x_0)^k.
	\end{align*}
	Note that when we write $a_k(x-x_0)^k$ we mean $a_k(\underbrace{x-x_0,\cdots,x-x_0}_{k-\text{times}})$.
\end{defi}
\begin{lemma}[Properties of analytic function]\label{LemmaPropertiesAnalyticFunction}\ \\
	Let $X, Y, \tilde{Y}, Z$ Banach spaces, $U\subset X$ open, $f_1,f_2: U\to Y, g: U\to\tilde{Y}$ analytic in $x_0\in U$ and $\tilde{g}: \tilde{Y}\to Z$ analytic in $\im(g)$. Then we have the following properties.
	\begin{itemize}
		\item[i.)] Analyticity in one point implies analyticity in a neighborhood of that point. 
		\item[ii.)] $f_1+f_2: U\to Y$ is analytic in $x_0$.
		\item[iii.)] If $\cdot: Y\times\tilde{Y}\to Z$ is a bilinear, continuous mapping, then 
		$$f_1\cdot g: U\to Z, x\mapsto f_1(x)\cdot g(x),$$
		is analytic in $x_0$. 
		\item[iv.)] The map $\tilde{g}\circ g: U\to Z$ is analytic in $x_0$.
		\item[v.)] Any bounded affine map is analytic.  
	\end{itemize}
\end{lemma}
\begin{proof}
	The first part follows directly from the definition of analyticity. For ii.) we just add the power series. For iii.) one can argue similar to the situation for the Cauchy product of real valued analytic functions. Part iv.) is for example discussed on \cite[p. 1079]{WhittleseyAnalyticFunctionBanachSpaces}. Finally, the last part is obvious as one can choose in the power series $a_k\equiv 0$ for all $k\ge 2$.
\end{proof}
\begin{lemma}[Analyticity of $\mathcal{M}$]\label{LemmaAnalyticityEPrime}\ \\
	Let $V,W, V\subset W\subset V'$  be as in \eqref{EquationLSISettingV}-\eqref{EquationLSISettingWToVPrime}.
	There is a neighborhood $U\subset V$ of $0$ such that $\mathcal{M}: U\to W$ is analytic, where $\mathcal{M}$ is as in \eqref{EquationGradientOfE}. 
\end{lemma}
\begin{proof}
	Similar calculations for all geometric quantities were done in \cite[Lemma 3.4]{AcquaPozziSpenerLSIOpenElasticCurve} for a curve without additional tangential part and in \cite[Lemma 8.1]{GarckeGoessweinNonlinearStabilityDoubleBubbleSDF} for a higher dimensional triple junction geometry with the same tangential part. For the reader's convenience we will sum up the most important ideas.\\
	Studying the formula for \eqref{EquationGradientOfE} wee see that it consists of sums and products of elemental geometric quantities. So if we prove their analyticity as functions in $h_0$, we can use Lemma \ref{LemmaPropertiesAnalyticFunction}ii.) and iii.) together with the fact that $H^2([0,1])$ and $H^1([0,1])$ are Banach algebras (see e.g. \cite[Cor. 8.10]{BrezisFunctionalAnalysis}) to show that the whole composition is analytic. We have to make one exception for the factor $\kappa_{h_0}^i$ as it takes values in $L^2$. But we will take care of this later.\\  
	We begin with the tangential vector $\tau^i_{h_0}$, for which we have the formula
	
	\begin{align*}
		\tau^i_{h_0}=\frac{\gamma_{h_{0}}'}{|\gamma_{h_{0}}'|}=\frac{(\gamma^i_{\ast})'+(h^i_0)'\nu_{\ast}^i+(\mu^i(h_0))'\tau_{\ast}^i}{|(\gamma^i_{\ast})'+(h^i_0)'\nu_{\ast}^i+(\mu^i(h_0))'\tau_{\ast}^i|}.
	\end{align*}
Let us first study the numerator.
Note that due to Lemma \ref{LemmaPropertiesAnalyticFunction}v.) and \eqref{EquationCOnstructionTangentialPart}, the tangential part $\mu$ forms an analytic function $V\to H^2([0,1])^3, h_0\mapsto \mu(h_0)$. Thus, the numerator is an analytic function
 \begin{align*} 
 	V\to H^1([0,1]; \R^2)^3, h_0\to \gamma_{\ast}'+h_0'\nu_{\ast}+\mu(h_0)'\tau_{\ast}
 \end{align*}
 due to Lemma \ref{LemmaPropertiesAnalyticFunction}v.). Consequently, the denominator is also an analytic function $V\to H^1([0,1]; \R^2)^3$ due to Lemma \ref{LemmaPropertiesAnalyticFunction}iv.). Using again Lemma \ref{LemmaPropertiesAnalyticFunction}iv.) and then Lemma \ref{LemmaPropertiesAnalyticFunction}iii.) we get that $h_0\to \tau_{h_0}$ is an analytic function $U\to H^1([0,1]; \R^2)^3$ for a neighborhood $0\in U\subset V$ small enough such that $\tau_{h_0}$ is well-defined. Note that we used the fact that $x^{-1}$ is analytic on $\R\backslash\{0\}$. The same holds for $h_0\mapsto\nu_{h_0}$ using Lemma \ref{LemmaPropertiesAnalyticFunction}iv.) and v.). Furthermore, $U\to H^1([0,1],\R^{2\times 2})^3, h_0\to D^2\varphi^{\circ}(\nu_{h_0})$ is analytic due to the prior results and Lemma \ref{LemmaPropertiesAnalyticFunction} iv.) using that $D^2\varphi^{\circ}$ is analytic. For the metric tensor we see that $J_{h_0}$ is analytic as map $U\to H^1([0,1])^3$ due to Lemma \ref{LemmaPropertiesAnalyticFunction} iv.).
 For the curvature a direct calculation shows that
 \begin{align}\label{EquationFormulaKappaH0}
 	\kappa_{h_0}=\frac{\det(\gamma_{h_0}',\gamma_{h_0}'')}{|\gamma_{h_0}'|^3}=\frac{(|\gamma_{\ast}'|+\mu(h_0)')h_0''-h_0'(\langle\gamma_{\ast}'',\tau_{\ast}\rangle+\mu(h_0)'')}{\sqrt{(h_0')^2+(|\gamma_{\ast}'|+\mu(h_0)')^2}^3}.
 \end{align}

Here we have to be a bit careful as both $h_0''$ and $\mu(h_0)''$ take values in $L^2$, which is not a Banach algebra. But as we only take products between $L^2$ and $H^1$ functions, we can use the fact that the pointwise multiplication forms a bounded, bilinear operator 
	\begin{align}\label{EquationProductBetweenL2AndH1} 
		H^1([0,1])\times L^2([0,1])\to L^2([0,1]), \, (f,g)\mapsto fg.
	\end{align} Using again Lemma \ref{LemmaPropertiesAnalyticFunction}iii.) we can argue similarly as before to see that $\kappa_{h_0}: U \to L^2([0,1])^3$ is analytic. To finish the argumentation for the $u$-terms in the formula for $\mathcal{M}(h_0)$ in \eqref{EquationGradientOfE} we again use that \eqref{EquationProductBetweenL2AndH1} is a bounded, bilinear operator and argue similarly as for $\kappa_{h_0}$.\\
	For the $a^1$ and $a^2$ the argumentation is even easier as these are only real valued and therefore all appearing products are easier to handle. In total this finishes the proof.
\end{proof}
\begin{lemma}[Fredholm property of $\mathcal{M}'(0)$]\label{LemmaFredholmPropEPrimePrime}\ \\
	Let $V,W, V\subset W\subset V'$ be as in \eqref{EquationLSISettingV}-\eqref{EquationLSISettingWToVPrime} and $\mathcal{M}'(0)$ as in Lemma \ref{LemmaSecondVariE}.The function $\mathcal{M}'(0): V\to W$ is a Fredholm operator of index $0$.
\end{lemma}
\begin{proof}
	Note that compact perturbations of a Fredholm operator preserve both the Fredholm property and the index, c.f. for example \cite[Proposition 8.14]{ZeidlerBook1}. Furthermore, the $(h_0^i)'$-terms in the $u^i$-terms in \eqref{EquationGradientDerivative} are compact perturbations as they take values in $H^1$ and we have the compact embedding $H^1\subset L^2$. Therefore, we can ignore these terms. Thus, for the $u$-terms in \eqref{EquationGradientDerivative} it remains to study
	\begin{align*}
		-J^i_{\ast}\frac{(D^2\varphi^{\circ}(\nu_{\ast}^i))\tau_{\ast}^i\cdot\tau_{\ast}^i)}{|(\gamma_{\ast}^i)'|^2}h_0''.
	\end{align*} 
	Observe that the coefficient before $h_0''$ is bounded above by a negative constant due to Lemma \ref{lemma2.1}i.). Therefore this forms a surjective mapping $H^2([0,1])^3\to L^2([0,1])^3$, whose kernel are affine functions. The boundary conditions of $V$, namely $h_0^i(0)=0, i=1,2,3,$ and $\sum_{i=1}^3\tilde{\alpha}_{\ast}^ih_0^i(1)=0$, only allow affine functions, which are elements of
	\begin{align}\label{EquationVaff}
	V_{aff}:=\left\{h_0: [0,1]\to\R^3, x\mapsto (\psi^1x, \psi^2x, \psi^3x)\big| \psi^1,\psi^2,\psi^3\in\R\wedge\sum_{i=1}^3\tilde{\alpha}_{\ast}^i\psi^i=0\right\}.  
	\end{align}
    
	Obviously $V_{aff}$ is a two dimensional subspace of $V$  and
	\begin{equation}\label{deppsi3}
	\psi^{3}= -\frac{\tilde{\alpha}_{\ast}^{1}}{\tilde{\alpha}_{\ast}^{3}}\psi^{1} -\frac{\tilde{\alpha}_{\ast}^{2}}{\tilde{\alpha}_{\ast}^{3}}\psi^{2}.
	\end{equation}
	 Our goal is now to see that the mapping 
	\begin{align}\label{EquationBoundaryOperatorOnAAff}
	V_{aff}\to \R^2, h_0\to (a^1(h_0),a^2(h_0)),
    \end{align} 
    where $a^1,a^2$ are as in Lemma \ref{LemmaSecondVariE}, is bijective. This then shows that the main part of $\mathcal{M}'(0)$ is bijective and consequently a Fredholm operator of index $0$. The perturbation argument mentioned in the beginning yields then the claim.
    
	In order to study \eqref{EquationBoundaryOperatorOnAAff} we have to write the right-hand side as function solely in $\psi^1$ and $\psi^2$ using the sum condition for $\psi^i$, where the $\psi^i$ are as in $\eqref{EquationVaff}$. This yields - using \eqref{deppsi3} as well as again the abbreviation \eqref{EquAbbrevD2} - that
		\begin{align*}
			a^1&=\psi^1D^1+\frac{(\tilde{\alpha}_{\ast}^1)^2}{(\tilde{\alpha}_{\ast}^3)^2}\psi^1D^3+\frac{\tilde{\alpha}_{\ast}^1\tilde{\alpha}_{\ast}^2}{(\tilde{\alpha}_{\ast}^3)^2}\psi^2D^3,\\
			a^2&=\psi^2D^2+\frac{\tilde{\alpha}_{\ast}^1\tilde{\alpha}_{\ast}^2}{(\tilde{\alpha}_{\ast}^3)^2}\psi^1D^3+\frac{(\tilde{\alpha}_{\ast}^2)^2}{(\tilde{\alpha}_{\ast}^3)^2}\psi^2D^3.
		\end{align*}
	
		This can be written as
		\begin{align*}
			\begin{pmatrix} a^1\\a^2\end{pmatrix}=F\begin{pmatrix}\psi^1\\ \psi^2	\end{pmatrix}, \quad F=\begin{pmatrix}
				D^1+\frac{(\tilde{\alpha}_{\ast}^1)^2}{(\tilde{\alpha}_{\ast}^3)^2}D^3 & \frac{\tilde{\alpha}_{\ast}^1\tilde{\alpha}_{\ast}^2}{(\tilde{\alpha}_{\ast}^3)^2}D^3 \\
				\frac{\tilde{\alpha}_{\ast}^1\tilde{\alpha}_{\ast}^2}{(\tilde{\alpha}_{\ast}^3)^2}D^3 & D^2+\frac{(\tilde{\alpha}_{\ast}^2)^2}{(\tilde{\alpha}_{\ast}^3)^2}D^3
			\end{pmatrix}.
		\end{align*}
	
		Now we observe that 
		
		\begin{align*}
			\det(F)=D^1D^2+D^1D^3\frac{(\tilde{\alpha}_{\ast}^2)^2}{(\tilde{\alpha}_{\ast}^3)^2}+D^2D^3\frac{(\tilde{\alpha}_{\ast}^1)^2}{(\tilde{\alpha}_{\ast}^3)^2}.
		\end{align*}
	
		As the $D^i$ are positive due to Lemma \ref{lemma2.1}i.), we conclude that $\det(F)>0$. Thus, $F$ is invertible and so \eqref{EquationBoundaryOperatorOnAAff} is bijective. In total, this shows that the main part of $\mathcal{M}'(0)$ is a bijective mapping $V\to W$. This finishes the proof.
\end{proof}
With these two properties we can prove the first version of a LSI for the anisotropic length energy.
\begin{teo}[\L ojasiewicz-Simon gradient inequality for $E$ - Analytic Version]\label{TheoremLSIanalytic}\ \\
	Let $V,W, V\subset W\subset V'$ be as in \eqref{EquationLSISettingV}-\eqref{EquationLSISettingWToVPrime} and $\mathcal{M}$ as in Lemma \ref{LemmaFirstVariE}.There exist $\sigma_{LSI},C_{LSI}>0$ and $\theta\in(0,\frac{1}{2}]$ such that for all $h\in V$ with $\|h\|_V\le\sigma_{LSI}$ it holds that
	\begin{align}\label{EquationLSIanalytic}
		|E(h)-E(0)|^{1-\theta}\le C_{LSI}\|\mathcal{M}(h)\|_W.
	\end{align}
	Hereby, $\mathcal{M}(h)\in W$ is given by \eqref{EquationGradientOfE}. 
\end{teo}
\begin{proof}
	We first check that $V\subset V'$ with respect to \eqref{EquationLSISettingVToW}, \eqref{EquationLSISettingWToVPrime} is a definite embedding, cf. Appendix \ref{AppendixUsedResults}. For this we observe that the canonical pairing yields the bilinear form
	\begin{align*}
		F: V\times V\to\R, (f,g)\mapsto \sum_{i=1}^3\int_0^1f^ig^idx.
    \end{align*} 
    Note that the additional components in \eqref{EquationLSISettingWToVPrime} vanish due to the embedding \eqref{EquationLSISettingVToW}. In other words, $F$ is the $L^2$ product on $V$, which is clearly definite. Thus, $V\subset V'$ is definite.
    
    Now, Lemma \ref{LemmaPropertiesAnalyticFunction}iv.) and the considerations in Lemma \ref{LemmaAnalyticityEPrime} show that there is neighborhood $U\subset V$ such that (considered as functions in the variable $h_0$) $\nu_{h_0}: U\to (H^1([0,1], \R^2))^3, J_{h_0}: U\to (H^1([0,1]))^3$ and $\varphi^{\circ}(\nu_{h_0}): U\to (H^1([0,1]))^3$ are analytic. Together with the analyticity of $\int: H^1([0,1])^3\to\R$ and again Lemma \ref{LemmaPropertiesAnalyticFunction}iv.) this shows that $E: U\to\R$ is analytic and thus in $C^2(U, \R)$, cf. \cite[p. 362]{ZeidlerBook1}. 
    
    With these two facts the claim follows directly from Theorem \ref{TheoremFeehanMaridakisTheo2} using Lemma~\ref{LemmaAnalyticityEPrime} and Lemma~\ref{LemmaFredholmPropEPrimePrime}.
\end{proof} 
\section{Analytic properties of the height function}\label{SectionAnalysisH}
	Before we can finally carry out the stability analysis in the next section, we have to verify three properties for the height function $h$ we used so far. These properties  will  play a crucial role also in the stability analysis later on. Firstly, we will derive time regularity, when $h$ parametrizes a geometric solution of \eqref{EquationGeometricProblem} over a given reference frame (cf. Assumption~\ref{RemarkAssumstionsGammaast}). Lemma \ref{LemmaExistenceRefFramGraphPara} alone is not sufficient to guarantee this. Thus, we will prove a short time existence result for $h$ in Lemma~\ref{LemmaSTEForH}. Secondly, we need a bound for the higher norms of such a solution, i.e., a parabolic smoothing property. This is proven in Lemma \ref{LemmaParabolicSmoothigH}. Finally, we need continuous dependency on the initial data of the norm of the solution from Lemma \ref{LemmaSTEForH}. This is carried out in Lemma \ref{LemmaContinuousDependency}.
	For our convenience and for consistency with the short time existence analysis carried out already in \cite{KroenerNovagaPozziAnisoCurvFloImmeredNetworks} we work in the following with H\"older spaces.

\bigskip

	To begin with, we have to translate \eqref{EquationGeometricProblem} into the resulting system for $h$. 
	We recall that due to the work in Section \ref{SubsectionParametrization} we have that
	\begin{align*}
		\Gamma^i_h(t)=\gamma^i_h(t,[0,1]),\qquad i=1,2,3,
	\end{align*}
    with
    \begin{align*}
    	\gamma^i_h(t,x)=\gamma^i_{\ast}(x)+h^i(t,x)\nu^i_{\ast}+\mu^i(t,x)\tau^i_{\ast},
    \end{align*}
    whereby
    \begin{align*}
    	(\mu^i(t,x))_{i=1,2,3}=\mathcal{I}(h^i(t,x))_{i=1,2,3},
    \end{align*}
    with $\mathcal{I}$ as in \eqref{EqRelationNormalTangentialPart}. Using Lemma \ref{LemmaRelationNormalTangentialPart} we deduce from \eqref{EquationGeometricProblem} the following problem for $h$, whereby as usual an index $h$ denotes a geometric quantity of $\Gamma_h$.
	\begin{align}\label{EquationSystemForH}
		\begin{cases}
			V^i_h=\varphi^{\circ}(\nu^i_h)(D^2\varphi^{\circ}(\nu^i_h)\tau^i_h\cdot\tau^i_h)\kappa^i_h & \text{on }(0,T)\times(0,1), i=1,2,3,\\
			h^i(t,0)=0 & \forall t\in[0,T], i=1,2,3,\\
			\sum_{i=1}^3\tilde{\alpha}^i_{\ast}h^i(t,1)=0 &\forall t\in[0,T],\\
			\sum_{i=1}^3D\varphi^{\circ}(\nu^i_h(t,1))=0 &\forall t\in[0,T],\\
			h(0,x)=h_0(x) & \forall x\in [0,1], i=1,2,3.
		\end{cases}
	\end{align}
	Although this system seems to have no good structure for an equation in $h$, it can actually be written as a parabolic PDE system for $h$. To see this, observe that
	\begin{align}\label{EquationNormalVelocityInH}
		V^i_h=\partial_th^i(\nu_{\ast}^i\cdot\nu_h^i)+\partial_t\mu^i(h)(\tau_{\ast}^i\cdot\nu_h^i).
	\end{align}
	Now, $\partial_t\mu(h)$ can be written in terms of $\partial_th$ using \eqref{EquationCOnstructionTangentialPart}, yielding
	\begin{align*}
		(V_h^i)_{i=1,2,3}=F_h\partial_t h,\quad F_h=\mathop{diag}(\nu^1_h\cdot\nu^1_{\ast}, \nu_h^2\cdot\nu_{\ast}^2,\nu_h^3\cdot\nu_{\ast}^{3})+\mathop{diag}(\nu_{h}^1\cdot\tau_{\ast}^1, \nu_h^2\cdot\tau_{\ast}^2, \nu_h^3\cdot\tau_{\ast}^3)\mathcal{I}.
	\end{align*}
    Thus, $\eqref{EquationSystemForH}_1$ can be written as 
    \begin{align}\label{EquationParabolicVersionOfPDE}
    	F_h\partial_th=(\varphi^{\circ}(\nu^i_h)(D^2\varphi^{\circ}(\nu^i_h)\tau^i_h\cdot\tau^i_h)\kappa^i_h)_{i=1,2,3}
    \end{align}
	As $F_h=\Id$ for $h\equiv 0$, we have that $F_h$ is invertible as long as $\|h\|_{C^1}$ is small enough. With this we can rewrite $\eqref{EquationSystemForH}_1$ to an equation for $\partial_th$. Using \eqref{EquationFormulaKappaH0} and again \eqref{EqRelationNormalTangentialPart} to express the space derivatives of $\mu$ in terms of space derivatives of $h$, this shows that we indeed have a parabolic problem for $h$.\\
	For this system we can readily prove the following short time existence result.
	\begin{lemma}[Short time existence for \eqref{EquationSystemForH}]\label{LemmaSTEForH}\ \\
		Let $\Gamma_{\ast}^i, \gamma_{\ast}^i, \tilde{\alpha}_{\ast}^i, i=1,2,3$ fulfil Assumption \ref{RemarkAssumstionsGammaast} and $\alpha\in(0,1)$. Then there exists positive constants $\varepsilon_{STE}(\alpha),T_{STE}(\alpha)$ and $R_{STE}(\alpha)$ such that the following holds: for any initial data $h_0\in (C^{2+\alpha}([0,1]))^3$ fulfilling the compatibility conditions 
		\begin{align}\label{EquationAnalyticCompatibilityCondition}
			\begin{cases}
				h_0\text{ fulfills }\eqref{EquationSystemForH}_2-\eqref{EquationSystemForH}_4, \\
				\kappa_{h_0}^i=0 & \text{in }x=0, i=1,2,3,\\
				\sum_{i=1}^3\tilde{\alpha}^i_{\ast}F_{h_0}^{-1}(\varphi^{\circ}(\nu_{h_0}^i)(D^2\varphi^{\circ}(\nu_{h_0}^i)\tau_{h_0}^i\cdot\tau_{h_0}^i)\kappa_{h_0}^i)=0 & \text{in }x=1, 
			\end{cases}	
		\end{align} 
		and 
		$\|h_0\|_{(C^{2+\alpha}([0,1]))^3}\le\varepsilon_{STE}(\alpha)$ there exists $h\in (C^{\frac{2+\alpha}{2},2+\alpha}([0,T_{STE}(\alpha)]\times[0,1]))^3$ with 
		$$\|h\|_{(C^{\frac{2+\alpha}{2},2+\alpha}([0,T_{STE}(\alpha)]\times[0,1]))^3}\le R_{STE}(\alpha),$$
		 that solves \eqref{EquationSystemForH}. 
	\end{lemma}
\begin{rem}[Compatibility conditions]\label{RemarkCompatibilityConditionsH}\ \\
The compatibility conditions have the same motivation as in Remark \ref{RemarkMotivationSolution}v.). Hereby, $\eqref{EquationAnalyticCompatibilityCondition}_1$ postulates the boundary conditions for the system at $t=0$. $\eqref{EquationAnalyticCompatibilityCondition}_2$ resp. $\eqref{EquationAnalyticCompatibilityCondition}_3$ arise from differentiation of $\eqref{EquationSystemForH}_2$ resp. $\eqref{EquationSystemForH}_3$ in time using \eqref{EquationParabolicVersionOfPDE}. Note that $\eqref{EquationAnalyticCompatibilityCondition}_2$ is a simplified statement using the positivity of $(\varphi^{\circ}(\nu^i_h)(D^2\varphi^{\circ}(\nu^i_h)\tau^i_h\cdot\tau^i_h)$ and the invertibility of $F_{h_0}^{-1}$.
\end{rem}
\begin{rem}\label{rem-43}
With the same notation and assumptions as in  Lemma~\ref{LemmaExistenceRefFramGraphPara}: if $(\gamma^{i})_{i=1,2,3}$ is a geometrically admissible network according Definition~\ref{def-gan}, then $(h^{i})_{i=1,2,3}$ satisfies \eqref{EquationAnalyticCompatibilityCondition}.
\end{rem}
\begin{proof}
This is a direct consequence of Lemma~\ref{LemmaRelationNormalTangentialPart} and the geometric nature of the boundary conditions.
\end{proof}
	\begin{proof}
		The proof of Lemma~\ref{LemmaSTEForH} can be carried out analogously to \cite[Section 3]{KroenerNovagaPozziAnisoCurvFloImmeredNetworks}, using a fixed-point argument after showing an existence result for a suitable linear problem. Additionally, we want the uniform existence time $T_{STE}(\alpha)$ for small enough $h_0$ as stated above. But by carefully observing the contraction estimates one sees that directly. This is similar to the argumentation in \cite[Section 6]{GarckeGoessweinOnTheSurfaceDiffusionFlow}, especially the choice of $X^{\varepsilon}_{R,\delta}$ therein. As we will use the result on the linear problem again in Lemma \ref{LemmaParabolicSmoothigH}, we  give now  some comments on how to derive and solve it. Afterwards, we will sketch the proof of the fixed-point argument.
		
		\smallskip
		
		\underline{\textit{Linear problem:}} As opposed to \cite[Section 3]{KroenerNovagaPozziAnisoCurvFloImmeredNetworks} we have to study not any suitable linear problem but the proper linearization of the problem, i.e., the linearization in $h\equiv 0$ of the differential $\mathcal{A}_{NL}$ and boundary operator $\mathcal{B}_{NL}$ as functions in $h$, which are given by
	\begin{align}
		\mathcal{A}_{NL}: (C^{\frac{2+\alpha}{2},2+\alpha}([0,T]\times[0,1]))^3&\to (C^{\frac{\alpha}{2},\alpha}([0,T]\times[0,1]))^3,\\
		h&\mapsto (V^i_h-\varphi^{\circ}(\nu^i_h)(D^2\varphi^{\circ}(\nu^i_h)\tau^i_h\cdot\tau^i_h)\kappa^i_h)_{i=1,2,3},\\
		\mathcal{B}_{NL}: (C^{\frac{2+\alpha}{2},2+\alpha}([0,T]\times[0,1]))^3&\to (C^{\frac{2+\alpha}{2}}([0,T]))^4\times (C^{\frac{1+\alpha}{2}}([0,T]))^2,\label{EquationSystemForHBoundaryOperator}\\
		h&\mapsto (h(t,0), \sum_{i=1}^3\tilde{\alpha}_{\ast}^ih^i(t,1), \sum_{i=1}^3D\varphi^{\circ}(\nu^i_h(t,1)),
	\end{align}
	where recall that $(V_h^i)_{i=1,2,3}=F_h\partial_t h$.
Hence, for the linear analysis we compute 
\begin{align*}
	\underset{\varepsilon\to 0}{\lim}\frac{\mathcal{A}_{NL}(\varepsilon h)-\mathcal{A}_{NL}(0)}{\varepsilon},\qquad \qquad \underset{\varepsilon\to 0}{\lim}\frac{\mathcal{B}_{NL}(\varepsilon h)-\mathcal{B}_{NL}(0)}{\varepsilon},
\end{align*}
and consider the resulting PDE system. 
First notice that due to $\kappa_{\ast}^i\equiv 0, i=1,2,3,$ we have that
\begin{align*}
\frac{d}{d\varepsilon}\Big|_{\varepsilon=0}\mathcal{A}_{NL}(\varepsilon h) = \left(\partial_{t} h^{i} -
\varphi^{\circ}(\nu^i_{\ast})(D^2\varphi^{\circ}(\nu^i_{\ast})\tau^i_{\ast}\cdot\tau^i_{\ast})
\frac{d}{d\epsilon}\Big|_{\epsilon=0} \kappa^{i}_{\varepsilon h} \right)_{i=1,2,3}
\end{align*}
and, upon recalling \eqref{EquationFormulaKappaH0}, we compute
\begin{align*}
\frac{d}{d\epsilon}\Big|_{\epsilon=0} \kappa_{\varepsilon h}=\frac{h_{xx}}{|\gamma_{\ast}'|^{2}}-h_{x} \langle \frac{\gamma_{\ast}^{''}}{|\gamma_{\ast}'|^{3}}, \tau_{\ast} \rangle.
\end{align*}
For the boundary terms we observe that $\eqref{EquationSystemForH}_4$ is equivalent to the two conditions
		\begin{align*}
			\sum_{i=1}^3D\varphi^{\circ}(\nu_h^i(t,1))\cdot\nu_{\ast}^1&=0,\\
			\sum_{i=1}^3D\varphi^{\circ}(\nu_h^i(t,1))\cdot\nu_{\ast}^2&=0.
		\end{align*} 
	    Now we have that (recall Lemma \ref{lemma2.1}iii.))
	    \begin{align*}
	    \left(\frac{d}{d\varepsilon}D\varphi^{\circ}(\nu_{\varepsilon h}(t,1))\right)\big|_{\varepsilon=0}&=D^2\varphi^{\circ}(\nu_{\ast})\left(\frac{d}{d\varepsilon}(\nu_{\varepsilon h}(t,1))\right)\big|_{\varepsilon=0}=-D^2\varphi^{\circ}(\nu_{\ast})(h_x(t,1)\tau_{\ast})\\
	    &=-(D^2\varphi^{\circ}(\nu_{\ast})\tau_{\ast} \cdot\tau_{\ast})h_x(t,1) \tau_{\ast}.
	    \end{align*}
    For the second last identity recall that the normal is given as anticlockwise rotation of the tangent, thus
    \begin{align}\label{EquationFormulaNormalFunctionInH}
	\nu_{h}=\left(\frac{\gamma_{\ast}'+h_x\nu_{\ast}+\mu(h_x)\tau_{\ast}}{\sqrt{(h_x)^2+(|\gamma_{\ast}'|+\mu(h_x))^2}}\right)^{\bot}.
\end{align}
An easy calculation shows that
\begin{align*}
	\frac{d}{d\varepsilon}\nu_{\varepsilon h}\big|_{\varepsilon=0}=\left(\frac{|\gamma_{\ast}'|(h_x\nu_{\ast}+\mu(h_{x})\tau_{\ast})-\frac{1}{|\gamma_{\ast}'|}\mu(h_x)|\gamma_{\ast}'|\gamma_{\ast}'}{|\gamma_{\ast}'|^2}\right)^{\bot}.
\end{align*}
Now using that we chose $\gamma_{\ast}'$ such that $\gamma_{\ast}'(1)=\tau_{\ast}$ this shows
\begin{align*}
	\frac{d}{d\varepsilon}\nu_{\varepsilon h}(t,1)\big|_{\varepsilon=0}=-h_x(t,1)\tau_{\ast}.
\end{align*}
Collecting all previous calculations,  the linearization of \eqref{EquationSystemForH} in $h\equiv 0$  is given by:
		\begin{align}\label{EquationSystemForHLinearized}
			\begin{cases}
				\partial_th^i=\frac{1}{|(\gamma_{\ast}^i)'|^2}\varphi^{\circ}(\nu^i_{\ast})(D^2\varphi^{\circ}(\nu^i_{\ast})\tau^i_{\ast}\cdot\tau^i_{\ast})\partial_{xx}h^i+\mathcal{A}^i_{LOT}h^i& \text{on }[0,T]\times(0,1), i=1,2,3,\\
				h^i(t,0)=0 & \forall t\in[0,T], i=1,2,3,\\
				\sum_{i=1}^3\tilde{\alpha}^i_{\ast}h^i(t,1)=0 &\forall t\in[0,T],\\
				\partial_xh^2(t,1)=-\frac{\DZweiNuTauTau{3}{3}(\tau_{\ast}^3\cdot\nu_{\ast}^1)}{\DZweiNuTauTau{2}{2}(\tau_{\ast}^2\cdot\nu_{\ast}^1)}\partial_xh^3(t,1) &\forall t\in[0,T],\\
				\partial_xh^1(t,1)=-\frac{\DZweiNuTauTau{3}{3}(\tau_{\ast}^3\cdot\nu_{\ast}^2)}{\DZweiNuTauTau{1}{1}(\tau_{\ast}^1\cdot\nu_{\ast}^2)}\partial_xh^3(t,1) &\forall t\in[0,T],\\
				h(0,x)=h_0(x) & \forall x\in [0,1], i=1,2,3.
			\end{cases}
		\end{align}
		Hereby, we used the abbreviation \begin{align}\label{EquatioLinearSystemHLOT} \mathcal{A}^i_{LOT}h^i:=-\varphi^{\circ}(\nu^i_{\ast})(D^2\varphi^{\circ}(\nu^i_{\ast})\tau^i_{\ast}\cdot\tau^i_{\ast})\langle \frac{(\gamma_{\ast}^i)^{''}}{|(\gamma_{\ast}^i)'|^{3}}, \tau_{\ast} \rangle h^{i}_{x}.
		\end{align} 
		 Note that the coupled system in \eqref{EquationSystemForH} gets decoupled in the linearization as the tangential parts do not contribute there. This is very natural due to the geometric nature of the problem.
		 
		 \smallskip
		 
		 \underline{\textit{Lopatinskii-Shapiro conditions:}} The system  \eqref{EquationSystemForHLinearized} can be discussed using the theory in \cite{SolonnikovExistenceTheory}. Besides the parabolicity, which is guaranteed due to Lemma \ref{lemma2.1}i.), the most technical part in the application is the verification of the Lopatinskii-Shapiro conditions. Thus we want to explain how one can prove them. Hereby, we prefer to work with their ODE formulation which can be found, e.g., in \cite[(6), (7)]{LatushkinPruessSchnaubeltStableUnstableManifolds}.
		 A proof of the equivalence to the standard version can be found in \cite[Section I.2]{EidelmanZhitarashuParabolicBoundaryValueProblems}. At the triple junction these read as follows: show that for any $\lambda\in\mathbb{C}$ with $\mathrm{Re}(\lambda)>0$ the only continuous solution $(\varphi^i)_{i=1,2,3}$ with $\varphi^i: [0,\infty)\to\mathbb{C}, \varphi^i(y)\overset{y\to\infty}{\rightarrow}0, i=1,2,3$  of
		 \begin{align*}
		 	\lambda\varphi^i-\varphi^{\circ}(\nu_{\ast}^i)(D^2\varphi^{\circ}(\nu_{\ast}^i)\tau_{\ast}^i\cdot\tau_{\ast}^i)\varphi_{yy}^{i}&=0, & &y>0, i=1,2,3,\\
		 	\sum_{i=1}^3\tilde{\alpha}_{\ast}^i\varphi^i&=0, & &y=0,\\
		 	\varphi^2_y+\frac{\DZweiNuTauTau{3}{3}(\tau_{\ast}^3\cdot\nu_{\ast}^1)}{\DZweiNuTauTau{2}{2}(\tau_{\ast}^2\cdot\nu_{\ast}^1)}\varphi_y^3&=0, & &y=0, \\
		 	\varphi^1_y+\frac{\DZweiNuTauTau{3}{3}(\tau_{\ast}^3\cdot\nu_{\ast}^2)}{\DZweiNuTauTau{1}{1}(\tau_{\ast}^1\cdot\nu_{\ast}^2)}\varphi_y^3&=0, & &y=0,
		 \end{align*}
		 is the zero solution. Note that the $|\gamma_{\ast}'|^{-1}$-coefficient in the first line is equal to one due to \eqref{EquationConditionParaOfGammaAst}.
		 Observe that $\frac{\DZweiNuTauTau{3}{3}(\tau_{\ast}^3\cdot\nu_{\ast}^1)}{\DZweiNuTauTau{2}{2}(\tau_{\ast}^2\cdot\nu_{\ast}^1)}$ and $\frac{\DZweiNuTauTau{3}{3}(\tau_{\ast}^3\cdot\nu_{\ast}^2)}{\DZweiNuTauTau{1}{1}(\tau_{\ast}^1\cdot\nu_{\ast}^2)}$ have a negative sign due to Lemma~\ref{lemma2.1}i.) and Assumption \ref{RemarkAssumstionsGammaast}, which in particular implies that
		 	\begin{align}\label{EquationSignConditionProductNuTau}
		 		\frac{\tau_{\ast}^3\cdot\nu_{\ast}^1}{\tau_{\ast}^2\cdot\nu_{\ast}^1}<0\quad \wedge \quad \frac{\tau_{\ast}^3\cdot\nu_{\ast}^2}{\tau_{\ast}^1\cdot\nu_{\ast}^2}<0.
		 	\end{align}
		 	This is due to the fact that the angles between any two tangents $\tau_{\ast}^{i}$ and $\tau_{\ast}^{j}$ lies in  $(0,\pi)$.
				Then the normal $\nu_{\ast}^2$ resp. $\nu_{\ast}^1$ will have an 
			angle increased by $\frac{\pi}{2}$  with one of the $\tau_{\ast}^i$ and an 
			angle decreased by $\frac{\pi}{2}$ with the other. Thus, one contact angle will be in $(-\frac{\pi}{2},\frac{\pi}{2})$ and the other one in $(\frac{\pi}{2}, \frac{3\pi}{2})$. This then implies \eqref{EquationSignConditionProductNuTau}.
			  Now, testing the equation for
		 	\begin{itemize}
		 		\item $\varphi^1$ with $-\frac{(\tau_{\ast}^1\cdot\nu_{\ast}^2)}{\varphi^{\circ}(\nu_{\ast}^1)\DZweiNuTauTau{3}{3}(\tau_{\ast}^3\cdot\nu_{\ast}^2)}\tilde{\alpha}^1_{\ast}\bar{\varphi}^1$,
		 		\item $\varphi^2$  with $-\frac{(\tau_{\ast}^2\cdot\nu_{\ast}^1)}{\varphi^{\circ}(\nu_{\ast}^2)\DZweiNuTauTau{3}{3}(\tau_{\ast}^3\cdot\nu_{\ast}^1)}\tilde{\alpha}^2_{\ast}\bar{\varphi}^2$,
		 		\item $\varphi^3$ with $\frac{1}{\varphi^{\circ}(\nu^3_{\ast})\DZweiNuTauTau{3}{3}}\tilde{\alpha}^3_{\ast}\bar{\varphi}^3$,
		 	\end{itemize} 
	 
		 summing up, integrating over $\int_{0}^{\infty} \ldots dy$, and integrating by parts yields for the boundary term
		 	\begin{align*}
		 		&\left(\frac{\DZweiNuTauTau{1}{1}(\tau_{\ast}^1\cdot\nu_{\ast}^2)}{\DZweiNuTauTau{3}{3}(\tau_{\ast}^3\cdot\nu_{\ast}^2)}\tilde{\alpha}^1_{\ast}\varphi_y\bar{\varphi}^1+\frac{\DZweiNuTauTau{2}{2}(\tau_{\ast}^2\cdot\nu_{\ast}^1)}{\DZweiNuTauTau{3}{3}(\tau_{\ast}^3\cdot\nu_{\ast}^1)}\tilde{\alpha}^2_{\ast}\varphi^2_y\bar{\varphi}^2-\tilde{\alpha}_{\ast}^3\varphi_y^3\bar{\varphi}^3\right)\big|_{y=0}\\
		 		=&\left(-\varphi_y^3\tilde{\alpha}_{\ast}^1\bar{\varphi}^1-\varphi_y^3\tilde{\alpha}_{\ast}^2\bar{\varphi}^2-\varphi_y^3\tilde{\alpha}_{\ast}^3\bar{\varphi}^3\right)\big|_{y=0}=\left(-\varphi_y^3\sum_{i=1}^3\tilde{\alpha}_{\ast}^i\bar{\varphi}^i\right)\big|_{y=0}=0,
		 	\end{align*}
		 	where we used the boundary conditions at $y=0$.
		
			 Therefore we conclude that 
		 	\begin{align*}
		 		\int_0^{\infty}-&\lambda\frac{(\tau_{\ast}^1\cdot\nu_{\ast}^2)}{\varphi^{\circ}(\nu_{\ast}^1)\DZweiNuTauTau{3}{3}(\tau_{\ast}^3\cdot\nu_{\ast}^2)}\tilde{\alpha}^1_{\ast}|\varphi^1|^2-\lambda\frac{(\tau_{\ast}^2\cdot\nu_{\ast}^2)}{\varphi^{\circ}(\nu_{\ast}^3)\DZweiNuTauTau{3}{3}(\tau_{\ast}^3\cdot\nu_{\ast}^1)}\tilde{\alpha}^2_{\ast}|\varphi^2|^2\\
		 		+&\lambda\frac{1}{\varphi^{\circ}(\nu^3_{\ast})\DZweiNuTauTau{3}{3}}\tilde{\alpha}^3_{\ast}|\varphi^3|^2-\frac{\DZweiNuTauTau{1}{1}(\tau_{\ast}^1\cdot\nu_{\ast}^2)}{\DZweiNuTauTau{3}{3}(\tau_{\ast}^3\cdot\nu_{\ast}^2)}\tilde{\alpha}^1_{\ast}|\varphi^1_y|^2\\
		 		-&\frac{\DZweiNuTauTau{2}{2}(\tau_{\ast}^2\cdot\nu_{\ast}^1)}{\DZweiNuTauTau{3}{3}(\tau_{\ast}^3\cdot\nu_{\ast}^1)}\tilde{\alpha}^2_{\ast}|\varphi^2_y|^2+\tilde{\alpha}^3_{\ast}|\varphi^3_y|^2dy=0
		 	\end{align*}
		 
		 Again using Assumption \ref{RemarkAssumstionsGammaast}, Lemma \ref{lemma2.1}i.), \eqref{EquationSignConditionProductNuTau} and the positive sign of the $\tilde{\alpha}^i_{\ast}$ we see that all appearing terms have a real part with a positive sign and therefore we conclude that $\varphi^i\equiv 0, i=1,2,3.$
		  This shows the Lopatinskii-Shapiro conditions in $x=1$. In $x=0$ they are easier to verify as the boundary conditions are easier to handle. In total we have seen that the linear problem can be discussed using the theory in \cite{SolonnikovExistenceTheory}. 
		  
		  \smallskip
		  
		\underline{\textit{Fixed-point argument:}} With the knowledge about the solvability of the linear problem \eqref{EquationSystemForHLinearized} we are now able to construct the solution of \eqref{EquationSystemForH} as a fixed-point of a suitable functional. We will not do the detailed calculations here as they are quite voluminous and basically completely analogous to the ones done in \cite[Section 3.2]{KroenerNovagaPozziAnisoCurvFloImmeredNetworks}. We will only give the general structure and one exemplary calculation.
		
		For given $\tilde{h}\in (C^{\frac{2+\alpha}{2}, 2+\alpha}([0,T]\times [0,1]))^3$ we consider the mapping $\Lambda: \tilde{h}\to h$, where $h$ is the solution of
		\begin{align}\label{EquationSystemForFixedPointArgument}
			\begin{cases}
				\partial_th^i-\frac{1}{|(\gamma_{\ast}^i)'|^{2}}\varphi^{\circ}(\nu^i_{\ast})(D^2\varphi^{\circ}(\nu^i_{\ast})\tau^i_{\ast}\cdot\tau^i_{\ast})\partial_{xx}h^i-\mathcal{A}^i_{LOT}h^i=\mathfrak{f}^i(\tilde{h}) & \text{on }[0,T]\times(0,1), i=1,2,3,\\
				h^i(t,0)=0 & \forall t\in[0,T], i=1,2,3,\\
				\sum_{i=1}^3\tilde{\alpha}^i_{\ast}h^i(t,1)=0 &\forall t\in[0,T],\\
				\partial_xh^2(t,1)+\frac{\DZweiNuTauTau{3}{3}(\tau_{\ast}^3\cdot\nu_{\ast}^1)}{\DZweiNuTauTau{2}{2}(\tau_{\ast}^2\cdot\nu_{\ast}^1)}\partial_xh^3(t,1)=\mathfrak{b}^3(\tilde{h}) &\forall t\in[0,T],\\
				\partial_xh^1(t,1)+\frac{\DZweiNuTauTau{3}{3}(\tau_{\ast}^3\cdot\nu_{\ast}^2)}{\DZweiNuTauTau{1}{1}(\tau_{\ast}^1\cdot\nu_{\ast}^2)}\partial_xh^3(t,1)=\mathfrak{b}^4(\tilde{h}) &\forall t\in[0,T],\\
				h(0,x)=h_0(x) & \forall x\in [0,1], i=1,2,3.
			\end{cases}
		\end{align} 
	     with the right-hand sides 
	    \begin{align*}
	    	\mathfrak{f}^i(\tilde{h})&:=(\partial_t\tilde{h}^i-V_{\tilde{h}}^i)-\left(\varphiDZweiNuTauTau{i}{i}\frac{\partial_{xx}\tilde{h}^i}{|(\gamma_{\ast}^{i})'|^2}-\varphi^{\circ}(\nu^i_{\tilde{h}})(D^2\varphi^{\circ}(\nu^i_{\tilde{h}})\tau^i_{\tilde{h}}\cdot\tau^i_{\tilde{h}})\kappa^i_{\tilde{h}}\right)-\mathcal{A}^i_{LOT} \tilde{h}^i,\\
	    	\mathfrak{b}^3(\tilde{h})&:=\partial_x \tilde{h}^2(t,1)+\frac{\DZweiNuTauTau{3}{3}(\tau_{\ast}^3\cdot\nu_{\ast}^1)}{\DZweiNuTauTau{2}{2}(\tau_{\ast}^2\cdot\nu_{\ast}^1)}\partial_x \tilde{h}^3(t,1)-\sum_{i=1}^3D\varphi^{\circ}(\nu_{\tilde{h}}^i(t,1))\cdot\nu_{\ast}^1,\\
	    	\mathfrak{b}^4(\tilde{h})&:=\partial_x \tilde{h}^1(t,1)+\frac{\DZweiNuTauTau{3}{3}(\tau_{\ast}^3\cdot\nu_{\ast}^2)}{\DZweiNuTauTau{1}{1}(\tau_{\ast}^1\cdot\nu_{\ast}^2)}\partial_x \tilde{h}^3(t,1)-\sum_{i=1}^3D\varphi^{\circ}(\nu_{\tilde{h}}^i(t,1))\cdot\nu_{\ast}^2.
	    \end{align*} 
    Now we can use the Schauder estimates, which are part of the linear existence theory in \cite{SolonnikovExistenceTheory}, and obtain that
    \begin{align*}
    	\|h\|_{(C^{\frac{2+\alpha}{2}, 2+\alpha}([0,T]\times [0,1]))^3}\le C(&\|\mathfrak{f}(\tilde{h})\|_{(C^{\frac{\alpha}{2}, \alpha}([0,T]\times [0,1]))^3}+\|\mathfrak{b}^3(\tilde{h})\|_{C^{\frac{1+\alpha}{2}}([0,T])}+\|\mathfrak{b}^4(\tilde{h})\|_{C^{\frac{1+\alpha}{2}}([0,T])}\\
    	+&\|h_0\|_{(C^{2+\alpha}([0,1]))^3} ).
\end{align*}
    This shows that if we can prove contractivity of the mapping $\tilde{h}\to(\mathfrak{f}(\tilde{h}), \mathfrak{b}^3(\tilde{h}),\mathfrak{b}^4(\tilde{h}))$, $\Lambda$ will inherit this contractivity \footnote{Note that the difference $\Lambda(\tilde{h}_A)-\Lambda(\tilde{h}_B)$ for given $\tilde{h}_A, \tilde{h}_B$ solves \eqref{EquationSystemForFixedPointArgument} with zero initial data and right-hand sides $\mathfrak{f}(\tilde{h}_A)-\mathfrak{f}(\tilde{h}_B), \mathfrak{b}^3(\tilde{h}_A)-\mathfrak{b}^3(\tilde{h}_B)$ and $\mathfrak{b}^4(\tilde{h}_A)-\mathfrak{b}^4(\tilde{h}_B)$.}. More precisely, we will restrict the domain of $\Lambda$ to sets of the form
    \begin{align*}
    	B^{h_0}_{R,T}:=\{h\in (C^{\frac{2+\alpha}{2}, 2+\alpha}([0,T]\times [0,1]))^3\big| h\big|_{t=0}=h_0, \|h-h_0\|_{(C^{\frac{2+\alpha}{2}, 2+\alpha}([0,T]\times [0,1]))^3}\le R \}
\end{align*}
for $h_0\in (C^{2+\alpha}([0,1]))^3$. Then we are able to show that for arbitrary $R$ we can choose $T$ and $\varepsilon:=\|h_0\|_{(C^{2+\alpha}([0,1]))^3}$ small enough such that $\Lambda$ is a $\frac{1}{2}$-contraction on this set. Then by choosing $R$ large enough we will also get that $\Lambda$ is a self-mapping. The resulting values of $\varepsilon, R, T$ then correspond to the $\varepsilon_{STE}(\alpha), R_{STE}(\alpha), T_{STE}(\alpha)$ from the statement of the result.

The main difficulty is to understand how the contraction estimates can be derived.
     We will discuss the contraction estimates for the first summand in 
   $  \mathfrak{f}(\tilde{h}_A)-\mathfrak{f}(\tilde{h}_B)$ to give an idea of the general strategy. For simplicity we will state the function spaces  without the domains where they are defined on. 
Note that by the choice of domain $B^{h_0}_{R,T}$ all elements $\tilde{h}$ in this ball 
   will have a norm less than $R+\varepsilon$, where $\varepsilon$ denotes the norm of $h_0$. Now, for $h_A, h_B \in B^{h_0}_{R,T}$  
   we have that (recall \eqref{EquationNormalVelocityInH})
    \begin{align*}
    	(\partial_t h_A-V_{h_A})-(\partial_t h_B -V_{h_B})&=\underbrace{\partial_t h_A(1-\nu_{h_A}\cdot\nu_{\ast})-\partial_th_B(1-\nu_{h_B}\cdot\nu_{\ast})}_{(I)}\\
    	&\quad - (\underbrace{\partial_t\mu(h_A)(\nu_{h_A}\cdot\tau_{\ast})-\partial_t\mu(h_B)(\nu_{h_B}\cdot\tau_{\ast})}_{(II)}).
\end{align*}
    Now, we can write $(I)$ as
    \begin{align*}
    (I)=\underbrace{\partial_t h_A(\nu_{h_B}\cdot\nu_{\ast}-\nu_{h_A}\cdot\nu_{\ast})}_{(A)}+\underbrace{(\partial_th_A-\partial_t h_B)(1-\nu_{h_B}\cdot\nu_{\ast})}_{(B)}.
\end{align*}
    For $(B)$ we have that
    \begin{align*}
    	\|(B)\|_{C^{\frac{\alpha}{2},\alpha}}&\le \|\partial_th_A-\partial_t h_B\|_{C^{\frac{\alpha}{2},\alpha}}\|(1-\nu_{h_B}\cdot\nu_{\ast})\|_{C^{\frac{\alpha}{2},\alpha}}\\
    	&\le \|h_A-h_B\|_{C^{\frac{2+\alpha}{2}, 2+\alpha}}(\varepsilon+T^{\beta}\|1-\nu_{h_B}\cdot\nu_{\ast}\|_{C^{\frac{1+\alpha}{2}, 1+\alpha}})\\
    	&\le C\|h_A-h_B\|_{C^{\frac{2+\alpha}{2}, 2+\alpha}}(\varepsilon+T^{\beta}\|h_B\|_{C^{\frac{2+\alpha}{2}, 2+\alpha}})\\
    	&\le C((R+\varepsilon)T^{\beta}+\varepsilon)\|h_A-h_B\|_{C^{\frac{2+\alpha}{2}, 2+\alpha}}.
\end{align*} 
Hereby, we used product estimates for parabolic H\"older spaces, contractivity of lower order terms ($\beta$ is a suitable exponent in $(0,1)$) and Lipschitz continuity of the function $h\mapsto \nu_{h_B}\cdot\nu_{\ast}$, which maps the zero function to $1$. All these results can be found for instance in \cite[Appendix A]{KroenerNovagaPozziAnisoCurvFloImmeredNetworks} and \cite[Section 2]{GarckeGoessweinOnTheSurfaceDiffusionFlow}. Note that for the contractivity of the lower order terms we need \cite[Lemma 2.7]{GarckeGoessweinOnTheSurfaceDiffusionFlow}, which allows for initial data unequal to zero. Therefore we see that for any $R$ we can choose $T$ and $\varepsilon$ small enough such that we have a $\frac{1}{2}$-contraction. The estimate for $(A)$ is similar. Now, for $(II)$ we have that
\begin{align*}
	(II)=\underbrace{\partial_t\mu(h_A)(\nu_{h_A}\cdot\tau_{\ast}-\nu_{h_B}\cdot\tau_{\ast})}_{(A)}+\underbrace{(\partial_t\mu(h_A)-\partial_t\mu(h_B))(\nu_{h_B}\cdot\tau_{\ast})}_{(B)}.
\end{align*}
    For $(B)$ we have that
    \begin{align*}
    	\|(B)\|_{C^{\frac{\alpha}{2},\alpha}}&\le \|\partial_t\mu(h_A)-\partial_t\mu(h_B)\|_{C^{\frac{\alpha}{2},\alpha}}\|\nu_{h_B}\cdot\tau_{\ast}\|_{C^{\frac{\alpha}{2},\alpha}}\\
    	&\le C \|\mu(h_A-h_B)\|_{C^{\frac{2+\alpha}{2},2+\alpha}}(\varepsilon+T^{\beta}\|\nu_{h_B}\cdot\tau_{\ast}\|_{C^{\frac{1+\alpha}{2}, 1+\alpha}})\\
    	&\le C((R+\varepsilon)T^{\beta}+\varepsilon)\|h_A-h_B\|_{C^{\frac{2+\alpha}{2}, 2+\alpha}}.
\end{align*}
Again, we see that for any $R$ we can choose $T$ and $\varepsilon$ small enough such that we have a $\frac{1}{2}$-contraction. The estimate for $(A)$ is similar. In particular, observe that all estimates do not depend directly on $h_0$ but only its norm. Thus, we get the sought uniform estimate for $h_0$ small enough.
 For the boundary conditions the discussion is slightly   
 more complicated as we do not have a quasilinear structure there. But one can argue as for (118) in \cite[Lemma 6.5]{GarckeGoessweinOnTheSurfaceDiffusionFlow}.  
	\end{proof}
For the stability analysis we will additionally need that our solution has higher regularity and - even more importantly - that we have a control for these higher norms. For this we use ideas from \cite[Section 4]{GarckeGoessweinNonlinearStabilityDoubleBubbleSDF}. Mainly, we will apply Angenent's parameter trick. This would allow us to get space regularity of any order away from $t=0$. But the required bootstrap argument would be unnecessary complicated for what we need later, therefore we restrict to estimates for the $C^{3+\alpha}$-norm in space.

	\begin{lemma}[Smoothing estimates for $h$] \label{LemmaParabolicSmoothigH}\ \\
	Let $\Gamma_{\ast}^i, \gamma_{\ast}^i, \tilde{\alpha}_{\ast}^i, i=1,2,3$ fulfil Assumption \ref{RemarkAssumstionsGammaast}, $\alpha\in(0,1)$. Moreover for  $h_{0} \in C^{2+\alpha}$ satisfying
	the compatibility conditions \eqref{EquationAnalyticCompatibilityCondition}, let $h\in (C^{\frac{2+\alpha}{2},2+\alpha}([0,T]\times[0,1]))^3$ solve \eqref{EquationSystemForH} for some $T>0$ with initial data $h_0$. Additionally, let $R(\alpha)>0$ be small enough  
	such that  
		\begin{align} \label{EquationAssumptionBoundH}
			\|h\|_{(C^{\frac{2+\alpha}{2},2+\alpha}([0,T]\times[0,1]))^3}\le R(\alpha).
		\end{align}
		Then there 
		exists for every $T_k\in (0,T)$ close to $T$ a $\varepsilon_k>0$ such that, if
		 \begin{align}
			\|h_0\|_{(C^{2+\alpha}([0,1]))^3}\le \varepsilon_k,
		\end{align} 
		then
		 for any $t_k\in (0,T_k)$  there exists $C(t_k, T_k)>0$ such that for all $t\in[t_k,T_k]$ 
		we have  $h(t,\cdot)\in (C^{3+\alpha}([0,1]))^3$ and 
		\begin{align}\label{EquationSmoothingForH}
			\|h(t,\cdot)\|_{(C^{3+\alpha}( [0,1]))^3}\le C(R(\alpha), C(t_k, T_k))(1+\|h_0\|_{(C^{2+\alpha}([0,1]))^3}).
		\end{align}
	\end{lemma}
\begin{rem}\ 
\begin{itemize}
\item[(i)] The precise restriction  on the constant $R(\alpha)>0$  will become clear in Step 3 in the proof.
\item[(ii)]	As we will see in the proof of the lemma the constants $C(t_k, T_k)$ and $C(R(\alpha), C(t_k, T_k))$ blow up like $t_k^{-1}$ as $t_{k}$ approaches zero.
\end{itemize}
\end{rem}

	\begin{proof}
		The proof splits into three separate steps. The first one is of technical nature and is only  needed for the precise application of the parameter trick, which is carried out in the second step. As a consequence we derive higher regularity for the time derivative. Finally, this is used in the last step to also get higher space regularity. Before we start with the proof itself, we want to give some ideas for the parameter trick, point out  the problems that arise when working with non-linear boundary conditions and our workaround for these issues.
		
			\textit{Notation:} We will denote by $\mathcal{A}_{lin}$ the elliptic operator from $\eqref{EquationSystemForHLinearized}_1$, i.e., 
			\begin{align*}
				\mathcal{A}_{lin}h:=\left( \frac{1}{|\gamma_{\ast}'|^2}\varphi^{\circ}(\nu^i_{\ast})(D^2\varphi^{\circ}(\nu^i_{\ast})\tau^i_{\ast}\cdot\tau^i_{\ast})\partial_{xx}h^i-\varphi^{\circ}(\nu^i_{\ast})(D^2\varphi^{\circ}(\nu^i_{\ast})\tau^i_{\ast}\cdot\tau^i_{\ast})\langle \frac{(\gamma_{\ast}^{i})^{''}}{|(\gamma_{\ast}^{i})'|^{3}}, \tau_{\ast} \rangle h^{i}_{x}\right)_{i=1,2,3},
			\end{align*}
		 by $\mathcal{B}_{lin}$ the boundary operator due to $\eqref{EquationSystemForHLinearized}_2\eqref{EquationSystemForHLinearized}_5$, i.e.,
			\begin{align*}
				\mathcal{B}_{lin}h:=(h\big|_{x=0}, \sum_{i=1}^3\tilde{\alpha}^i_{\ast}h^i\big|_{x=1}, &(\partial_xh^2+\frac{\DZweiNuTauTau{3}{3}(\tau_{\ast}^3\cdot\nu_{\ast}^1)}{\DZweiNuTauTau{2}{2}(\tau_{\ast}^2\cdot\nu_{\ast}^1)}\partial_xh^3)\big|_{x=1},\\ &(\partial_xh^1+\frac{\DZweiNuTauTau{3}{3}(\tau_{\ast}^3\cdot\nu_{\ast}^2)}{\DZweiNuTauTau{1}{1}(\tau_{\ast}^1\cdot\nu_{\ast}^2)}\partial_xh^3)\big|_{x=1}),
			\end{align*}
			the right-hand side for a general linear parabolic differential operator by $\mathfrak{f}$, and the right-hand sides of the boundary conditions for a general differential system by $\mathfrak{b}=(\mathfrak{b}^1,\mathfrak{b}^2,\mathfrak{b}^3,\mathfrak{b}^4)\in X_{comp, T}$ with
			\begin{align*}
				X_{comp, T}:=(C^{\frac{2+\alpha}{2}}([0,T]))^3 \times C^{\frac{2+\alpha}{2}}([0,T])\times C^{\frac{1+\alpha}{2}}([0,T]) \times C^{\frac{1+\alpha}{2}}([0,T]) .
			\end{align*}
		
			In other words we will be considering the PDE-system:	
\begin{align}\label{EquationSystemForHLinearized23}
			\begin{cases}
				\partial_t h = \mathcal{A}_{lin}h + \mathfrak{f} & \text{on }[0,T]\times(0,1), \\
				\mathcal{B}_{lin}h =\mathfrak{b} &\forall t\in[0,T],\\ 
				h(0,x)=h_0(x) & \forall x\in [0,1].
			\end{cases}
		\end{align}
\textit{Preliminary considerations and main ideas:} Denote by $h_{h_0}$ the solution from Lemma \ref{LemmaSTEForH} for some initial data $h_0$. Additionally, for $\lambda\in (1-\delta, 1+\delta)$ for $\delta>0$ small enough, denote  by $h_{\lambda, h_0}$ the  time scaled version of $h_{h_0}$, i.e.,
			\begin{align*}
				h_{\lambda, h_0}(t,x):=h_{h_0}(\lambda t,x) \qquad t \in \left[0,\frac{T}{\lambda} \right],\,\, x \in [0,1].
			\end{align*}
	Using that $\partial_{t} h_{\lambda, h_0}(t,x)= \lambda \partial_{t} h_{h_0}(\lambda t,x)$ and $\partial_{x} h_{\lambda, h_0}(t,x)=  \partial_{x} h_{h_0}(\lambda t,x)$ 	one sees that $h_{\lambda, h_0}$ solves a time scaled version of \eqref{EquationSystemForH} with the same boundary and initial conditions, i.e. $h_{\lambda, h_0}$ solves the $\lambda$-problem $$ V_h-\lambda\varphi^{\circ}(\nu_h)(D^2\varphi^{\circ}(\nu_h)\tau_h\cdot\tau_h)$$
	together with initial and boundary conditions $h_{0}$, $\mathcal{B}_{NL}(h_{\lambda, h_0}) =0$ on the time interval $[0, \frac{T}{1+\delta} ]$. 
	 The idea is now to show smoothness in $\lambda$ of the map $h_{\lambda, h_0}$. This will then imply smoothness in $t$  (away from zero) of the original $h_{h_0}$,  since
	$$\partial_{\lambda} h_{\lambda, h_0}(t,x) = \partial_{\lambda}(h_{h_0}(\lambda t,x)) = t \partial_{t} h_{h_0}(\lambda t,x).$$

 The desired smoothness of  $ h_{\lambda, h_0}$ is obtained by application of
 the implicit functions theorem (cf. Theorem \ref{TheoremImplicitFunction}): a natural candidate for a suitable functional for this approach is given by
 \begin{align*}
 F: (1-\delta, 1+\delta) \times (C^{2+\alpha}([0,1]))^3\times (C^{\frac{2+\alpha}{2}, 2+\alpha}([0, \frac{T}{1+\delta}]\times[0,1]))^3 \to M 
 \end{align*}
 with
 $$ M \subset(C^{2+\alpha}([0,1]))^3\times (C^{\frac{\alpha}{2}, \alpha}([0, \frac{T}{1+\delta}]\times[0,1]))^3 \times
 X_{comp, \frac{T}{1+\delta}}
 $$
 and 
		\begin{align*}
			F: (\lambda, h_0, h)\mapsto \left(h\big|_{t=0}-h_0, V_h-\lambda\varphi^{\circ}(\nu_h)(D^2\varphi^{\circ}(\nu_h)\tau_h\cdot\tau_h)\kappa_{h}, \mathcal{B}_{NL}(h) \right).
		\end{align*} 
		where $\mathcal{B}_{NL}(h)$ is given by \eqref{EquationSystemForHBoundaryOperator}.
		Thus, $\mathcal{B}_{NL}$ is just the boundary operator of \eqref{EquationSystemForH}.
	Clearly, for any $\lambda, h_0$ small enough a root of this  map is given by $h_{\lambda, h_0}$, 
	 i.e. $F(\lambda, h_{0}, h_{\lambda, h_{0}})=0$ (recall Lemma~\ref{LemmaSTEForH}). In particular $F(1,0,0)=0$, i.e. the functional vanishes on the stationary configuration $\Gamma_{\ast}$ (which corresponds to $h=0$.)
	If we are able to apply Theorem \ref{TheoremImplicitFunction} (by showing the bijectivity of $\partial_3 F(1,0,0)$ and smoothness of $F$ around $(1,0,0)$), we get a unique solution $ (\lambda, h_{0}) \to h((\lambda, h_{0}))$(which then has to coincide with $h_{\lambda, h_0}$!) which is smooth in $\lambda$. The crucial part for this now is to check bijectivity of $\partial_3 F(1,0,0)$. This yields the map
	\begin{align*}
		h\mapsto (h\big|_{t=0}, \partial_t h-\mathcal{A}_{lin}h, \mathcal{B}_{lin}h).
	\end{align*}
This looks promising as the right-hand side is precisely the linear problem \eqref{EquationSystemForHLinearized} for which we already have a solution theory. The problem now is that the range of $\partial_3 F(1,0,0)$ will only contain (linear) compatible data while $F$ does not. So, to make things work we have to  force $F$ to take the right range. Our method for this is to project $F$ onto the right space. The construction of this projections is the subject of the next step.\\
		
		\textit{Step 1: } We consider the space $\mathbb{D}_{lin, \frac{T}{1+\delta}}$ of all triplets $(h_0, \mathfrak{f}, \mathfrak{b})$ of initial data, and right-hand sides in \eqref{EquationSystemForHLinearized23} fulfilling the necessary compatibility conditions for the system to be solvable.

Precisely, we have that
		\begin{align*}
			\mathbb{D}_{lin, \frac{T}{1+\delta}}&=\{(h_0,\mathfrak{f},\mathfrak{b})\in (C^{2+\alpha}([0,1]))^3\times C^{\frac{\alpha}{2}, \alpha}([0,\frac{T}{1+\delta}]\times[0,1])^3\times X_{comp, \frac{T}{1+\delta}} \, : \quad  \mathcal{B}_{lin}(h_0)=\mathfrak{b}\big|_{t=0},\\
			&\mathcal{A}_{lin}(h_0)\big|_{x=0}+\mathfrak{f}\big|_{x=0, t=0}=\partial_t \mathfrak{b}^1\big|_{t=0}, \quad\sum_{i=1}^3\tilde{\alpha}^i_{\ast}(\mathcal{A}^i_{lin}(h_0)\big|_{x=1}+\mathfrak{f}^i\big|_{x=1, t=0}) = \partial_t\mathfrak{b}^2\big|_{t=0} \}\\
			&=\ker(K_{\frac{T}{1+\delta}}),
		\end{align*}
	
where $K_{\frac{T}{1+\delta}}: (C^{2+\alpha}([0,1]))^3\times C^{\frac{\alpha}{2}, \alpha}([0,\frac{T}{1+\delta}]\times[0,1])^3\times X_{comp, \frac{T}{1+\delta}}\to\R^{10}$ is defined by
\begin{align*}
	K(h_0,\mathfrak{f},\mathfrak{b})=\begin{pmatrix}
		\mathcal{B}_{lin}(h_0)-\mathfrak{b}\big|_{t=0} \\ \mathcal{A}_{lin}(h_0)\big|_{x=0}+\mathfrak{f}\big|_{x=0, t=0}-\partial_t\mathfrak{b}^1\big|_{t=0} \\ 
		\sum_{i=1}^3\tilde{\alpha}^i_{\ast}(\mathcal{A}^i_{lin}(h_0)\big|_{x=1}+\mathfrak{f}^i\big|_{x=1, t=0}) - \partial_t\mathfrak{b}^2\big|_{t=0}
	\end{pmatrix}.
\end{align*} 
Note that $K_{\frac{T}{1+\delta}}$ is a linear operator and thus $\mathbb{D}_{lin, \frac{T}{1+\delta}}$ is a closed subspace. Now, to be able to find a continuous projection onto $\mathbb{D}_{lin}$ it is sufficient - see \cite[P. 34-36]{GrothendieckTopologicalVectorSpaces} - to find a closed complementary space $Z_{\frac{T}{1+\delta}}$ in $(C^{2+\alpha}([0,1]))^3\times C^{\frac{\alpha}{2}, \alpha}([0,\frac{T}{1+\delta}]\times[0,1])^3\times X_{comp, \frac{T}{1+\delta}}$,
i.e.
\begin{align}\label{EquationDirectSumCondition} Z_{\frac{T}{1+\delta}} \oplus \mathbb{D}_{lin, \frac{T}{1+\delta}} =(C^{2+\alpha}([0,1]))^3\times C^{\frac{\alpha}{2}, \alpha}([0,\frac{T}{1+\delta}]\times[0,1])^3\times X_{comp, \frac{T}{1+\delta}}.
\end{align}
 To find such a $Z_{\frac{T}{1+\delta}}$ let $d:=\dim(\im(K))$ and 
 $$e_1,..., e_d\in (C^{2+\alpha}([0,1]))^3\times C^{\frac{\alpha}{2}, \alpha}([0,\frac{T}{1+\delta}]\times[0,1])^3\times X_{comp, \frac{T}{1+\delta}}$$
  be maps such that $\langle K_{\frac{T}{1+\delta}}(e_1),...,K_{\frac{T}{1+\delta}}(e_d)\rangle=\im(K_{\frac{T}{1+\delta}})$. Then set 

  $$Z_{\frac{T}{1+\delta}}:=span \{ e_1,...,e_d\} \subset (C^{2+\alpha}([0,1]))^3\times C^{\frac{\alpha}{2}, \alpha}([0,\frac{T}{1+\delta}]\times[0,1])^3\times X_{comp, \frac{T}{1+\delta}}.$$
  
   As $\dim(Z_{\frac{T}{1+\delta}})\le 10$, we see that $Z_{\frac{T}{1+\delta}}$ is a closed subspace. Additionally, we have that $Z_{\frac{T}{1+\delta}}\cap \mathbb{D}_{lin, \frac{T}{1+\delta}}=\{0\}$ by construction of $Z_{\frac{T}{1+\delta}}$. Finally, let $v\in (C^{2+\alpha}([0,1]))^3\times C^{\frac{\alpha}{2}, \alpha}([0,\frac{T}{1+\delta}]\times[0,1])^3\times X_{comp, \frac{T}{1+\delta}}$. Then there exists a unique $\tilde{v}\in Z_{\frac{T}{1+\delta}}$ such that $K_{\frac{T}{1+\delta}}(v)=K_{\frac{T}{1+\delta}}(\tilde{v})$. Furthermore, $K_{\frac{T}{1+\delta}}(v-\tilde{v})=0$, thus $v-\tilde{v}\in \mathbb{D}_{lin, \frac{T}{1+\delta}}$. This shows that with this choice of $Z_{\frac{T}{1+\delta}}$ condition \eqref{EquationDirectSumCondition} is fulfilled.  
 
 Consequently, we have a continuous projection $P_{lin, \frac{T}{1+\delta}}$ onto $\mathbb{D}_{lin, \frac{T}{1+\delta}}$:

 $$ P_{lin, \frac{T}{1+\delta}}: (C^{2+\alpha}([0,1]))^3\times C^{\frac{\alpha}{2}, \alpha}([0,\frac{T}{1+\delta}]\times[0,1])^3\times X_{comp, \frac{T}{1+\delta}} \to \mathbb{D}_{lin, \frac{T}{1+\delta}}.$$

		\textit{Step 2:} We are now able to perform the parameter trick itself. For this we consider for $\delta$ small enough the mapping
		\begin{align*}
			F: (1-\delta, 1+\delta)&\times (C^{2+\alpha}([0,1]))^3\times (C^{\frac{2+\alpha}{2}, 2+\alpha}([0,\frac{T}{1+\delta}]\times[0,1]))^3\to \mathbb{D}_{lin, \frac{T}{1+\delta}},\\
			(\lambda, h_0, h)&\mapsto P_{lin, \frac{T}{1+\delta}}((h\big|_{t=0}-h_0, V_h-\lambda\varphi^{\circ}(\nu_h)(D^2\varphi^{\circ}(\nu_h)\tau_h\cdot\tau_h)\kappa_h ,\mathcal{B}_{NL}(h))).
		\end{align*}

		On this mapping, which is analytic due to similar arguments as performed in Section \ref{SectionParaAndLoja}, we apply now the implicit function theorem. For $\partial_3F(1,0,0)$ we get that
		\begin{align*}
			\partial_3F(1,0,0)h=(h\big|_{t=0}, \mathcal{A}_{lin}h, \mathcal{B}_{lin}h).
		\end{align*}
	Note that the projection $P_{lin, \frac{T}{1+\delta}}$ does not appear as by the essence of compatibility conditions we already have that $h\big|_{t=0}$ is compatible with the inhomogeneities $\mathcal{A}_{lin}h$ and $\mathcal{B}_{lin}h$. Now the invertibility of $\partial_3F(1,0,0)$ follows directly from the linear existence theory discussed in the proof of Lemma \ref{LemmaSTEForH} and the choice of $\mathbb{D}_{lin, \frac{T}{1+\delta}}$. Therefore, $\partial_3F(1,0,0)$ is bijective. So we can apply the implicit functions theorem (see Theorem \ref{TheoremImplicitFunction}) to get a unique solution $h_{PT}(\lambda, h_0)$ of
		\begin{align}\label{EquationRootProblem}
			F(\lambda, h_0, h)=0,
		\end{align}
		which is smooth in $\lambda$ and $h_0$ for $\lambda$ sufficiently close to one and $\|h_{0}\|_{C^{2+\alpha}([0,1])^{3}} < \varepsilon_{k}$. On the other hand, we easily see that the scaled solution $$h_{\lambda, h_0}:=h(\lambda\cdot,\cdot)$$ with $h$ as in the prerequisites of the lemma is also a solution of \eqref{EquationRootProblem}. Note that the projection $P_{lin,\frac{T}{1+\delta}}$ is not a problem here as 
	\begin{align*}
		F(\lambda, h_0, h_{\lambda, h_0})= P_{lin, \frac{T}{1+\delta}}((0,0,0))=(0,0,0).
\end{align*}  
By uniqueness of the roots of $F$ we conclude that 
		\begin{align*} 	
		h_{PT}(\lambda, h_0) \equiv h_{\lambda, h_0}.
		\end{align*}	
		In particular, we see that $h_{\lambda,h_0}$ is smooth as function in $\lambda$ and that 
		\begin{align*}
			\partial_th(\lambda t,x)=\frac{1}{t}\partial_{\lambda}h_{PT}(\lambda,h_0)(x,t).
		\end{align*} 
		Consequently, we see that $\partial_th\in (C^{\frac{2+\alpha}{2}, 2+\alpha}([t_k,\frac{T}{1+\delta}]\times[0,1]))^3$ for all $t_k\in (0,\frac{T}{1+\delta})$. 
		Furthermore, using that $h_{PT}(\lambda,0)=0$ for any $\lambda$,
		for $h_0$ small enough we have that
	    \begin{align*}
	    \|\partial_{\lambda}h_{PT}(1,h_0)\|_{(C^{\frac{2+\alpha}{2}, 2+\alpha}([0, \frac{T}{1+\delta}]\times[0,1]))^3}&\le C\int_0^1\|\partial_2\partial_{\lambda}h_{PT}(1,sh_0)h_0\|_{(C^{\frac{2+\alpha}{2}, 2+\alpha}([0, \frac{T}{1+\delta}]\times[0,1]))^3}ds\\
	    &\le C\int_0^1\|h_0\|_{(C^{2+\alpha}([0,1]))^3}ds\le C\|h_0\|_{(C^{2+\alpha}([0,1]))^3}.
	    \end{align*}
        Hereby, we used analyticity of $h_{PT}$ in both $\lambda$ and $h_0$ in a small neighborhood of $(1,0)$.
         This shows that
		\begin{align}\label{EquationBoundForPartialTH}
			\|\partial_t h\|_{(C^{\frac{2+\alpha}{2}, 2+\alpha}([t_k, \frac{T}{1+\delta}], [0,1]))^3} \le C(t_k, \frac{T}{1+\delta})\|h_0\|_{(C^{2+\alpha}([0,1]))^3}.
		\end{align}
	
		\textit{Step 3:} The higher space regularity of $\partial_t h$  gives us the higher space regularity of $h$ as we can use $\eqref{EquationSystemForH}_1$ and \eqref{EquationFormulaKappaH0} to see that
		\begin{align*}
		\kappa_{h}^{i} &=\frac{(Ah_{xx})^{i} -h^{i}_{x} \langle (\gamma^{i}_{\ast})'', \tau^{i}_{\ast}\rangle}{[(h^{i}_{x})^{2} + (|(\gamma_{\ast}^{i})'| + \mu^{i}(h_{x}))^{2}]^{3/2}} \qquad \qquad \text{and }\\
		\kappa_{h}^{i}&=\frac{(F_{h} \partial_{t} h)^{i}}{\varphi^{\circ}(\nu_{h}^{i})(D^2(\varphi^{\circ})(\nu_h^{i})\tau_h^{i}\cdot\tau_h^{i})},
		\end{align*}
		with
\begin{align*}
	A:=\mathrm{diag}\left(|(\gamma^1_{\ast})'|+\mu^1(h_x), |(\gamma^2_{\ast})'|+\mu^2(h_x), |(\gamma^3_{\ast})'|+\mu^3(h_x) \right)-\mathrm{diag}(h^1_x, h^2_x, h^3_x)\mathcal{I}.
\end{align*}
	Note that for $h\equiv 0$ we have that $A=\mathrm{diag}(|(\gamma_{\ast}^1)'|, |(\gamma_{\ast}^2)'|, |(\gamma_{\ast}^3)'|)$. This is clearly invertible at every point. Consequently, $A$ is invertible as long as $h_x$ is small enough. This can be guaranteed for $R(\alpha)$ small enough. Additionally, all entries of $A$ are bounded in $C^{1+\alpha}$ due to \eqref{EquationAssumptionBoundH}. The same holds then for the inverse of $A$ due to Cramer's rule. 
With this and the the higher space regularity of $\partial_t h$ we conclude that $h_{xx}(t, \cdot)\in C^{1+\alpha}$ away from $t=0$. Finally, \eqref{EquationSmoothingForH} follows from \eqref{EquationBoundForPartialTH} and the bound \eqref{EquationAssumptionBoundH}.
\end{proof}
It remains to show the continuous dependency of our solution on the initial data.
To that end we use similar ideas as in the previous lemma.
\begin{lemma}[Continuous Dependency on initial data]\label{LemmaContinuousDependency}\ \\
		Let $\Gamma_{\ast}^i, \gamma_{\ast}^i, \tilde{\alpha}_{\ast}^i, i=1,2,3$ fulfil Assumption \ref{RemarkAssumstionsGammaast}, $\alpha\in(0,1)$. Moreover for  initial data $h_{0} \in C^{2+\alpha}$ satisfying
	the compatibility conditions \eqref{EquationAnalyticCompatibilityCondition}, let $h\in (C^{\frac{2+\alpha}{2},2+\alpha}([0,T]\times[0,1]))^3$ solve \eqref{EquationSystemForH} for some $T>0$. Then, there exists an $\varepsilon_k>0$ such that the mapping
	\begin{align*}
		(C^{2+\alpha}([0,1]))^3\supset B_{\varepsilon_k}(0)&\to\R \\
		h_0&\mapsto \|h\|_{(C^{\frac{2+\alpha}{2},2+\alpha}([0,T]\times[0,1]))^3}
	\end{align*}
is continuous.
\end{lemma}
\begin{proof}
	We basically copy Step 2 from Lemma \ref{LemmaParabolicSmoothigH} but without the $\lambda$-parameter in the definition of $F$ (in other words we choose $\lambda \equiv 1$). Thus, we do not have to modify the time interval and get the sought result from another application of Theorem \ref{TheoremImplicitFunction}.
\end{proof}

\begin{rem}[On the condition \eqref{EquationAssumptionBoundH}]\label{RemarkOnAssumptionBoundH}\ \\
   As a consequence of Lemma \ref{LemmaContinuousDependency} we see that the smallness of the bound \eqref{EquationAssumptionBoundH} can indeed guaranteed for a possible smaller choice of $\varepsilon_k$. In the following, we will assume that $\varepsilon_{k}$ is chosen in such a way.
\end{rem}

\section{Stability analysis}\label{SectionStabAna}

The version of the LSI we proved in Section \ref{SectionLoja}  is in itself  not very useful for the stability analysis. The reason for this is that Theorem \ref{TheoremLSIanalytic} does not consider the geometrical structure of our flow. But the argument we want to do now relies strongly on the geometric gradient flow we have. Therefore, we will modify Theorem \ref{TheoremLSIanalytic} for solutions of \eqref{EquationGeometricProblem}. Before we do so we want to specify resp. recall some notations for this section. We will denote by $u(t,x)\in (C^{\frac{2+\alpha}{2},2+\alpha}([0,T]\times[0,1],\R^2))^3$ any geometric solution of \eqref{EquationGeometricProblem} with initial values $u_0(x)\in (C^{2+\alpha}([0,1],\R^2))^3$. The corresponding normal parametrization when Lemma \ref{LemmaExistenceRefFramGraphPara} is applicable will be denoted by $h(t,x)$.
 
 The corresponding curves at time $t$ will be denoted by $\Gamma^i(t)$ and the whole triple network by $\Gamma(t)$. Geometric quantities on $\Gamma(t)$ will be denoted with an index $\Gamma(t)$: for example, $\nu_{\Gamma(t)}$ will denote the unit normal on $\Gamma(t)$. By  abuse of notation we will use the same notation for the quantities as functions on $\Gamma(t)$ and in local coordinates as functions on $[0,1]^3$. 
\begin{cor}[\L ojasiewicz-Simon gradient inequality for $E$ - Geometric Version]\label{TheoremLSIgeometric}\ \\
Let $\Gamma_{\ast}$, $\gamma_{\ast}$, be as in Assumption~\ref{RemarkAssumstionsGammaast}. 
		Furthermore let $T>0$, $\alpha\in(0,1)$, $$u \in (C^{\frac{2+\alpha}{2},2+\alpha}([0,T]\times[0,1],\R^2))^3$$ be a solution of \eqref{EquationGeometricProblem} in the sense of Definition \ref{def-geomflow} with respect to some initial data $u_0\in (C^{2+\alpha}([0,1], \R^2))^3$ fulfilling \eqref{EquationCompatibilityConditions}, and let $\Gamma(t)$ be the corresponding evolving geometry. 
		
		Then, there exist  $\delta_{LSI}, C_{LSI}>0$, and $\theta \in (0, 1/2]$ such that for all $t\in[0,T]$ with 
		\begin{align} \label{EquationConditionUAnalyticLSI}
			\|u(t,\cdot)-\gamma_{\ast}\|_{C^{2+\alpha}}\le \delta_{LSI}
	    \end{align}
we have that
		\begin{align}\label{EquationLSIgeometric}
				|E(\Gamma(t))-E(\Gamma_{\ast})|^{1-\theta}\le C_{LSI} \left(\sum_{i=1}^3\int_0^1 ((D^2\varphi^{\circ}(\nu_{\Gamma(t)}^i)\tau_{\Gamma(t)}^i\cdot\tau_{\Gamma(t)}^i)\kappa_{\Gamma(t)}^i)^2ds_{\Gamma(t)}^i\right)^{1/2}.
		\end{align}
\end{cor}
\begin{proof}
	Due to Lemma \ref{LemmaExistenceRefFramGraphPara} condition \eqref{EquationConditionUAnalyticLSI} guarantees for $\delta_{LSI}$ small enough that we can find $h\in (C^{2+\alpha}([0,1]))^3$ parametrizing $u$ in the sense of \eqref{EquationHReparametrizingGamma}. Also, we have that $h\in V$ due to Lemma~\ref{LemmaRelationNormalTangentialPart} (whereby $V$ is as in \eqref{EquationLSISettingV}). Choosing $\delta_{LSI}$ possibly smaller we can use \eqref{EquationReparametrizationToNormalGraphSmallnesH} to guarantee that $\|h\|_V\le \sigma_{LSI}$ as in Theorem \ref{TheoremLSIanalytic}. Therefore we can apply Theorem \ref{TheoremLSIanalytic} to see that \eqref{EquationLSIanalytic} holds. Now it remains to rewrite  \eqref{EquationLSIanalytic} in the way we proposed.\\
	We first note hat $\gamma_h$ (i.e. $u$ reparametrized over the reference frame $\gamma_{\ast}$) fulfills the anisotropic angle condition $\eqref{EquationGeometricProblem}_4$ and therefore the terms at $x=1$ in \eqref{EquFirstVariationE} vanish (recall \eqref{EquationHelpEquationNormalGraphs} ). Consequently, $a^1$ and $a^2$ in the formula \eqref{EquationGradientOfE} for $\mathcal{M}$ are zero. Thus the right-hand side in \eqref{EquationLSIanalytic} only consists of the $L^2$-norms of the $u^i$. Studying the definition of the $u^i$ in \eqref{EquationGradientOfE} we observe that the first summand in each $u^i$ is already almost what we want. We only have an additional $(\nu^i_h\cdot\nu^i_{\ast})$-term. The other two summands have additional $(\nu^j_h\cdot\tau_{\ast}^j)$-terms and constant factors consisting of the $c^i$ and $s^i$. Also, the two additional factors are on the ``wrong'' curves. We will take care of these three problems one by one.

The first two problems are easy to handle as we have pointwise for $i, j=1,2,3$ that
	\begin{align*}
		|\nu^i_h\cdot\nu^i_{\ast}|+|\nu^i_h\cdot\tau_{\ast}^j|\le 2.
	\end{align*}
    Especially both factors  are controlled by a constant, which we can include into $C_{LSI}$.

	Finally, we observe that the metric tensors are uniformly bounded by each other, again assuming that $\sigma_{LSI}$ is small enough and therefore  $J^i_h(x)\in[\frac{1}{2},\frac{3}{2}]J^i_{\ast}(x)$ for all $x\in[0,1], i=1,2,3$. Thus we basically can shift the second and third summand in the $u^i$ to the right curve by changing the metric tensor.\\
	In total this shows that with $C_{LSI}, \sigma_{LSI}$ suitable adapted we can conclude \eqref{EquationLSIgeometric} from \eqref{EquationLSIanalytic}.
\end{proof}

Finally, we are able to prove the main result of this article.
\begin{teo}[Stability of the anisotropic curve shortening  flow on networks]\label{TheoremStabilityAnMCF}\ \\
	Let $\Gamma_{\ast}^i, \gamma_{\ast}^i, \tilde{\alpha}_{\ast}^i, i=1,2,3$ fulfill Assumption \ref{RemarkAssumstionsGammaast}. Then for every $\alpha\in(0,1)$ there exists $\delta(\alpha)>0$ such that for all geometrically admissible networks $\Gamma_0$ with a parametrization $u_0$ such that \begin{align}\label{EquationConvergenceCondition}
		\|u_0-\gamma_{\ast}\|_{C^{2+\alpha}}\le\delta(\alpha),
	\end{align}
	 the anisotropic curve shortening flow \eqref{EquationGeometricProblem} has a geometric solution on $[0,\infty)$ in the sense of Definition~\ref{def-geomflow}. 
	The flow may either collapse in finite time or evolve in infinite time to some stationary solution for $E$ with the same energy as $E(\Gamma_{\ast})$ (but with the stationary solution being possibly different from $\Gamma_{\ast}$). \\
	More precisely, the flow converges (written in terms of the normal graph formulation as in  \eqref{EquationNormalGraphs}) 
	with respect to the $C^{2+\alpha}$-norm and the limit is also a local minimum of $E$ in the sense of Assumption~\ref{RemarkAssumstionsGammaast} with the same anisotropic length energy as $\Gamma_{\ast}$.
\end{teo}
\begin{rem}[Convergence to $\Gamma_{\ast}$]\label{rem5.3}\ \\
	Note that if $\Gamma_{\ast}$ is an isolated local minimum, then our result yields that the flow 
	starting sufficiently close to $\Gamma_{\ast}$ converges to $\Gamma_{\ast}$. 
	
In our elementary geometry configuration  this is indeed the case when a local minimum  for our energy has a triple junction point that is  distinct  from the boundary points $P_{1}$, $P_{2}$, $P_{3}$. In  this case,  by \cite[Lemma~3.4]{Pratelli} (which exploits the smoothness, ellipticity and symmetry of $\varphi^{\circ}$) the energy $E$ admits a unique minimizer.
\end{rem}

\begin{proof}
 Given $u_{0}$ (or any reparametrization thereof fulfilling \eqref{EquationConvergenceCondition} for some $\delta(\alpha)$ small enough), construct $h_{0}$  as in Lemma~\ref{LemmaExistenceRefFramGraphPara}. By Remark~\ref{rem-43} and Lemma~\ref{LemmaSTEForH} we can let $h$ flow according to~\eqref{EquationSystemForH} on some time interval $[0, T_{STE}(\alpha)]$. 
	We will now do the usual procedure involved with the application of \L ojasiewicz-Simon gradient inequalities to stability analysis.  
	To this end set 
	\begin{align}\label{EquationDefSigma}
		\sigma:=\frac{1}{4}\min(\varepsilon_{STE}(\alpha),\delta_{LSI}, \varepsilon_{k}, 1),
	\end{align}
 where $\varepsilon_{STE}(\alpha)$ is from Lemma \ref{LemmaSTEForH},  $\delta_{LSI}$ from Corollary \ref{TheoremLSIgeometric}, and $\varepsilon_{k}$ from Lemma~\ref{LemmaParabolicSmoothigH} resp. Remark~\ref{RemarkOnAssumptionBoundH}.
	(Further smallness conditions on $\sigma$ will be added throughout the argumentation as needed along the proof.) In particular, this choice of $\sigma$ will give us in the following good properties associated with the constants in the right-hand side of \eqref{EquationDefSigma}.
	
	 By Lemma~\ref{LemmaExistenceRefFramGraphPara} we know that we can find   $\delta(\alpha)>0$ such that for all initial $u_0$ with $$\|u_0-\gamma_{\ast}\|_{(C^{2+\alpha}([0,1];\R^2))^3}\le\delta(\alpha)$$ 
	we have 
	$$h_0\in (C^{2+\alpha}([0,1]))^3\text{ with }\|h_{0}\|_{(C^{2+\alpha}([0,1]))^3}\le \sigma.$$
	 As already mentioned, by Lemma \ref{LemmaSTEForH} the map $h$ (with initial data $h_{0}$) solves \eqref{EquationSystemForH}, and the corresponding network parametrized by $u$ ( with  $u^{i}(t,x)= \gamma_{\ast}^{i}(x)+ h^{i}(t,x) \nu^{i}_{\ast} + \mu^{i}(t,x) \tau^{i}_{\ast} $ )
	 is by construction a solution of \eqref{EquationGeometricProblem}, at least on some time interval $[0,T_{STE}(\alpha)]$.  	 We have that $\Gamma(t)=\Gamma_{h(t,\cdot)}$.
	 Now, define $$T_{\max} \in (0, \infty]$$ to be  the largest time such that the solution  $h(t,x)$ exists\footnote{Hereby, we mean that $h$ solves \eqref{EquationSystemForH} classically and that 
	 $h\in C^{\frac{2+\alpha}{2},2+\alpha}([0,T]\times[0,1]))^3$ for any 
	 $0<T < T_{\max}$.
	 } and such that
\begin{align}\label{EquationContradictionAsumption}
		 \|h(t,\cdot)\|_{(C^{2+\alpha}([0,1]))^3}< 2\sigma \text{ for all $t \in [0,T_{\max})$}.
	\end{align}	 
	
	The constant $\delta(\alpha)$ and hence  $\sigma$ is here chosen so small that the network $\Gamma(t)=\Gamma_{h(t,\cdot)}$ falls in the $H^{2}$-ball  around $\Gamma_{\ast}$ characterizing its local minimal property: in other words $E(\Gamma_{h(t,\cdot)})\geq E(\Gamma_{\ast})$ as long as the flow exists.
	Note that $T_{\max}$ is indeed bigger than zero as we can guarantee for a potentially smaller $\sigma$ the continuous dependency on the initial data due to Lemma \ref{LemmaContinuousDependency}. Then, for small enough initial data (i.e. $\sigma$ small enough) Lemma \ref{LemmaContinuousDependency} guarantees us that 
		\begin{align*}
			\|h\|_{(C^{\frac{2+\alpha}{2}, 2+\alpha}([0, T_{STE}(\alpha)]\times [0,1]))^3}< 2\sigma.
		\end{align*} 
	In particular, for these initial data we have that 
	\begin{align*}
		\forall t\in[0, T_{STE}(\alpha)]: \|h(t,\cdot)\|_{(C^{2+\alpha}([0,1]))^3}\le \|h\|_{(C^{\frac{2+\alpha}{2}, 2+\alpha}([0, T_{STE}(\alpha)]\times [0,1]))^3}< 2\sigma.
	\end{align*}
    Thus for these initial data we know that $T_{\max}\ge T_{STE}(\alpha)$.

 Now assume that $T_{\max}$ is finite. We directly see that due to the choice of $\sigma$ we can apply Lemma~\ref{LemmaSTEForH} in $t=T_{\max}-\frac{T_{STE}(\alpha)}{2}$ using $\|h(t)\|_{C^{2+\alpha}}\le 2\sigma\le \varepsilon_{STE}(\alpha)$  (recall \eqref{EquationContradictionAsumption}) to see that the solution indeed exists at least on $[0, T_{\max}+\frac{T_{STE}(\alpha)}{2}]$. So the only possibility for $T_{\max}$ being finite is indeed that \eqref{EquationContradictionAsumption} is not fulfilled anymore. Furthermore, at every $\tilde{t}\in [0, T_{\max})$ we can apply Lemma \ref{LemmaParabolicSmoothigH} with $T=T_{STE}(\alpha)$, $R(\alpha)=R_{STE}(\alpha)$\footnote{Note that due to Remark \ref{RemarkOnAssumptionBoundH} the smallness condition for $R(\alpha)$ is guaranteed because of the choice of $\sigma$.} to conclude for arbitrary but fixed $0<t_k<T_k<T_{STE}(\alpha)$ and $t\in[\tilde{t}+t_k, \tilde{t}+T_k]$ that
 	\begin{align*}
 		\|h(t,\cdot)\|_{(C^{3+\alpha}([0,1]))^3}\le C(t_k, T_k, R_{STE}(\alpha), \sigma).
 	\end{align*}

As this estimate is independent of $\tilde{t}$ we see that 
for fixed $t_S\in(0, T_{\max})$  we have that
	\begin{align}\label{EquationIntegralEstimateH}
		\forall t\in[t_S, T_{\max}): \|h(t,\cdot)\|_{(C^{3+\alpha}([0,1]))^3}\le C(\sigma,t_{S}, R_{STE}(\alpha)).
	\end{align}
Still holding onto the assumption that $T_{\max} <\infty$, we distinguish now between two cases:
if $E(\Gamma(\bar{t}))=E(\Gamma_{\ast})$ for some $\bar{t} \in [0,T_{\max}]$ then $\Gamma(\bar{t})$ is also a local minimum. Therefore, the flow becomes stationary from $\bar{t}$ on, and $\Gamma(\bar{t})$ provides a stationary solution which has the same energy as $\Gamma_{\ast}$ (but is possibly different from $\Gamma_{\ast}$).

	Now let us assume that $E(\Gamma(t))>E(\Gamma_{\ast})$ for all $t \in [0,T_{\max}]$.
     Then, using that $\sigma <1$, we have for all $t\in(0,T_{\max})$ that\footnote{For better readability we drop the time variable in $h$ in the following!}
	\begin{align}
	-\frac{d}{dt}\abVal{E(\Gamma(t))&-E(\Gamma_{\ast})}^{\theta}=-\theta\abVal{E(\Gamma(t))-E(\Gamma_{\ast})}^{\theta-1}\frac{d}{dt}E(\Gamma(t))\notag\\
	&=\theta\abVal{E(\Gamma(t))-E(\Gamma_{\ast})}^{\theta-1}\sum_{i=1}^3\int_0^1(D^2\varphi^{\circ}(\nu^i_{\Gamma(t)})\tau_{\Gamma(t)}^i\cdot\tau_{\Gamma(t)}^i)\kappa_{\Gamma(t)}^i\nu^i_{\Gamma(t)}\cdot \partial_t u^{i}(t)ds_{\Gamma(t)}^i\notag\\
	&=\theta\abVal{E(\Gamma(t))-E(\Gamma_{\ast})}^{\theta-1}\sum_{i=1}^3\int_0^1\left((D^2\varphi^{\circ}(\nu^i_{\Gamma(t)})\tau_{\Gamma(t)}^i\cdot\tau_{\Gamma(t)}^i)\kappa_{\Gamma(t)}^i\nu^i_{\Gamma(t)}\right)^2ds_{\Gamma(t)}^i\notag\\
	&\ge C_{LSI}^{-1}\left(\sum_{i=1}^3\int_0^1\left((D^2\varphi^{\circ}(\nu^i_{\Gamma(t)})\tau_{\Gamma(t)}^i\cdot\tau_{\Gamma(t)}^i)\kappa_{\Gamma(t)}^i\nu^i_{\Gamma(t)}\right)^2ds_{\Gamma(t)}^i\right)^{1/2}\label{EquationStabAnaMainCalcu}\\
	&=C_{LSI}^{-1}\norm{V_{\Gamma(t)}}_{L^2(\Gamma(t))}\ge C(C_{LSI},\Gamma_{\ast})\norm{\partial_t h(t,\cdot)(\nu_{\ast}\cdot\nu_{h})+\partial_t\mu(t,\cdot)(\tau_{\ast}\cdot\nu_{h})}_{(L^2([0,1]))^3}\notag\\
	&\ge C(C_{LSI}, \Gamma_{\ast})|\|\partial_th(t,\cdot)(\nu_{\ast}\cdot\nu_{h})\|_{(L^2([0,1]))^3}-\|\partial_t\mu(t,\cdot)(\tau_{\ast}\cdot\nu_{h})\|_{(L^2([0,1]))^3}|\notag\\
	&\ge C(C_{LSI}, \Gamma_{\ast})\norm{\partial_t h(t,\cdot)}_{(L^2([0,1]))^3}.\notag
\end{align}
Hereby, we used in the second line \eqref{EquationVariationAnisotropicLength} and $\eqref{EquationGeometricProblem}_4$, in the third line $\eqref{EquationGeometricProblem}_1$, in the fourth line the LSI  \eqref{EquationLSIgeometric}, then \eqref{EquationNormalVelocityInH} together with the fact that the length of $\Gamma(t)$ is controlled above and below by the length of $\Gamma_{\ast}$ thanks to \eqref{EquationContradictionAsumption}  with possibly an even smaller $\sigma$, and in the fifth line the inverse triangle inequality. For the last inequality we used that due to \eqref{EqRelationNormalTangentialPart} we have that
\begin{align*}
	\|\partial_t\mu(t,\cdot)\|_{(L^2([0,1]))^3}\le C(\Gamma_{\ast})\|\partial_t h\|_{(L^2([0,1]))^3},
\end{align*}
where the constant depends on the contact angles (and thus the entries in the matrix $\mathcal{I}$).  
Now assuming $\sigma$ small enough such that 
\begin{align}\label{condnutau}
\tau_{\ast}\cdot\nu_h\le \frac{1}{4C(\Gamma_{\ast})} \qquad \text{ and } \qquad \nu_{\ast}\cdot\nu_h\ge \frac{1}{2}
\end{align} pointwise for all $t \in [0,T_{\max})$ (this is possible again thanks to  $h$ fulfilling $\eqref{EquationContradictionAsumption}$) we have that
\begin{align*}
	\|\partial_t\mu(t,\cdot)(\tau_{\ast}\cdot\nu_h)\|_{(L^2([0,1]))^3}\le \frac{1}{2}\|\partial_t h(t,\cdot)(\nu_{\ast}\cdot\nu_h)\|_{(L^2([0,1]))^3}.
\end{align*}
 Due to this the last inequality in \eqref{EquationStabAnaMainCalcu} holds. Summing these calculations up we conclude for a suitable $C=C(C_{LSI}, \Gamma_{\ast})>0$ that
\begin{align}\label{EquationEstimatepartialtu}
	\norm{\partial_th(t,\cdot)}_{(L^2([0,1]))^3}\le -C \frac{d}{dt}\abVal{E(\Gamma(t))-E(\Gamma_{\ast})}^{\theta}.
\end{align}
Integrating \eqref{EquationEstimatepartialtu} in time yields for all $t\in(0,T_{\max})$ that
\begin{align}
	\|h(t,\cdot)\|_{(L^2([0,1]))^3}&\le \int_0^t \|\partial_t h(t,\cdot)\|_{(L^2([0,1]))^3}dt+\|h_0\|_{(L^2([0,1]))^3} \notag\\
	&\le -C|E(\Gamma(t))-E(\Gamma_{\ast})|^{\theta}+C|E(\Gamma_0)-E(\Gamma_{\ast})|^{\theta}+\|h_0\|_{(L^2([0,1]))^3}\notag\\
	&\le C|E(\Gamma_0)-E(\Gamma_{\ast})|^{\theta}+\|h_0\|_{(L^2([0,1]))^3}\label{EquationEstimateHeightFunction}\\
	&\le C\|h_0\|_{(C^1([0,1]))^3}^{\theta}+\|h_0\|_{(C^1([0,1]))^3}\notag\\
	&\le C\|h_0\|_{(C^1([0,1]))^3}^{\theta}+\|h_0\|_{(C^1([0,1]))^3}^{\theta}\notag\\
	&\le C\|h_0\|_{(C^1([0,1]))^3}^{\theta},\notag
\end{align}
with $C=C(C_{LSI}, \Gamma_{\ast}, \varphi^{\circ})>0$ .
Hereby, we used in the fifth step the general smallness assumption on $\sigma$ and in the fourth step the Lipschitz continuity of the anisotropic surface energy due to the smoothness of $\varphi^{\circ}$ away from the origin, \eqref{condnutau}, the formula \eqref{EquationFormulaNormalFunctionInH} for the normal $\nu_h$, and the general fact that
$$|(\gamma_{h_{1}})_{x}| |(\gamma_{h_{2}})_{x}||\tau_{h_{1}} -\tau_{h_{2}}|^{2}+ (|(\gamma_{h_{1}})_{x}| -|(\gamma_{h_{2}})_{x}|)^{2} = |(\gamma_{h_{1}})_{x} -(\gamma_{h_{2}})_{x}|^{2}. $$
Applying real interpolation theory (see \cite[(i) of Theorem at page 29 and (ii) of Theorem at page 5]{TriebelTheoryOfFunctionSpaces2}, \cite[Theorem 6.4.5(3)]{BerghLoesfstroemInterpolationSpaces} and \cite[Theorem 4.6.1(e), page 328]{TriebelInterpolationTheoryFunctionSapcesDifferentialOperators}  ) we get with a suitable interpolation exponent $\beta\in(0,1)$ and using \eqref{EquationIntegralEstimateH} and \eqref{EquationEstimateHeightFunction} for any $t\in[t_S, T_{\max})$ that
\begin{align*}
	\|h(t,\cdot)\|_{(C^{2+\alpha}([0,1]))^3}\le\|h(t,\cdot)\|_{(C^{3}([0,1]))^3} &\le C\|h(t,\cdot)\|_{(L^2([0,1]))^3}^{\beta}\|h(t,\cdot)\|_{(C^{3+\alpha}([0,1]))^3}^{1-\beta}\\
	&\le C(\sigma, t_{S}, R_{STE}(\alpha), \beta, \Gamma_{\ast}, C_{LSI}, \varphi^{\circ})\|h_0\|_{(C^1([0,1]))^3}^{\beta\theta}\\&\le C(\sigma, t_{S}, R_{STE}(\alpha), \beta, \Gamma_{\ast}, C_{LSI},\varphi^{\circ})\|h_0\|_{(C^{2+\alpha}([0,1]))^3}^{\beta\theta}.
\end{align*}
 Now, by setting $C_{1}:=C(\sigma, t_{S}, R_{STE}(\alpha), \beta, \Gamma_{\ast}, C_{LSI},\varphi^{\circ})$ and choosing 
\begin{align*}
	\|h_0\|_{(C^{2+\alpha}([0,1]))^3}\le e^{\frac{1}{\beta\theta}\ln(\frac{\sigma}{C_{1}})},
\end{align*}
which is - again using Lemma \ref{LemmaExistenceRefFramGraphPara} - guaranteed for $u_0$ close enough to $\gamma_{\ast}$, it follows for all $t\in[t_S,T_{\max})$ that
\begin{align}\label{EquationContradictionFinal}
\|h(t,\cdot)\|_{C^{2+\alpha}([0,1])}\le \sigma.
\end{align}
Note that due to Step 3 in the proof of Lemma \ref{LemmaParabolicSmoothigH} we see that 
$[t_S, T_{\max}+T_k]\to \R, t\mapsto \|h(t,\cdot)\|_{C^3}$ is continuous. Using the interpolation property of H\"older spaces (see \cite[Example 1.8]{LunardiInterpolationTheory}) we see that 
\begin{align*}
	\|h(t_1,\cdot)-h(t_2,\cdot)\|_{C^{2+\alpha}}\le C \|h(t_1,\cdot)-h(t_2,\cdot)\|^{1-\alpha}_{C^{2}}\|h(t_1,\cdot)-h(t_2,\cdot)\|^{\alpha}_{C^{3}}
\end{align*}
 In particular $\|h(t,\cdot)\|_{C^{2+\alpha}}$ is also continuous in time as the $C^{2}$- and $C^3$-norm both are continuous in time. With this and \eqref{EquationContradictionFinal} we conclude that \eqref{EquationContradictionAsumption} is also fulfilled beyond $T_{\max}$, yielding a contradiction to the maximality of $T_{\max}$. If follows that $T_{\max}=\infty$.\\
If $T_{\max}=\infty$ in our assumption \eqref{EquationContradictionAsumption} then, if a local minimum is not attained in finite time, by proceeding as above we can still choose $\sigma$ so small that \eqref{EquationEstimateHeightFunction} holds.\\
Now we note that due to \eqref{EquationEstimatepartialtu} and the fact that $E$ is decreasing in time, we conclude that $\partial_t h\in L^1([0,\infty),(L^2([0,1]))^3)$. 
Therefore, there exists a unique $h_{\infty}\in (L^2([0,1]))^3$ - inducing a network $\Gamma_{\infty}$ - with 
\begin{align*}
	\underset{t\to\infty}{\lim}h(t,\cdot)=h_{\infty}\text{ in }(L^2([0,1]))^3.
\end{align*}
In particular, $h(t,\cdot)$ is a Cauchy sequence in $(L^2([0,1]))^3$.
Furthermore by \eqref{EquationEstimateHeightFunction} we have that $\|h_{\infty}\|_{(L^2([0,1]))^3}\le C \| h_{0}\|_{(C^{1}([0,1]))^{3}}^{\theta}$.

Observe that \eqref{EquationIntegralEstimateH} is a local estimate that holds for all $t>t_S$ and the constant $C(\sigma, t_S, R_{STE}(\alpha))$ does not depend on the final existence time. Thus for any $t,\tilde{t}>t_S$ we conclude with the same interpolation argument as before that
	\begin{align*}
		\|h(t,\cdot)-h(\tilde{t},\cdot)\|_{(C^{2+\alpha}([0,1]))^3}\le C(\sigma, t_{S}, R_{STE}(\alpha), \beta, \Gamma_{\ast}, C_{LSI},\varphi^{\circ})\|h(t,\cdot)-h(\tilde{t},\cdot)\|_{(L^2([0,1]))^3}^{\beta}.
	\end{align*}
This implies that $h(t,\cdot)$ is also a Cauchy sequence in $(C^{2+\alpha}([0,1]))^3$ and thus has a limit in this space. By uniqueness of limits we infer that $h_{\infty}$ is this limit and has the higher regularity. As there is a sequence $(t_n)_{n\in\N}\subset \R^+$ with $t_n\to\infty$ and $\|\partial_t h(t_n,\cdot)\|_{(L^2([0,1]))^3}\to 0$, this implies by \eqref{EquationParabolicVersionOfPDE} that 
\begin{align*}
	\|F_{h_{\infty}}^{-1}(\varphi^{\circ}(\nu_{h_{\infty}})(D^2\varphi^{\circ}(\nu_{h_{\infty}})\tau_{h_{\infty}}\cdot\tau_{h_{\infty}})\kappa_{h_{\infty}})\|_{(L^2([0,1]))^3}=0.
\end{align*}
Due to Lemma \ref{lemma2.1} this implies that $\kappa_{h_{\infty}}\equiv 0$. Thus $\Gamma_{\infty}$ consists of three straight lines. Finally, as $h_{\infty}$ also fulfills $\|h_{\infty}\|_{(C^{2+\alpha}([0,1]))^3}\le 2\sigma$ due to the convergence in the $C^{2+\alpha}$-norm, we can apply Corollary \ref{TheoremLSIgeometric} for $h_{\infty}$ and conclude that
\begin{align}
	|E(\Gamma_{\infty})-E(\Gamma_{\ast})|^{1-\theta}=0,
\end{align}
which directly implies that $\Gamma_{\infty}$ is a local energy minimum with $E(\Gamma_{\infty})=E(\Gamma_{\ast})$. This finishes the proof. 
\end{proof}

\section*{Acknowledgements}

MG und PP have been supported by the DFG (German Research Foundation) Projektnummer: 404870139.

\appendix
\section{Parabolic H\"older spaces}\label{AppendixHoelder}

In this section we want to introduce parabolic H\"older spaces and give some basic properties of them. They can be found, e.g., in \cite[p. 66 and 91]{SolonnikovExistenceTheory}. 

We will only consider $\Omega=[0,1]$ as space and $I=[0,T]$ for some $T>0$ as time domain. Most of the following is also true for more general situations. We first introduce the H\"older semi-norms in space and time given by
\begin{align*}
	\langle f\rangle_{x,\alpha}&:=\underset{x_1, x_1\in \Omega, t\in I}{\sup}\frac{|f(t,x_1)-f(t,x_2)|}{|x_1-x_2|^{\alpha}},\\
	\langle f\rangle_{t,\alpha}&:=\underset{x\in \Omega, t_1, t_2\in I}{\sup}\frac{|f(t_1,x)-f(t_2,x)|}{|t_1-t_2|^{\alpha}},
\end{align*} 
for a function $f: I\times \Omega\to\R^n$ and $\alpha\in(0,1)$.
Now, we define for  $\alpha\in (0,1), n\in\N$ the spaces
\begin{align*}
	C^{\frac{\alpha}{2}, \alpha}(I\times\Omega, \R^n)&:=\{f\in C(I\times\Omega, \R^n)| \langle f\rangle_{t,\frac{\alpha}{2}}<\infty, \langle f\rangle_{x,\alpha}<\infty \},\\
	\|f\|_{C^{\frac{\alpha}{2}, \alpha}(I\times\Omega, \R^n)}&=\|f\|_{\infty}+\langle f\rangle_{t,\frac{\alpha}{2}}+\langle f\rangle_{x,\alpha},\\
	C^{\frac{1+\alpha}{2},1+\alpha}(I\times\Omega, \R^n)&:=\{f\in C(I\times\Omega, \R^n)|\partial_xf\in C(I\times\Omega, \R^n), \langle \partial_xf\rangle_{x,\alpha}<\infty, \langle f\rangle_{t,\frac{1+\alpha}{2}}<\infty, \\
	 &\phantom{:=\{ }\langle\partial_x f\rangle_{t,\frac{\alpha}{2}}<\infty \},\\
	\|f\|_{C^{\frac{1+\alpha}{2}, 1+\alpha}(I\times\Omega, \R^n)}&=\|f\|_{\infty}+\|\partial_xf\|_{\infty}+\langle \partial_x f\rangle_{x,\alpha}+\langle f\rangle_{t,\frac{1+\alpha}{2}}+\langle\partial_x f\rangle_{t,\frac{\alpha}{2}} ,\\
	C^{\frac{2+\alpha}{2}, 2+\alpha}(I\times\Omega, \R^n)&:=\{f\in C(I\times\Omega, \R^n)| f, \partial_x f, \partial_{xx}f, \partial_tf\in C(I\times\Omega, \R^n), \langle \partial_t f\rangle_{x,\alpha}<\infty, \\
	&\phantom{:=\{ } \langle\partial_{xx}f\rangle_{x,\alpha}<\infty, \langle \partial_xf\rangle_{t,\frac{1+\alpha}{2}}<\infty, \langle\partial_{xx}f\rangle_{t,\frac{\alpha}{2}}<\infty, \langle\partial_t f\rangle_{t,\frac{\alpha}{2}}<\infty  \},\\
	\|f\|_{C^{\frac{2+\alpha}{2},2+\alpha}(I\times\Omega)}&:=\sum_{0\le 2i+j\le 2}\|\partial_t^i\partial^{j}_xf\|_{\infty}+\langle \partial_{xx}f\rangle_{x,\alpha}+\langle \partial_t f\rangle_{x,\alpha}+\langle\partial_x f \rangle_{t,\frac{1+\alpha}{2}}\\
	&+\langle \partial_{xx}f\rangle_{t,\frac{\alpha}{2}}+\langle \partial_t f\rangle_{t,\frac{\alpha}{2}}.
\end{align*}
\begin{defi}[Parabolic H\"older spaces]\ \\
	We call ($C^{\frac{k+\alpha}{2}, k+\alpha}(I\times\Omega)$, $\|\cdot\|_{C^{\frac{k+\alpha}{2}, k+\alpha}(I\times\Omega)})$ with $k\in \{0,1,2\}, \alpha\in(0,1)$ parabolic H\"older spaces.
\end{defi}
Note that parabolic H\"older spaces can be defined more generally. Then, the denominator depends on the order of the considered PDE. 
The following two properties are crucial for our work.
\begin{lemma}[Product estimates in parabolic H\"older spaces]\label{LemmaProductEstimatesParabolichHspace}\ \\
	Let $k\in\{0,1,2\},\alpha\in(0,1)$ and $f,g\in C^{\frac{k+\alpha}{2}, k+\alpha}(I\times\Omega)$. Then we have 
	\begin{align}\label{EquationProductareClosedinHolderSpaces}
		fg\in C^{\frac{k+\alpha}{2}, k+\alpha}(I\times\Omega),
	\end{align}
	and furthermore we have that
	\begin{align}
		\|fg\|_{C^{\frac{k+\alpha}{2}, k+\alpha}(I\times\Omega)}&\le C \|f\|_{C^{\frac{k+\alpha}{2}, k+\alpha}(I\times\Omega)}\|g\|_{C^{\frac{k+\alpha}{2}, k+\alpha}(I\times\Omega)}.\label{EquationProductEstimatesHolderSpaces1}
	\end{align}
\end{lemma}
\begin{proof}
	Cf. \cite[Lemma 2.16]{goesswein2019Dissertation}.
\end{proof}
\begin{lemma}[Contractivity property of lower order terms in parabolic H\"older spaces]\label{LemmaContractivityLowerOrderTerms}\ \\
	Let $k,k'\in\{0,1,2\},k'<k,\alpha\in(0,1)$. Then, we have for any $f\in C^{\frac{k+\alpha}{2}, k+\alpha}(I\times\Omega)$ that
	\begin{align}
		\|f\|_{C^{\frac{k'+\alpha}{2}, k'+\alpha}(I\times\Omega)}\le \|f\big|_{t=0}\|_{C^{k'+\alpha}(\Omega)}+CT^{\bar{\alpha}}\|f\|_{C^{\frac{k+\alpha}{2}, k+\alpha}(I\times\Omega)}.
	\end{align}
	Hereby, the constants $C$ and $\bar{\alpha}$ depend on $\alpha,k,k'$ and $\Omega$. Especially, if $f\big|_{t=0}\equiv 0$, we have
	\begin{align}\label{EquationContractivityofLowerOrderTerms}
		\|f\|_{C^{\frac{k'+\alpha}{2},k'+\alpha}(I\times\Omega)}\le CT^{\bar{\alpha}}\|f\|_{\frac{k+\alpha}{2}, C^{k+\alpha}(I\times\Omega)}.
	\end{align}
\end{lemma}
\begin{proof}
	Cf. \cite[Lemma 2.17]{goesswein2019Dissertation}.
\end{proof}

Further useful properties can be found for instance in \cite{goesswein2019Dissertation} and \cite{KroenerNovagaPozziAnisoCurvFloImmeredNetworks}.

As a final remark we want to note that to avoid confusion we denote the classical H\"older spaces by $C^{k+\alpha}$ instead of $C^{k,\alpha}$, which is the more common notation. So for example we have that
\begin{align*}
	C^{2+\alpha}(\Omega)=C^{2,\alpha}(\Omega):=\{f\in C^2(\Omega)\big| \sum_{i=0}^2\|\partial_x^if\|_{\infty}+[\partial_x^2f]_{x,\alpha}<\infty\}.
\end{align*}

\section{Additional proofs}\label{AppendixAdditionalProofs}

\begin{lemma}\label{LemmaEquivalenceForceBalanceYoungModulus}\ \\
	Suppose that $\theta^1, \theta^2, \theta^3\in(0,\pi)$ are such that $\theta^1+\theta^2+\theta^3=2\pi$. Furthermore, let $\nu^1,\nu^2,\nu^3\in\R^2$ be such that $\angle(\nu^1,\nu^2)=\theta^3, \angle(\nu^2,\nu^3)=\theta^1,\angle(\nu^3,\nu^1)=\theta^2$ and $|\nu^i|=1, i=1,2,3$. Finally, let $\tilde{\alpha}^1,\tilde{\alpha}^2,\tilde{\alpha}^3\in\R_+$. Then, \eqref{EquationYoungsModulus} and \eqref{EquationForceBalanceTripleJunction} are equivalent. 
\end{lemma}
\begin{proof}
	W.l.o.g. assume that $\theta^1$ is the smallest angle. Then due to $\theta^1+\theta^2+\theta^3=2\pi$ we have that $\theta^2,\theta^3\in(\pi/2,\pi)$. We can illustrate the geometrical situation in the following image.
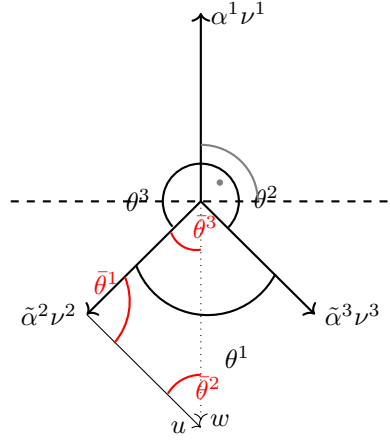
\begin{align*}
	\begin{tikzpicture}[scale=0.5]
	\draw[dashed, thick] (-5, 0) -- (5, 0);
	\draw[->, thick] (0, 0) -- (0, 5) node [right] {$\alpha^1 \nu^1$};
	\draw[->, dotted] (0, 0) -- (0, -5.8) node [right] {$w$};
	\draw[->, thick] (0, 0) -- (-3, -3) node [left] {$\tilde{\alpha}^2 \nu^2$};
	\draw[->, thick] (0, 0) -- (3, -3) node [right] {$\tilde{\alpha}^3 \nu^3$};
	\draw[->] (-3, -3) -- (0, -6) node [left=2pt] {$u$};
	\draw[thick] (0, 1) arc [start angle=90, end angle=-43.5, radius=1cm] node [xshift=5mm, yshift=4mm] {$\theta^2$};
	\draw[thick, gray] (0, 1.5) arc [start angle=90, end angle=0, radius=1.5cm]; 
	\filldraw[fill=gray, draw=gray] (0.5, 0.5) circle (2pt);
	\draw[thick] (0, 1) arc [start angle=90, end angle=222, radius=1cm] node [xshift=-4.5mm, yshift=3.5mm] {$\theta^3$};
	\draw[thick, color=red] (0, -1.28) arc [start angle=280, end angle=200, radius=0.7cm] node [xshift=4.5mm, yshift=0.8mm] {$\bar{\theta}^3$};
	\draw[thick, color=red] (0, -4.6) arc [start angle=90, end angle=150, radius=1cm] node [xshift=5.3mm, yshift=0.8mm] {$\bar{\theta}^2$};
	\draw[thick, color=red] (-2, -2) arc [start angle=22.5, end angle=-40.5, radius=1.7cm] node [xshift=-1mm, yshift=8mm] {$\bar{\theta}^1$};
	\draw[thick] (-1.7, -1.7) arc [start angle=200, end angle=333, radius=2cm] node [xshift=-5mm, yshift=-11mm] {$\theta^1$};
\end{tikzpicture}
\end{align*}
Hereby, $u\in\R^2$ is parallel to $\nu^3$ and such that $w:=\tilde{\alpha}^2\nu^2+u\in\langle\nu^1\rangle$. Observe that $-\nu^1$ will always be in the cone between $\nu^2$ and $\nu^3$ since by assumption $\theta_{i} <\pi$ for $i=1,2,3$.
Thus, $u$ is always a positive multiple of $\nu^3$. 
By construction we have that $\bar{\theta}_{1}=\pi -\theta_{1}$ and $\bar{\theta}_{3}=\pi -\theta_{3}$, so that using that $\sum_{i} \bar{\theta}_{i}=\pi$ and $\sum_{i} \theta_{i}=2\pi$, we obtain $\theta_{2}= \pi -\bar{\theta}_{2}$. It follows 
that $\sin(\theta^i)=\sin(\bar{\theta}^i)$ due to the symmetry of the sine function. 
First assume that $\eqref{EquationForceBalanceTripleJunction}$ holds. Then we have that $w=-\tilde{\alpha}^1\nu^1$ and $u=\tilde{\alpha}^3\nu^3$. Then, \eqref{EquationYoungsModulus} is a consequence of the law of sines. On the other hand if \eqref{EquationYoungsModulus} holds, we can use the law of sines to deduce that $|u|=\tilde{\alpha}^3$ and $|w|=\tilde{\alpha}^1$. Therefore we have that $u=\tilde{\alpha}^3\nu^3$ and $w=-\tilde{\alpha}^1\nu^1$. This implies then \eqref{EquationForceBalanceTripleJunction}.
\end{proof}

\section{Used results}\label{AppendixUsedResults}
\begin{teo}[Implicit function theorem on Banach spaces]\label{TheoremImplicitFunction} \ \\
	Let $X,Y,Z$ be real Banach spaces, $(x_0,y_0)\in X\times Y, \Lambda\times\Omega$ be an open neighborhood of $(x_0,y_0)$ in $X\times Y$ and $F: \Lambda\times\Omega\to Z$ such that
	\begin{itemize}
		\item[1.] $F(x_0,y_0)=0$;
		\item[2.] $\partial_y F$ exists as partial Fr\'echet-derivative on $\Lambda\times\Omega$ and $\partial_y F(x_0,y_0): Y\to Z$ is bijective;
		\item[3.] $F$ and $\partial_y F$ are continuous at $(x_0,y_0)$.
	\end{itemize}
Then there exist positive numbers $r_0$ and $r$ such that $B_{r_0}(x_0)\times B_r(y_0)\subset\Lambda\times\Omega$ and for every $x\in X$ satisfying $\|x-x_0\|_X\le r_0$, there is exactly one $y=y(x)\in Y$ for which $\|y-y_0\|_{Y}\le r$ and $F(x,y)=0$. Moreover,
\begin{itemize}
	\item if $F$ is continuous in a neighborhood of $(x_0,y_0)$, then $y(\cdot)$ is continuous in a neighborhood of $x_0$.
	\item if $F$ is a $C^{m}$-map for $1\le m\le\infty$ on a neighborhood of $(x_0,y_0)$, then $y$ is a $C^{m}$-map on a neighborhood of $x_0$.\\
	\item if $F$ is analytic at $(x_0,y_0)$, then $y$ is analytic at $x_0$.
\end{itemize}
\end{teo}
\begin{proof}
	See \cite[Theorem 4.B]{ZeidlerBook1} and \cite[Corollary 4.23]{ZeidlerBook1}.
\end{proof}
In the following all Banach spaces are assumed to be real. Following \cite{FeehannMaridakisLojasiewiczSimon}, let $V$ be a Banach space and $V'$ denote its continuous dual space. We call a bilinear form $b: V \times V \to \R$ definite if $b(x,x) \neq 0$ for all $x \in V \setminus \{0\} $.
We say that a continuous embedding of a Banach space into its continuous dual space, $j : V \to V'$, is definite if the pullback of the canonical pairing, $V \times V \ni (x, y) \mapsto \langle x, j(y) \rangle_{V \times V'} \to \R $, is a definite bilinear form.

	Before we state the next result, we have to define gradient maps (see \cite[Definition 1.5]{FeehannMaridakisLojasiewiczSimon}):
	\begin{defi}[Gradient map]\label{DefinitionGradientMap}\ \\
		Let $U\subset V$ be an open subset of a Banach space $V$ and $W$ a Banach space with continuous embedding $W\subset V'$. A continuous map $\mathcal{M}: U\to W$ is called a gradient map if there exists a $C^1$-function $E: U\to\R$, such that
		\begin{align*}
			E'(x)v=\langle v, \mathcal{M}(x)\rangle_{V\times V'},\quad \forall x\in U, v\in V,
		\end{align*}
		where $\langle\cdot,\cdot\rangle_{V\times V'}$ is the canonical bilinear form on $V\times V'$. The real valued function $E$ is called a potential for the gradient map $\mathcal{M}$.
	\end{defi}
	\begin{teo}[Theorem 2 in \cite{FeehannMaridakisLojasiewiczSimon}]\label{TheoremFeehanMaridakisTheo2}\ \\
		Let $V$ and $W$ be Banach spaces with continuous embedding, $V\subset W\subset V'$, and such that the embedding $V\subset V'$ is definite. Let $U\subset V$ be an open subset, $E: U\to\R$ a $C^2$-function with real analytic gradient map, $\mathcal{M}: U\to W$ and $x_{\ast}\in U$ a critical point of $E$, that is, $\mathcal{M}(x_{\ast})=0$. If $\mathcal{M}'(x_{\ast}): V\to W$ is a Fredholm operator with index zero, then there are constants, $C_{LSI}>0$, $\sigma_{LSI}\in(0,1]$ and $\theta\in(0,\frac{1}{2}]$ with the following significance. If $x\in U$ obeys
		\begin{align*}
			\|x-x_{\ast}\|_V<\sigma_{LSI},
		\end{align*}
		then
		\begin{align*}
		|E(x)-E(x_{\ast})|^{1-\theta}\le C_{LSI}\|\mathcal{M}(x)\|_{W}.
		\end{align*}
	\end{teo}
	\begin{proof}
		See proof of \cite[Theorem 2]{FeehannMaridakisLojasiewiczSimon}.
	\end{proof}


\end{document}